\documentclass[a4paper]{amsart}

\usepackage[T1]{fontenc} 
\usepackage[utf8]{inputenc} 

\usepackage{lmodern} 

\usepackage{microtype}
\usepackage{amssymb} 

\usepackage{mathtools}
\mathtoolsset{centercolon} 

\makeatletter
\@namedef{subjclassname@2020}{\textup{2020} Mathematics Subject Classification}
\makeatother

\newtheorem{theorem}{Theorem}[section]
\newtheorem{lemma}[theorem]{Lemma}
\newtheorem{corollary}[theorem]{Corollary} 
\newtheorem{proposition}[theorem]{Proposition} 
\newtheorem{fact}[theorem]{Fact} 

\newtheorem{maintheorem}{Theorem} 

\theoremstyle{definition}

\newtheorem*{setting*}{Setting} 

\theoremstyle{remark}
\newtheorem{remark}[theorem]{Remark}

\numberwithin{equation}{section}

\usepackage[shortlabels]{enumitem}
\setlist{nosep} 
\setlist[1]{labelindent=\parindent} 

\newcommand*{\RR}{\mathbb{R}}

\newcommand*{\NN}{\mathbb{N}}

\newcommand*{\sphere}{\mathbb{S}} 
\DeclareMathOperator{\support}{supp} 

\DeclarePairedDelimiter{\abs}{\lvert}{\rvert}
\DeclarePairedDelimiter{\norm}{\lVert}{\rVert}

\newcommand*{\PP}{\mathbb{P}}
\DeclareMathOperator{\EE}{\mathbb{E}}

\DeclareMathOperator{\Var}{Var}
\DeclareMathOperator{\Cov}{Cov}
\DeclareMathOperator{\Ent}{Ent}

\newcommand*{\indicator}{\mathbf{1}}
\newcommand*{\indicatorbraces}[1]{\mathbf{1}_{\{#1\}}}

\newcommand*{\abscont}{\ll}

\DeclarePairedDelimiterX\conditionalbrackets[2]{[}{]}%
{#1\,\delimsize\vert\,\mathopen{}#2}
\DeclarePairedDelimiterX\conditionalparenth[2]{(}{)}%
{#1\,\delimsize\vert\,\mathopen{}#2}

\DeclarePairedDelimiterXPP\conditionalEE[2]{\EE}{[}{]}{}%
{#1\,\delimsize\vert\,\mathopen{}#2}

\DeclarePairedDelimiterXPP\conditionalPP[2]{\PP}{(}{)}{}%
{#1\,\delimsize\vert\,\mathopen{}#2}

\newcommand*{\mix}{\textup{mix}}

\newcommand*{\capitalPhiTheta}{\Phi}

\newcommand*{\Borel}[1]{\mathcal{B}(#1)} 
\newcommand*{\Probab}[1]{\mathcal{P}(#1)} 
\newcommand*{\ProbabOne}[1]{\mathcal{P}_1(#1)} 

\newcommand*{\bfT}{\mathbf{T}}
\newcommand*{\calT}{\mathcal{T}}

\newcommand*{\bfTbar}{\overline{\mathbf{T}}\vphantom{\mathbf{T}}}

\newcommand{\mybar}[1]{\makebox[0pt]{$\phantom{#1}\overline{\phantom{#1}}$}#1}
\newcommand*{\calTbar}{\mybar{\mathcal{T}}}

\DeclarePairedDelimiterXPP\RelativeEntropy[2]{H}{(}{)}{}%
{#1\,\delimsize\vert\,\mathopen{}#2}

\DeclarePairedDelimiterXPP\chisquare[2]{\chi^2}{(}{)}{}%
{#1\,\delimsize\Vert\,\mathopen{}#2}

\DeclarePairedDelimiterXPP\Kullback[2]{D_{\textup{KL}}}{(}{)}{}%
{#1\,\delimsize\Vert\,\mathopen{}#2}

\DeclarePairedDelimiterXPP\BarCostQuadratic[2]{\mybar{\mathcal{T}}_2}{(}{)}{}%
{#1\,\delimsize\vert\,\mathopen{}#2}

\DeclarePairedDelimiterXPP\BarCostGeneral[3]{\mybar{\mathcal{T}}_{#1}}{(}{)}{}%
{#2\,\delimsize\vert\,\mathopen{}#3}

\newcommand*{\ConstantWithSubscript}[2]{C_{#1}(#2)}
\newcommand*{\ConstantWithSubscriptSuperscript}[3]{C_{#1}^{#2}(#3)}

\newcommand*{\ConstantPoincare}[1]{\ConstantWithSubscript{\textup{P}}{#1}}
\newcommand*{\ConstantLogSobolev}[1]{\ConstantWithSubscript{\textup{LS}}{#1}}
\newcommand*{\ConstantConvexPoincare}[1]{\ConstantWithSubscriptSuperscript{\textup{P}}{\textup{conv}}{#1}}
\newcommand*{\ConstantConvexLogSobolev}[1]{\ConstantWithSubscriptSuperscript{\textup{LS}}{\textup{conv}}{#1}}
\newcommand*{\ConstantConvexModifiedLogSobolev}[1]{\ConstantWithSubscriptSuperscript{\textup{mLS}}{\textup{conv}}{#1}}

\newcommand*{\ConstantTransportToggle}{\ConstantWithSubscript}

\newcommand*{\ConstantTransportTGeneral}[2]{\ConstantTransportToggle{\bfT_{#1}}{#2}}

\newcommand*{\ConstantTransportTTwo}[1]{\ConstantTransportToggle{\bfT_2}{#1}}

\newcommand*{\ConstantTransportTOne}[1]{\ConstantTransportToggle{\bfT_1}{#1}}

\newcommand*{\ConstantTransportTbarGeneralMinus}[2]{\ConstantTransportToggle{\bfTbar_{#1}^-}{#2}}
\newcommand*{\ConstantTransportTbarGeneralPlus}[2]{\ConstantTransportToggle{\bfTbar_{#1}^+}{#2}}
\newcommand*{\ConstantTransportTbarGeneralBoth}[2]{\ConstantTransportToggle{\bfTbar_{#1}}{#2}}

\newcommand*{\ConstantTransportTbarTwoMinus}[1]{\ConstantTransportToggle{\bfTbar_2^-}{#1}}

\usepackage[hidelinks,pagebackref]{hyperref} 
\hypersetup{final}

\title[Stability under mixtures]{Stability under mixtures
of transportation inequalities
and restricted log-Sobolev inequalities}

\author[R. Adamczak]{Rados{\l}aw Adamczak}
\address{University of Warsaw,
Institute of Mathematics,
Banacha 2,
02--097 Warsaw, Poland.}
\email{radamcz@mimuw.edu.pl}

\author[D. Kutek]{Dominik Kutek}
\address{University of Warsaw,
Institute of Mathematics,
Banacha 2,
02--097 Warsaw, Poland.}
\email{dk394337@mimuw.edu.pl}

\author[M. Strzelecki]{Micha{\l} Strzelecki}
\address{University of Warsaw,
Institute of Mathematics,
Banacha 2,
02--097 Warsaw, Poland.}
\email{michalst@mimuw.edu.pl}

\thanks{R.A.'s research was partially supported by The National Science Center, Poland, grant IMPRESS-U 2023/05/Y/ST1/00188}

\date{September 22, 2026}



\begin{document}

\begin{abstract}
We prove results about stability under mixtures for transport--entropy inequalities and restricted functional inequalities of Poincar\'e and log-Sobolev type.
We treat Talagrand inequalities of type $\mathbf{T}_2$ and $\mathbf{T}_1$ together with their generalizations,
as well as weak transportation inequalities introduced by Gozlan et al.
Our main results provide mild sufficient conditions under which a given inequality may be transferred from mixture components to the mixture itself.
\end{abstract}

\maketitle

\tableofcontents

%
%

\section{Introduction}

\subsection{Goal and motivation}

Mixtures of probability measures are an important object both in applications and in pure mathematics. They appear naturally in statistics and machine learning \cite{MR1274699,MR1789474,MR4719738} as well as in theoretical probability \cite{MR883646,MR786142}, and even in functional analysis \cite{MR1835574} and geometry \cite{MR3846841,MR4499279}. A generic question arising in the analysis of mixtures is: \emph{To what extent do they inherit the regularity of their components?} Clearly, a situation in which good properties of individual probability measures indeed pass to their mixtures is advantageous, as it allows, e.g., to build more complicated statistical models from simpler ones or to attempt to prove a desired property for a probability measure by finding its appropriate representation as a mixture.

In this article we investigate the above question in the context of transportation of measure and functional inequalities.
This line of research was started in 2010 by Chafa\"i and Malrieu~\cite{chafai-malrieu} who initiated the systematic study of fine properties of mixtures with respect to concentration of measure and Sobolev-type inequalities (some earlier results can be found in~\cite{MR2219345,MR2141356}).
Following their article, the topic attracted considerable attention, see, e.g. \cite{zimmermann-2013,zimmermann-2016,MR4231819,wang-wang,schlichting,bardet-et-al,chen-et-al};
see also~\cite{srinivasan2026twoscalecriteriapoincarelogsobolev}
for results motivated by Markov chain Monte Carlo
and~\cite{nilesweed2026reweightedinformationinequalities}
for transport-information inequalities for mixture distributions
with applications to the analysis of Langevin Monte Carlo for multimodal distributions.
We describe some of these developments in Section \ref{sec:state-of-the-art}.

Our goal is to provide results about stability under mixtures
for transport-entropy inequalities
and functional inequalities restricted to special classes of functions. To the best of our knowledge, such inequalities have not previously been studied from the perspective of stability under mixtures, apart from the
isolated results \cite[Theorem~1.4]{bardet-et-al} and \cite[Proposition~3.2]{MR2797986}.

The motivation for studying this question comes from the fact that transportation inequalities and restricted functional inequalities characterize various dimension-free concentration properties of probability measures. A detailed description of applications of such concentration inequalities is clearly beyond the scope of this introduction. We refer to the monographs \cite{MR1849347,boucheron-lugosi-massart-2012} for a detailed presentation and only mention that they are widely used in theoretical probability, statistics and computer science in the analysis of large sample behaviour of sequences of independent random variables, providing non-asymptotic deviation inequalities and leading to limit theorems. Effective criteria allowing to infer transportation-cost inequalities and dimension-free concentration results for mixtures may thus become an efficient tool in the study of inference in mixture-models as well as in problems of high-dimensional probability and geometry related to measures which have a mixture representation.

In the article we provide sufficient conditions for stability under mixtures of three different classes of transportation inequalities:
\begin{itemize}
\item Talagrand $\bfT_2$ inequalities, responsible for dimension-free subgaussian concentration for Lipschitz functions as well as their generalizations,
\item weak transportation inequalities, introduced by Gozlan et al. \cite{gozlan-roberto-samson-new-characterization}, corresponding to dimension-free concentration for convex Lipschitz functions,
\item $\bfT_1$ type inequalities and their generalizations, which as shown by Bobkov and G\"otze~\cite{bobkov-goetze}, characterize concentration for Lipschitz functions in fixed dimension.
\end{itemize}

It turns out that the behaviour under mixtures of each of these classes of inequalities is qualitatively different and their investigation requires a diverse range of approaches, leading to new technical difficulties not present in the study of stability for classical functional inequalities.

As a by-product of our approach we also obtain results on stability under mixtures of modified log-Sobolev inequalities, complementing and improving certain aspects of theorems from \cite{chen-et-al}.

The organization of the article is as follows. In this section, after introducing basic definitions, we proceed with a general review of the relevant  literature
 and a short, high-level overview of our main results. We postpone the definitions
of, among others, the Poincar\'e inequality, the log-Sobolev inequality,
the transportation inequalities $\bfT_2$ and $\bfT_1$,
and weak transportation inequalities,
and a full formulation of all our results to Section~\ref{sec:detailed-introduction}. The proofs are provided in Sections~\ref{sec:proofs-T_2},
\ref{sec:proofs-weak}, and \ref{sec:proofs-T_1}.

\subsection{Mixtures}

Whenever we consider a family $\{\mu_\theta\}_{\theta\in\Theta}$
of probability measures on $\RR^d$
and introduce their mixture $\mu_\mix$,
we assume the following setting
(unless specifically explained otherwise).

\begin{setting*}
Let $\{\mu_\theta\}_{\theta\in\Theta}$
be a family of Borel probability measures on $\RR^d$.
Let $m$ be a probability measure
on $(\Theta,\mathfrak{M})$,
where
$\mathfrak{M}$ is some $\sigma$-field of subsets of $\Theta$.
Suppose moreover that for every Borel set $A\subset \RR^d$
the map $\Theta\ni\theta \mapsto \mu_\theta(A)$ is Borel measurable.
The Borel probability measure $\mu_{\mix}$ defined by
\begin{equation}
\label{eq:definition-of-mixture}
\mu_{\mix}(A) \coloneqq \int_\Theta \mu_\theta(A) dm(\theta), \quad A\in\Borel{\RR^d},
\end{equation}
is called the \emph{mixture} of the measures $\mu_\theta$.
The measures $\mu_\theta$ are the
\emph{mixed components}
and we call
$m$ the \emph{mixing measure}.
\end{setting*}

The simplest situation in~\eqref{eq:definition-of-mixture}
is when
the index set $\Theta = \{0,1\}$ has only two elements and $m = (1-p)\delta_0 + p \delta_1$ is a convex combination of two  Dirac deltas (for some weight $p\in [0,1])$.
In this case the definition~\eqref{eq:definition-of-mixture} reads
\begin{equation}
\label{eq:definintion-two-component-mixture}
\mu_\mix = (1-p) \mu_0 + p \mu_1.
\end{equation}
More generally, one can consider finite convex combinations of the form
\begin{equation}
\label{eq:definintion-n-component-mixture}
\mu_\mix = p_1 \mu_1 + \dots + p_n \mu_n,
\end{equation}
where $p_1+\dots+p_n = 1$, $p_1,\dots, p_n\geq 0$.
Such mixtures are often considered
in information theory and theoretical computer science
(see, e.g., \cite{DBLP:conf/focs/Dasgupta99,DBLP:conf/stoc/SanjeevK01}),
but they also arise in statistical physics (see, e.g., \cite{MR2678899}).

Another important example of a mixture
is provided by the convolution $\mu*\nu$
of two probability measures
(i.e.,  the  law of a sum of independent random vectors with distributions $\mu$ and $\nu$).
Indeed,
for a Borel set $A\in\Borel{\RR^d}$,
\begin{equation}
\label{eq:definition-convolution-mixture}
\mu*\nu(A)
= \int_{\RR^d}\int_{\RR^d} \indicator_{\{x+y\in A\}} d\mu(x) d\nu(y)
=\int_{\RR^d} \mu_y(A) d\nu(y),
\end{equation}
where $\mu_y(A) \coloneqq \mu(\{a-y : a\in A\})$,
so $\mu*\nu$ is a mixture  with $\Theta = \RR^d$
and the mixing measure $m=\nu$ in~\eqref{eq:definition-of-mixture}.
The convolution $\mu*\nu$ can be
regarded either as a~mollification (of the less regular measure by the more regular measure)
or a perturbation (of the more regular measure by the less regular measure),
so it naturally appears in the formulation of many problems.

\subsection{State of the art}\label{sec:state-of-the-art}

Suppose that all the mixed  components $\mu_\theta$ satisfy, e.g., the Poincar\'e inequality
(see definition~\eqref{eq:definition-Poincare} below)
with uniformly bounded constants:
$\sup_{\theta\in\Theta} \ConstantPoincare{\mu_\theta} < \infty$.
Since the variance of a function with respect to the mixture $\mu_{\mix}$ can be
decomposed as
\begin{equation}
\label{eq:variance-of-mixture-introduction}
\Var_{\mu_{\mix}} (f)
=
\int_\Theta \Var_{\mu_\theta} (f)  dm(\theta)  + \Var_m\Bigl(\theta \mapsto \int_{\RR^d} f(x) d\mu_\theta(x) \Bigr),
\end{equation}
in order to prove that the mixture $\mu_{\mix}$ also satisfies the Poincar\'e inequality,
it suffices to estimate the second summand on the right-hand side of~\eqref{eq:variance-of-mixture-introduction}.
Note that
$\sup_{\theta\in\Theta} \ConstantPoincare{\mu_\theta} < \infty$
alone
is not a sufficient condition for the mixture $\mu_\mix$
to satisfy the Poincar\'e inequality,
because the connectedness of the support of the measure
-- which may fail for the mixture $\mu_\mix$ --
is a~necessary condition for the Poincar\'e inequality to hold.

Chafa\"i and Malrieu~\cite[Corollary~4.5]{chafai-malrieu}
proved that if $\mu_0$, $\mu_1$ are two probability measures on the real line (i.e., $d=1$)
satisfying the Poincar\'e inequality,
absolutely continuous with respect to the Lebesgue measure
and, moreover, a certain quantity expressed in terms of the densities of $\mu_0$ and $\mu_1$ is finite,
then
the constant in the Poincar\'e inequality of a two component mixture $\mu_\mix = (1-p)\mu_0 + p \mu_1$ remains uniformly bounded in $p\in[0,1]$.

Surprisingly,
a similar result does not hold in the setting of the log-Sobolev inequality
(see definition~\eqref{eq:definition-log-Sobolev} below):
here
the log-Sobolev constant of $\mu_{\mix}$ can explode as $p\to 0^+$ (see~\cite[Section~4.5.5]{chafai-malrieu}).
This phenomenon,
as well as other observations from~\cite{chafai-malrieu} and the subsequent research articles,
show that questions about stability under mixtures are subtle problems.

Zimmermann~\cite{zimmermann-2013,zimmermann-2016} and
 Zimmermann and Kemp \cite{MR4231819}, motivated by questions from Random Matrix Theory,
considered mixtures of the form $ \mu*\gamma_{\sigma^2}$,
where
$\mu$ is any compactly supported probability measure on $\RR^d$,
while $\gamma_{\sigma^2}$ is the Gaussian measure on $\RR^d$
with mean $0$ and covariance matrix $\sigma^2 I$.
In various regimes
(depending on the parameter $R^2/\sigma^2$,
where $R$ is the radius of the ball containing the support of $\mu$),
 they obtained estimates of the log-Sobolev constant of  $\mu*\gamma_{\sigma^2}$.
 Unfortunately, some of the estimates
 were exponential in the dimension $d$.

Bardet, Gozlan, Malrieu, and Zitt~\cite{bardet-et-al}
improved upon the results of  \cite{zimmermann-2013,MR4231819}
by focusing on quantitative estimates
and their dependence on the dimension $d$.
In \cite[Theorem~1.3]{bardet-et-al} they showed among others
that in the large variance case, $\sigma^2>R^2$,
the log-Sobolev constant of $\mu*\gamma_{\sigma^2}$
can be bounded uniformly in the dimension $d$.
Moreover, in \cite[Theorem 1.2]{bardet-et-al}
they gave a dimension-free estimate of the Poincar\'e constant  of $\mu*\gamma_{\sigma^2}$.

Wang and Wang \cite[Theorem~1.2]{wang-wang}
proved a qualitative result under a weaker assumption
(that $\mu$ is subgaussian instead of compactly supported).

Finally, Chen, Chewi, and Niles-Weed~\cite{chen-et-al}
proved the following general result:
if the measures $\mu_\theta$ all satisfy either the Poincar\'e or log-Sobolev inequality
with uniformly bounded constants
and, moreover,
\begin{equation*}
 \sup_{\theta,\theta'\in\Theta} \chisquare{\mu_\theta}{\mu_{\theta'}} < \infty
\end{equation*}
(see definition~\eqref{eq:definition-chi-squared-divergence} below),
then the mixture $\mu_\mix$ satisfies the same functional inequality
as the mixed components (see~\eqref{eq:chen-et-al-Poincare-main-result}
and~\eqref{eq:chen-et-al-log-Sobolev-main-result} below for
an exact statement with explicit constant estimates;
note that~\cite{chen-et-al} actually contains a stronger result:
one only needs to assume that the function $(\theta,\theta')\mapsto \chisquare{\mu_\theta}{\mu_{\theta'}}$
is in $L^p$ for some $1< p\leq \infty$).
In particular,
since the convolution $\mu*\gamma_{\sigma^2}$
 is a mixture of shifted Gaussian measures
(with the mixing measure equal to $\mu$, see~\eqref{eq:definition-convolution-mixture})
and the assumption that $\mu$ is compactly supported translates into the fact
that the chi-squared divergences of the mixed components are uniformly bounded,
this yields a dimension-free
estimate of the log-Sobolev constant of  $\mu*\gamma_{\sigma^2}$.

%
%

\subsection{Overview of new results}

Our first main result is the following theorem
about stability under mixtures of Talagrand's $\bfT_2$ inequality
(see definitions~\eqref{eq:definition-T_c-inequality} and~\eqref{eq:definition-quadratic-cost} below);
for a precise formulation with explicit constants
see Theorem~\ref{thm:stability-of-T_2} below.

\begin{maintheorem}[stability of $\bfT_2$ under mixtures]
\label{thm:stability-of-T_2-short-statement}
If all mixed components
satisfy
the $\bfT_2$ inequality with uniformly bounded constants
and
\[
 \sup_{\theta,\theta' \in \Theta} \chisquare{\mu_\theta}{\mu_{\theta'}} < \infty,
\]
then the mixture $\mu_\mix$ satisfies the $\bfT_2$ inequality.
\end{maintheorem}

The proof of Theorem~\ref{thm:stability-of-T_2-short-statement} is based on a characterization of transport inequalities in terms of
restricted log-Sobolev inequalities due to Gozlan, Roberto, and Samson~\cite{gozlan-roberto-samson-new-characterization}
and a modification of the approach used by Chen, Chewi, and Niles-Weed~\cite{chen-et-al}
to prove the stability under mixtures of the log-Sobolev inequality
(in our argument we have to bypass the defective version
of the restricted log-Sobolev inequality).
These ideas allow us to also  prove the following companion results:
\begin{itemize}
\item a similar result for more general, non-quadratic cost functions
(see~Theorem~\ref{thm:stability-of-T_c} below);
\item a proof of the stability of the log-Sobolev inequality
which does not pass through a defective log-Sobolev inequality
and yields better constants than in~\cite{chen-et-al}
(see Proposition~\ref{prop:stability-of-log-Sobolev-with-improved-constants} below);
\item versions for weak transportation inequalities
and log-Sobolev inequalities restricted to the classes of convex functions
(see Theorem~\ref{thm:stability-of-covex-log-Sobolev-under-chi-square}
and Corollary~\ref{cor:stability-of-Tbar_2^--under-chi-square} below).
\end{itemize}

Our second main result concerns weak transportation inequalities
(see definitions~\eqref{eq:definition-weak-transport-cost} and~\eqref{eq:definition-Tbar-minus} below),
introduced by Gozlan, Roberto, Samson, and Tetali~\cite{grst-Kantorovich-duality},
and deeply rooted in the work of Marton~\cite{MR838213,MR1404531,MR1392329}.
We stress that these inequalities do not require restrictive conditions on the distributions
-- in particular, in contrast to Theorem~\ref{thm:stability-of-T_2-short-statement}
and the results of~\cite{chen-et-al},
the mixed components may have disjoint supports.
Moreover,
under the assumptions of Theorem~\ref{thm:stability-of-Tbar_2^--short-statement}
the weak quadratic transport costs $\BarCostQuadratic{\cdot}{\cdot} $
(see definition~\eqref{eq:definition-notation-weak-quadratic-transport-cost})
appearing in the statement can be estimated from above
by the chi-square divergence $\chisquare[]{\cdot}{\cdot}$
(cf.\ Remark~\ref{rem:chi-square-vs-calTbar}).

\begin{maintheorem}[stability of $\bfTbar^-_2$ under mixtures]
\label{thm:stability-of-Tbar_2^--short-statement}
If all mixed components satisfy the $\bfTbar^-_2$ inequality
with uniformly bounded constants
and
\[
\sup_{\theta\in\Theta, \theta'\in\Theta} \BarCostQuadratic{\mu_\theta}{\mu_{\theta'}} < \infty,
\]
then the mixture $\mu_\mix$
satisfies the $\bfTbar_2^-$ inequality.
\end{maintheorem}

For a precise formulation with explicit constants
see Theorem~\ref{thm:stability-of-Tbar_2^-} below.
As above we have a number of satellite results:
\begin{itemize}
\item versions for more general, non-quadratic weak transport costs
(see Theorem~\ref{thm:stability-of-Tbar_c^-} below);
\item a corollary about stability under mixtures
of the convex log-Sobolev inequality
(see Corollary~\ref{cor:stability-of-convex-log-Sobolev} below);
\item a direct proof of the stability of the convex Poincar\'e inequality
(see Proposition~\ref{prop:stability-of-convex-Poincare-improved} below);
\item results under weaker assumptions for measures on the real line
(see Theorem~\ref{thm:stability-of-Tbar_c-real-line} below).
\end{itemize}
Somewhat related, we also provide a negative answer to \cite[Question~7.3]{adamczak-strzelecki}
about estimates of lower tails of convex functions
implied by the convex Poincar\'e inequality (see~Appendix~\ref{sec:counterexample-convex-poincare}).

The last main result concerns the stability of the $\bfT_1$ inequality
(see definition~\eqref{eq:definition-T_1-inequality} below);
for a precise formulation with explicit constants
see Theorem~\ref{thm:stability-of-T_1-improved} below.
It holds in the setting of measures on some Polish space $(E,d)$.

\begin{maintheorem}[stability of $\bfT_1$ under mixtures]
\label{thm:stability-of-T_1-short-statement}
If all mixed components
satisfy
the $\bfT_1$ inequality with uniformly bounded constants
and
for some $x_0\in E$ and some $\delta>0$ we have
\[
\int_{\Theta} \exp\Bigl(\delta \Bigl(\int_E d(x,x_0) d\mu_\theta(x) \Bigr)^2\Bigr) dm(\theta) < \infty,
\]
then the mixture $\mu_\mix$
satisfies the $\bfT_1$ inequality.
\end{maintheorem}

As above, we actually prove  a more general result concerning
generalized transportation inequalities considered by Gozlan~\cite{gozlan-integral-criteria} and Gozlan and L\'eonard~\cite{gozlan-leonard}.

%
%

\section{Results}

\label{sec:detailed-introduction}

\subsection{Basic definitions}

Given a probability measure $\mu$ on $\RR^d$\
and a (appropriately integrable) function $f\colon\RR^d\to\RR$,
the variance and entropy with respect to the measure $\mu$ are defined as
\begin{align*}
\Var_{\mu}(f)
&\coloneqq \int_{\RR^d} f(x)^2 d\mu(x) - \Bigl(\int_{\RR^d}f(x) d\mu(x)\Bigr)^2,\\
\Ent_\mu(f^2)
&\coloneqq \int_{\RR^d} f(x)^2 \log \bigl(f(x)^2\bigr)d\mu(x) - \int_{\RR^d} f(x)^2d\mu(x) \log\Bigl(\int_{\RR^d}f(x)^2 d\mu(x)\Bigr),
\end{align*}
respectively.
In some places it will be also convenient to use the shorter notation
and write $\EE_\mu f$ or $\mu_\theta(f)$ instead of $\int_{\RR^d} f(x) d\mu(x)$.

For two probability measures $\mu_1$ and $\mu_2$ on $\RR^d$,
with $\mu_1\abscont\mu_2$,
the chi-squared divergence
is defined as
\begin{equation}
\label{eq:definition-chi-squared-divergence}
\chisquare{\mu_1}{\mu_{2}}
\coloneqq \Var_{\mu_{2}}
\Bigl(\frac{d\mu_1}{d\mu_{2}} \Bigr)
=
\int_{\RR^d}\Bigl(\frac{d\mu_1}{d\mu_{2}} -1 \Bigr)^2 d\mu_{2}
= \int_{\RR^d} \frac{d\mu_1}{d\mu_{2}} d\mu_1 -1,
\end{equation}
while the relative entropy of $\mu_1$ with respect to $\mu_2$ is
\begin{equation}
\label{eq:definition-relative-entropy}
\RelativeEntropy{\mu_1}{\mu_2}
\coloneqq \Ent_{\mu_2}\Bigl( \frac{d\mu_1}{d\mu_2}\Bigr)
= \int_{\RR^d} \frac{d\mu_1}{d\mu_2} \log\bigl( \frac{d\mu_1}{d\mu_2} \bigr) d\mu_2
= \int_{\RR^d}  \log\bigl( \frac{d\mu_1}{d\mu_2} \bigr) d\mu_1.
\end{equation}
Both quantities are, by definition, equal to $+\infty$
if $\mu_1$ is not absolutely continuous with respect to $\mu_2$.
Note that by the elementary inequality $\ln(x) \leq x- 1$ for $x>0$,
\begin{equation}
\label{eq:kullback-leibler-vs-chi-square}
\RelativeEntropy{\mu_1}{\mu_2}
= \int_{\RR^d}  \log\bigl( \frac{d\mu_1}{d\mu_2} \bigr) d\mu_1
\leq \int_{\RR^d}   \frac{d\mu_1}{d\mu_2} - 1 d\mu_1
= \chisquare{\mu_1}{\mu_{2}}.
\end{equation}

In what follows,
the set of all (Borel) probability measures on $\RR^d$
is denoted $\Probab{\RR^d}$,
while $\ProbabOne{\RR^d}$
is the class of all probability measures on $\mu$ on $\RR^d$
such that $\int_{\RR^d} \abs{x} \mu(dx) < \infty$,
where  $\abs{\,\cdot\,}$ is the standard Euclidean norm.
We interpret all estimates of the form $A\leq B$ as:
if $B<\infty$, then $A$ is well-defined and the inequality holds.

Recall that one says that a probability measure $\mu$ on $\RR^d$
satisfies the Poincar\'e inequality
and the log-Sobolev inequality
with constant $C\in(0,\infty)$
if
\begin{align}
\label{eq:definition-Poincare}
\Var_{\mu}(f) &\leq C \int_{\RR^d} |\nabla f(x)|^2 d\mu(x),\\
\label{eq:definition-log-Sobolev}
\Ent_\mu(f^2) &\leq 2C \int_{\RR^d} |\nabla f(x)|^2 d\mu(x),
\end{align}
respectively,
for all sufficiently smooth functions $f\colon\RR^d\to\RR$.
We denote the optimal values of $C$
in~\eqref{eq:definition-Poincare}
and~\eqref{eq:definition-log-Sobolev}
by $\ConstantPoincare{\mu}$
and  $\ConstantLogSobolev{\mu}$, respectively.
The normalization $2C$ in~\eqref{eq:definition-log-Sobolev}
is chosen so that
\[
\ConstantPoincare{\mu} \leq \ConstantLogSobolev{\mu}.
\]

We recall that the mixture $\mu_\mix$ and the mixed family $\{\mu_\theta\}_{\theta\in\Theta}$ are always assumed to satisfy~\eqref{eq:definition-of-mixture}.		
In the results below,
$C_\Theta$ always denotes the supremum over $\theta\in\Theta$
of the constants with which the individual mixed components $\mu_\theta$
satisfy the functional or transportation inequality currently considered
(i.e., the meaning of $C_\Theta$ will change depending on the setting).
Similarly,
while $D_\Theta$ will always denote some kind of `diameter' of the mixed family $\{\mu_\theta\}_{\theta\in\Theta}$,
its exact meaning will change depending on the context.

%
%

\subsection{Stability of \texorpdfstring{$\bfT_2$}{T2} and related inequalities}

\label{sec:detailed-introduction-T_2}

Let $c\colon\RR^d\to[0,\infty)$
be a symmetric continuous function
such that $c(0)=0$.
The optimal transport cost
between two probability measures $\mu$ and $\nu$ on $\RR^d$
(with respect to the cost function $c$)
is
\begin{equation}
\calT_{c}(\mu,\nu) \coloneqq \inf_{\pi}\Bigl\{ \int_{\RR^d\times\RR^d} c(x-y) d\pi(x,y) \Bigr\},
\end{equation}
where the infimum is taken
over all couplings $\pi$ of $\mu$ and $\nu$.

We say that a probability measure $\mu$ on $\RR^d$
satisfies the transport--entropy inequality $\bfT_c$
with constant $C\in(0,\infty)$
if for all $\nu\in\Probab{\RR^d}$,
\begin{equation}
\label{eq:definition-T_c-inequality}
\calT_{c}(\mu,\nu) \leq C \RelativeEntropy{\nu}{\mu}.
\end{equation}
We denote the smallest constant $C$ in~\eqref{eq:definition-T_c-inequality}
by $\ConstantTransportTGeneral{c}{\mu}$.
By convention, $\ConstantTransportTGeneral{c}{\mu} = +\infty$ if $\mu$ does not satisfy the $\bfT_c$ inequality.

In the special case of the quadratic cost function
$c(\cdot) = \abs{\,\cdot\,}^2/2$,
we denote the inequality by $\bfT_2$
and the corresponding transport cost by
\begin{equation}
\label{eq:definition-quadratic-cost}
\calT_{2}(\mu,\nu)
\coloneqq
\inf_{\pi}\Bigl\{ \int_{\RR^d\times\RR^d} \frac{\abs{x-y}^2}{2}  d\pi(x,y) \Bigr\}.
\end{equation}
The importance of this inequality was emphasized  by Talagrand~\cite{MR1392331}.
It is known that
\begin{equation}
\label{eq:implication-chain}
\ConstantPoincare{\mu} \leq \ConstantTransportTTwo{\mu} \leq \ConstantLogSobolev{\mu}.
\end{equation}
Here the first inequality follows by a perturbation argument and Taylor's expansion,
while the second one is the celebrated result of Otto and Villani~\cite{MR1760620}.
Bobkov, Gentil, and Ledoux~\cite{MR1846020} provided a proof of the second inequality based on Hamilton--Jacobi equations.
Yet another approach was proposed by Gozlan~\cite{MR2573565}
who showed that the inequality $\bfT_2$ is equivalent to dimension-free subgaussian concentration of measure property of the measure $\mu$.
Let us also stress that in general the inequality $\bfT_2$
is strictly weaker than the log-Sobolev inequality (see, e.g., \cite{MR2257848,MR2946156}).

%
%

As mentioned above, Chen, Chewi, and Niles-Weed~\cite[Theorem~1]{chen-et-al} proved that
if all of the measures $\mu_\theta$ satisfy either the Poincar\'e or log-Sobolev inequality
with uniformly bounded constants
and, moreover,
\begin{equation}
\label{eq:uniform-bound-on-chisquared-divergences-introduction}
M_{\Theta}\coloneqq \sup_{\theta,\theta'\in\Theta} \chisquare{\mu_\theta}{\mu_{\theta'}} < \infty,
\end{equation}
then the mixture $\mu_\mix$ satisfies the same functional inequality
as the mixed components
\begin{align}
\label{eq:chen-et-al-Poincare-main-result}
\ConstantPoincare{\mu_\mix}
&\leq \sup_{\theta\in\Theta} \ConstantPoincare{\mu_\theta}
\cdot
(2+M_{\Theta}),\\
\label{eq:chen-et-al-log-Sobolev-main-result}
\ConstantLogSobolev{\mu_{\mix}}
&\leq \sup_{\theta\in\Theta} \ConstantLogSobolev{\mu_\theta}
\cdot 3(2+M_{\Theta})\bigl(1+\log(1+M_{\Theta})\bigr).
\end{align}
We remark that the numerical constants in both estimates can be improved,
see Propositions~\ref{prop:stability-of-Poincare-with-improved-constants}
and~\ref{prop:stability-of-log-Sobolev-with-improved-constants} below.

We prove the following analogue for Talagrand's $\bfT_2$ inequality.

\begin{theorem}[main result: stability of $\bfT_2$ under mixtures]
\label{thm:stability-of-T_2}
Let
$\{\mu_\theta\}_{\theta\in\Theta}$
be a family of probability measures on $\RR^d$
which satisfy the $\bfT_2$ inequality with uniformly bounded constants:
\[
C_\Theta \coloneqq \sup_{\theta\in\Theta} \ConstantTransportTTwo{\mu_\theta} < \infty.
\]
Suppose moreover that
\[
M_\Theta\coloneqq \sup_{\theta,\theta' \in \Theta} \chisquare{\mu_\theta}{\mu_{\theta'}} < \infty.
\]
Then the mixture $\mu_\mix$,
defined as in~\eqref{eq:definition-of-mixture},
satisfies the $\bfT_2$ inequality
and
\[
\ConstantTransportTTwo{\mu_\mix} \leq
C_\Theta
\cdot
\bigl(16 + \frac{1}{2} (2+M_{\Theta})\log(1+M_{\Theta})\bigr).
\]
\end{theorem}

Our approach to the above theorem is based on a modification of the ideas developed in \cite{chen-et-al} for proving the inequality \eqref{eq:chen-et-al-log-Sobolev-main-result} and on a characterization of transportation inequalities in terms of restricted log-Sobolev inequalities, introduced by Gozlan, Roberto, and Samson \cite{gozlan-roberto-samson-new-characterization} (see Theorem \ref{thm:a-new-characterization-of-T_c} below). The original approach from \cite{chen-et-al} relied on
obtaining first the defective form of the log-Sobolev inequality and then improving it to the standard version via Rothaus' lemma \cite{MR826433}.
Such a strategy cannot work in our setting because of the restriction of the class of functions.
Our modification allows to bypass the defective version of the log-Sobolev inequality and so it can be applied also in the setting of other restricted inequalities,
see Theorem~\ref{thm:stability-of-covex-log-Sobolev-under-chi-square}
and Corollary~\ref{cor:stability-of-Tbar_2^--under-chi-square} below.
As a by-product of our proof
we obtain also the following version of~\eqref{eq:chen-et-al-log-Sobolev-main-result}
with better numerical constants.

\begin{proposition}[numerical constants in~\eqref{eq:chen-et-al-log-Sobolev-main-result}]
\label{prop:stability-of-log-Sobolev-with-improved-constants}
Let
$\{\mu_\theta\}_{\theta\in\Theta}$
be a family of probability measures on $\RR^d$
which satisfy the log-Sobolev inequality with uniformly bounded constants:
\[
C_\Theta \coloneqq \sup_{\theta\in\Theta} \ConstantLogSobolev{\mu_\theta} < \infty.
\]
Suppose moreover that
\[
M_{\Theta}\coloneqq \sup_{\theta,\theta' \in \Theta} \chisquare{\mu_\theta}{\mu_{\theta'}} < \infty.
\]
Then the mixture $\mu_\mix$,
defined as in~\eqref{eq:definition-of-mixture},
satisfies the log-Sobolev inequality
with
\[
\ConstantLogSobolev{\mu_\mix} \leq
C_\Theta
\cdot \Bigl(2  + \frac{1}{4}(2+M_\Theta)\log(1+M_\Theta)\Bigr).
\]
\end{proposition}

We also have the following result for more general transport cost.

\begin{theorem}[stability of $\bfT_c$ under mixtures]
\label{thm:stability-of-T_c}
Let $\alpha\colon\RR\to[0,\infty)$ be a symmetric convex function  of class $C^1$,
such that $\alpha(t) = \frac{1}{2}t^2$ for $t\in[0,t_0]$ for some $t_0>0$, and $\alpha'$ is concave on $[0,\infty)$.
Let $c\colon\RR^d\to[0,\infty)$
be a cost function of the form
\begin{equation}
\label{eq:special-form-of-the-cost-function}
c(x) = \sum_{i=1}^d \alpha(x_i), \quad x=(x_1,\dots,x_d)\in\RR^d.
\end{equation}

Let
$\{\mu_\theta\}_{\theta\in\Theta}$
be a family of probability measures on $\RR^d$
which satisfy the $\bfT_c$ inequality with uniformly bounded constants:
\[
C_\Theta \coloneqq \sup_{\theta\in\Theta} \ConstantTransportTGeneral{c}{\mu_\theta} < \infty.
\]
Suppose moreover that
\[
M_\Theta\coloneqq \sup_{\theta,\theta' \in \Theta} \chisquare{\mu_\theta}{\mu_{\theta'}} < \infty.
\]
Then the mixture $\mu_\mix$,
defined as in~\eqref{eq:definition-of-mixture},
satisfies the $\bfT_c$ inequality
and
\[
\ConstantTransportTGeneral{c}{\mu_\mix} \leq
C_\Theta
\cdot
\bigl(16 + 2(2+M_{\Theta})\log(1+M_{\Theta})\bigr).
\]
\end{theorem}

A typical example of a function satisfying the assumptions of the above theorem
is a smooth convex version of the function
\[
\alpha_p(t) = \min\{t^2/2, \abs{t}^p/p\}, \quad t\in\RR,
\]
where $p\in [1,2]$.

%
%

\subsection{Stability of weak transportation inequalities}

\label{sec:detailed-introduction-weak}

Let $c\colon\RR^d\to[0,\infty)$
be a continuous convex function
such that $c(0)=0$.
The weak (barycentric) optimal transport cost between
two probability measures $\mu_1$ and $\mu_2$
on $\RR^d$,
where we assume that   $\mu_2 \in \mathcal{P}_1(\RR^d)$,
is defined as
\begin{displaymath}
\BarCostGeneral{c}{\mu_2}{\mu_1}
\coloneqq \inf_{\pi}\Bigl\{
 \int_{\RR^d}
 c\bigl(x - \int_{\RR^d} y p_x(dy)\bigr) \mu_1(dx)
 \Bigr\},
\end{displaymath}
where the infimum is taken over all couplings $\pi$ between $\mu_1$ and $\mu_2$,
while for $x \in \RR^d$, $p_x(\cdot)$ is the conditional measure defined ($\mu_1$ almost surely) by $\pi(dxdy) = p_x(dy)\mu_1(dx)$.
Note that this weak transport cost is not symmetric
and in the probabilistic notation one can write
\begin{equation}
\label{eq:definition-weak-transport-cost}
\BarCostGeneral{c}{\mu_2}{\mu_1}
 = \inf_{(X_1,X_2)}
 \bigl\{
 \EE c\bigl(X_1 - \conditionalEE{X_2}{X_1}\bigr)
 \bigr\},
\end{equation}
where the infimum is taken over all pairs of random vectors $(X_1,X_2)$
with values in $\RR^d \times \RR^d$,
such that $X_1$ is distributed according to $\mu_1$
and $X_2$ according to $\mu_2$.
Note that if the convex cost function $c$ is symmetric,
then Jensen's inequality implies that
\[
\max\bigl\{ \BarCostGeneral{c}{\mu_2}{\mu_1}, \BarCostGeneral{c}{\mu_1}{\mu_2}\bigr\}\leq \calT_c(\mu_1, \mu_2).
\]

We say 	that $\mu$ satisfies the
weak transportation inequality
$\bfTbar_c^-$
with constant $C\in(0,\infty)$
if for all probability measures $\nu\in\mathcal{P}_1(\RR^d)$ we have
\begin{equation}
\label{eq:definition-Tbar-minus}	
\BarCostGeneral{c}{\mu}{\nu} \leq C \RelativeEntropy{\nu}{\mu}.
\end{equation}
We say 	that $\mu$ satisfies
the weak transportation inequality
$\bfTbar_c^+$
with constant $C\in(0,\infty)$
if for all probability measures $\nu\in\mathcal{P}_1(\RR^d)$ we have
\begin{equation}
\label{eq:definition-Tbar-plus}	
\BarCostGeneral{c}{\nu}{\mu} \leq C \RelativeEntropy{\nu}{\mu}
\end{equation}
(note that $\mu$ and $\nu$ switched places on the left-hand side).	
We say that $\mu$ satisfies
the weak transportation inequality
$\bfTbar_c$
with constant $C\in(0,\infty)$
if both $\bfTbar_c^-$ and $\bfTbar^+_c$ hold (with constant $C$).
We denote the best constants in these inequalities by
$\ConstantTransportTbarGeneralMinus{c}{\mu}$,
$\ConstantTransportTbarGeneralPlus{c}{\mu}$,
$\ConstantTransportTbarGeneralBoth{c}{\mu}$,
respectively.
Clearly,
\[
\max\{ \ConstantTransportTbarGeneralMinus{c}{\mu},
\ConstantTransportTbarGeneralPlus{c}{\mu}\}
=
\ConstantTransportTbarGeneralBoth{c}{\mu}
\leq
\ConstantTransportTGeneral{c}{\mu}.
\]

In the special case of the quadratic cost function
$c(\cdot) = \abs{\,\cdot\,}^2/2$,
we denote the inequalities by $\bfTbar_2^-$, $\bfTbar_2^+$, $\bfTbar_2$,
and the corresponding transport cost by
\begin{equation}
\label{eq:definition-notation-weak-quadratic-transport-cost}
\BarCostQuadratic{\mu_2}{\mu_1}
\coloneqq \inf_{\pi}
\Bigl\{
\int_{\RR^d} \frac{\bigl|x - \int_{\RR^d} y p_x(dy)\bigr|^2}{2}\mu_1(dx)
\Bigr\}.
\end{equation}

We say that a probability measure $\mu$ on $\RR^d$
satisfies the convex Poincar\'e inequality
with constant $C\in(0,\infty)$
if
\begin{equation}
\label{eq:definition-convex-Poincare}
\Var_{\mu}(f) \leq C \int_{\RR^d} |\nabla f(x)|^2 d\mu(x)
\end{equation}
for all convex functions $f\colon\RR^d\to\RR$.
Here, for the sake of notational consistency with \cite{MR3311918,adamczak-strzelecki}, by $\abs{\nabla f(x)}$ we mean the length of gradient at $x$, defined as
\begin{equation}
\label{eq:grad-length}
\abs{\nabla f(x)} \coloneqq \limsup_{y \to x} \frac{\abs{f(y) - f(x)}}{\abs{y-x}}.
\end{equation}
This coincides with the length of the `true' gradient provided $f$ is differentiable at $x$
(and it is in fact enough to assume that \eqref{eq:definition-convex-Poincare}
holds for smooth convex Lipschitz functions,
see Lemma~\ref{lem:convex-poincare-class-of-funtions}).
We denote the optimal value of $C$
in~\eqref{eq:definition-convex-Poincare}
by $\ConstantConvexPoincare{\mu}$. We remark that it is known \cite{adamczak-strzelecki} that the convex Poincar\'e inequality is equivalent to the $\bfTbar_c^-$ inequality for the linear-quadratic cost $c(x) = \min(|x|^2/2,|x|)$
(with universal constants)
and also to the corresponding $\bfTbar_c$ inequality
(with constants depending on the dimension~$d$,
cf.\ Appendix~\ref{sec:counterexample-convex-poincare}).

We say that a probability measure $\mu$ on $\RR^d$
satisfies the convex log-Sobolev inequality
with constant $C\in(0,\infty)$
if
\begin{equation}
\label{eq:definition-convex-log-Sobolev}
\Ent_{\mu}(e^f) \leq
\frac{C}{2} \int_{\RR^d} |\nabla f(x)|^2 e^f d\mu(x)
\end{equation}
for all sufficiently smooth convex functions $f\colon\RR^d\to\RR$
(and, say Lipschitz and bounded from below,
but by arguments as in the proof of Lemma~\ref{lem:convex-poincare-class-of-funtions}
one can also pass to general convex functions with the length of the gradient).
We denote the optimal value of $C$
in~\eqref{eq:definition-convex-log-Sobolev}
$\ConstantConvexLogSobolev{\mu}$.
The normalization $C/2$ in~\eqref{eq:definition-convex-log-Sobolev}
is chosen because
substituting $f\coloneqq \exp(f/2)$ in~\eqref{eq:definition-log-Sobolev}
would yield an inequality of the form~\eqref{eq:definition-convex-log-Sobolev}.
The convex log-Sobolev inequality
and the $\bfTbar_2^-$ inequality
are equivalent~\cite[Section~8]{grst-Kantorovich-duality}
(cf.\ also Theorem~\ref{thm:characterization-from-Kantorovich-duality}
and Corollary~\ref{cor:weak-transport-implies-convex-Poincare} below) and, moreover,
\begin{equation}
\label{eq:implication-chain-convex-setting}
\ConstantConvexPoincare{\mu}
\leq \ConstantTransportTbarTwoMinus{\mu}
\leq \ConstantConvexLogSobolev{\mu}
\leq 4 \ConstantTransportTbarTwoMinus{\mu}.
\end{equation}
%
%

Having introduced all the notation and terminology,
we are ready to present results about
stability of weak transportation inequalities
and functional inequalities for convex functions.
We start with the following convex counterpart of~\eqref{eq:chen-et-al-Poincare-main-result}
(i.e., of~\cite[Theorem~1 (1)]{chen-et-al} or, more precisely, its version
stated in Proposition~\ref{prop:stability-of-Poincare-with-improved-constants} below).

\begin{proposition}
\label{prop:stability-of-convex-Poincare-under-chi-square}
Let
$\{\mu_\theta\}_{\theta \in \Theta}$
be a family of  probability measures on $\RR^d$
which satisfy the convex Poincar\'e inequality
with uniformly bounded constants:
\[
C_\Theta \coloneqq
\sup_{\theta\in\Theta} \ConstantConvexPoincare{\mu_\theta} < \infty.
\]  	
Suppose moreover that
\[
M_\Theta\coloneqq \sup_{\theta,\theta' \in \Theta} \chisquare{\mu_\theta}{\mu_{\theta'}} < \infty.
\]
Then the mixture $\mu_\mix$,
defined as in~\eqref{eq:definition-of-mixture},
satisfies the convex Poincar\'e inequality
and
\[
\ConstantConvexPoincare{\mu_\mix}
\leq
C_\Theta \cdot \bigl(1 +\frac{1}{2}  M_\Theta\bigr).
\]
\end{proposition}

If one works under the assumption that the $\chi^2$-distances
between the mixed components are uniformly bounded,
then the general method of the proof of Theorem~\ref{thm:stability-of-T_2}
can also be employed in the setting of weak transportation inequalities
and the convex log-Sobolev inequality
-- see Remark~\ref{rem:chi-square-vs-calTbar} for a short statement of these results
and comparison with our other results valid under less restrictive assumptions
about the mixed components.

We continue with results which do not assume
that the mixed components have densities with respect to each other.
We first provide a refinement of \cite[Proposition 5.3]{adamczak-strzelecki}.
This result is not directly comparable with Proposition~\ref{prop:stability-of-convex-Poincare-under-chi-square}, see Remarks~\ref{rem:chi-square-vs-calTbar} and~\ref{rem:chi-square-vs-calTbar-for-Poincare}.

\begin{proposition}
\label{prop:stability-of-convex-Poincare-improved}
Let
$\{\mu_\theta\}_{\theta \in \Theta}$
be a family of  probability measures on $\RR^d$
which satisfy the convex Poincar\'e inequality
with uniformly bounded constants:
\[
C_\Theta \coloneqq \sup_{\theta\in\Theta} \ConstantConvexPoincare{\mu_\theta} < \infty.
\]  	
Suppose moreover that
\[
D_\Theta \coloneqq \sup_{\theta\in\Theta, \theta'\in\Theta}
\BarCostQuadratic{\mu_\theta}{\mu_{\theta'}} < \infty.
\]
Then the mixture $\mu_\mix$,
defined as in~\eqref{eq:definition-of-mixture},
satisfies the convex Poincar\'e inequality
and
\[
\ConstantConvexPoincare{\mu_\mix}
\leq
C_\Theta
 + D_\Theta.
\]
\end{proposition}

Observe that the above proposition allows for mixtures of measures with disjoint supports, which is in contrast with results for functional or transportation inequalities for all smooth functions, for which connectedness of support is a necessary condition.
This phenomenon is characteristic for inequalities related to convex concentration and will be present also in other results of this section.

Our main result for weak transportation inequalities reads as follows.

\begin{theorem}[main result: stability of $\bfTbar^-_2$ under mixtures]
\label{thm:stability-of-Tbar_2^-}
Let
$\{\mu_\theta\}_{\theta \in \Theta}$
be a family of  probability measures on $\RR^d$
which satisfy the $\bfTbar^-_2$ inequality
with uniformly bounded constants:
\[
C_\Theta \coloneqq \sup_{\theta\in\Theta} \ConstantTransportTbarTwoMinus{\mu_\theta} < \infty.
\]  	
Suppose moreover that
\[
D_{\Theta} \coloneqq \sup_{\theta\in\Theta, \theta'\in\Theta} \BarCostQuadratic{\mu_\theta}{\mu_{\theta'}} < \infty.
\]
Then the mixture $\mu_\mix$,
defined as in~\eqref{eq:definition-of-mixture},
satisfies the $\bfTbar_2^-$ inequality
and
\[
\ConstantTransportTbarTwoMinus{\mu_\mix}
 \leq
\max\bigl\{ 2 C_\Theta,  4 D_\Theta\bigr\}.
\]
\end{theorem}

By~\eqref{eq:implication-chain-convex-setting},
this theorem immediately implies the following corollary.

\begin{corollary}
\label{cor:stability-of-convex-log-Sobolev}
Let
$\{\mu_\theta\}_{\theta \in \Theta}$
be a family of  probability measures on $\RR^d$
which satisfy the convex log-Sobolev inequality
with uniformly bounded constants:
\[
C_\Theta \coloneqq \sup_{\theta\in\Theta} \ConstantConvexLogSobolev{\mu_\theta} < \infty
\]
and such that $\int_{\RR^d} \exp(s\abs{x}) d\mu_{\theta}(x)<\infty$ for all $\theta\in\Theta$ and $s>0$.
Suppose moreover that
\[
D_{\Theta}
\coloneqq
\sup_{\theta\in\Theta, \theta'\in\Theta}
\BarCostQuadratic{\mu_\theta}{\mu_{\theta'}} < \infty.
\]
Then the mixture $\mu_\mix$,
defined as in~\eqref{eq:definition-of-mixture},
satisfies the convex log-Sobolev inequality
and
\[
\ConstantConvexLogSobolev{\mu_\mix}
 \leq
\max\bigl\{ 8 C_\Theta,  16 D_\Theta\bigr\}.
\]
\end{corollary}

Let us note that in contrast to the proof of Theorem~\ref{thm:stability-of-T_2}
(where we exploit the characterization in terms of restricted log-Sobolev inequalities
and work with functions),
in order to prove Theorem~\ref{thm:stability-of-Tbar_2^-}
we work directly at the level of weak transportation inequalities and couplings;
moreover, the proof uses a deep inequality of Marton~\cite{MR1392329}.
Only then do we deduce the results about the  convex log-Sobolev inequality
stated in Corollary~\ref{cor:stability-of-convex-log-Sobolev}.

Several other remarks are in order.

\begin{remark}
Note that we can express any probability measure on $[0,1]$
as a mixture of  Dirac deltas (supported at points from the interval $[0,1]$);
in this case $C_\Theta = 0$ and  $D_\Theta\leq 1/2$ in Corollary~\ref{cor:stability-of-convex-log-Sobolev}.
By tensorization we recover the convex log-Sobolev inequality
for compactly supported product measures
obtained by Ledoux~\cite[Theorem~1.2]{MR1399224}
(up to a multiplicative factor of $4$, which we loose via~\eqref{eq:implication-chain-convex-setting}).
\end{remark}

\begin{remark}
\label{rem:chi-square-vs-calTbar}
As mentioned above,
our proof of Theorem~\ref{thm:stability-of-T_2}
can be adapted to the setting of weak transportation inequalities
and one can prove that
\begin{align}
\label{eq:stability-of-convex-log-Sobolev-under-chi-square-remark}
\ConstantConvexLogSobolev{\mu_\mix} \leq
\sup_{\theta\in\Theta} \ConstantConvexLogSobolev{\mu_\theta}
\cdot
\bigl(2 + \frac{1}{4}(2+M_{\Theta})\log(1+M_{\Theta})\bigr),\\
\label{eq:stability-of-weak-transport-quadratic-under-chi-square-remark}
\ConstantTransportTbarTwoMinus{\mu_\mix} \leq
\sup_{\theta\in\Theta} \ConstantTransportTbarTwoMinus{\mu_\theta}
\cdot
\bigl(8 + \frac{1}{4}(2+M_{\Theta})\log(1+M_{\Theta})\bigr),
\end{align}
where
\[
M_\Theta \coloneqq \sup_{\theta,\theta' \in \Theta} \chisquare{\mu_\theta}{\mu_{\theta'}}
\]
(see Theorem~\ref{thm:stability-of-covex-log-Sobolev-under-chi-square} and Corollary~\ref{cor:stability-of-Tbar_2^--under-chi-square};
we remark that this can be further generalized to weak transportation inequalities
with  more general transport costs, cf.\ Theorem~\ref{thm:equivalence}).
However, if we denote
\begin{equation*}
C_\Theta \coloneqq \sup_{\theta\in\Theta} \ConstantTransportTbarTwoMinus{\mu_\theta} < \infty,
\end{equation*}
then
\begin{align*}
\BarCostQuadratic{\mu_\theta}{\mu_{\theta'}}
\leq C_\Theta \RelativeEntropy{\mu_{\theta'}}{\mu_\theta}
\leq C_\Theta \chisquare{\mu_{\theta'}}{\mu_\theta}
\leq C_\Theta M_\Theta,
\end{align*}
where we used the fact that $\mu_\theta$ satisfies $\bfTbar_2^-$
and the relation~\eqref{eq:kullback-leibler-vs-chi-square}.
Thus,
\[
D_{\Theta} \coloneqq \sup_{\theta\in\Theta, \theta'\in\Theta} \BarCostQuadratic{\mu_\theta}{\mu_{\theta'}} \leq C_\Theta M_\Theta
\]
and Theorem~\ref{thm:stability-of-Tbar_2^-} implies
\[
\ConstantTransportTbarTwoMinus{\mu_\mix}
 \leq  C_\Theta \cdot (2+ 4 M_\Theta).
\]
This bound is strictly better than~\eqref{eq:stability-of-weak-transport-quadratic-under-chi-square-remark}
if $M_\Theta\to 0^+$ or $M_\Theta\to\infty$.
On the other hand,
the estimate~\eqref{eq:stability-of-weak-transport-quadratic-under-chi-square-remark} is sharper for some intermediate values of $M_\Theta$
 and its proof is considerably simpler
(see Theorem~\ref{thm:stability-of-covex-log-Sobolev-under-chi-square}
and Corollary~\ref{cor:stability-of-Tbar_2^--under-chi-square}).
\end{remark}

\begin{remark}
\label{rem:chi-square-vs-calTbar-for-Poincare}
Proposition~\ref{prop:stability-of-convex-Poincare-improved}
does not directly imply Proposition~\ref{prop:stability-of-convex-Poincare-under-chi-square}
since in these results we only assume that the mixed components satisfy
the convex Poincar\'e inequality,
so they do not have to satisfy the $\bfTbar_2^-$ inequality
and the quantities
\[
D_{\Theta} \coloneqq \sup_{\theta\in\Theta, \theta'\in\Theta} \BarCostQuadratic{\mu_\theta}{\mu_{\theta'}},
\qquad
M_\Theta \coloneqq \sup_{\theta,\theta' \in \Theta} \chisquare{\mu_\theta}{\mu_{\theta'}}
\]
cannot be compared as above.
\end{remark}

We proceed with results for non-quadratic weak transport costs.
For $p\geq 2$ we denote by
$\BarCostGeneral{p}{\cdot}{\cdot}$
and $\BarCostGeneral{2,p}{\cdot}{\cdot}$
the weak transport costs
associated with the cost functions
$x\mapsto \frac{1}{p} \abs{x}^p$
and $x\mapsto \abs{x}^2 + \abs{x}^p$ ($x\in\RR^d$),
respectively. Note that, by Jensen's inequality,
\[
\BarCostGeneral{2,p}{\mu}{\nu}
\leq \bigl( p\BarCostGeneral{p}{\mu}{\nu} \bigr)^{2/p} + p\BarCostGeneral{p}{\mu}{\nu}.
\]

\begin{theorem}
\label{thm:stability-of-Tbar_c^-}
Assume that $c \colon \RR^d \to [0,\infty)$ is a convex function such that for some constants $\varepsilon > 0$, $p \ge 2$ and $K$,
\begin{displaymath}
  \varepsilon\indicatorbraces{|x| < \varepsilon}|x|^2
  \le c(x)
  \le K\bigl(\abs{x}^2 +\abs{x}^p\bigr).
\end{displaymath}

Let
$\{\mu_\theta\}_{\theta \in \Theta}$
be a family of  probability measures on $\RR^d$
which satisfy the $\bfTbar^-_c$ inequality
with uniformly bounded constants:
\[
C_\Theta \coloneqq \sup_{\theta\in\Theta} \ConstantTransportTbarGeneralMinus{c}{\mu_\theta} < \infty.
\]  	
Suppose moreover that
\[
D_{\Theta} \coloneqq \sup_{\theta\in\Theta, \theta'\in\Theta} \BarCostGeneral{2,p}{\mu_\theta}{\mu_{\theta'}} < \infty.
\]
Then the mixture $\mu_\mix$,
defined as in~\eqref{eq:definition-of-mixture},
satisfies the $\bfTbar_{2c(\cdot/2)}^-$ inequality
and
\[
\ConstantTransportTbarGeneralMinus{2c(\cdot/2)}{\mu_\mix}
 \leq
\max\bigl\{ C_\Theta,  4K D_\Theta\bigr\}.
\]
\end{theorem}

The assumption that the cost function $c$
can be estimated from below by the quadratic cost on some small ball
guarantees that the mixed components satisfy the convex Poincar\'e inequality
and due to concentration all functions considered in the proofs are integrable.
Also, for simplicity, in the assertion the weak transportation inequality
with the cost function $2c(\cdot/2)$ is used,
but in many situations this new cost function can be up to constants estimated from below by the original cost function $c$ (cf., e.g.,
Lemma~\ref{lem:convex-conjugate-elementary-lemma}).

As an example, note that a cost of the form
\[
c(x) = \sum_{i=1}^d \min\{ \abs{x_i}^2, \abs{x_i}^r\}, \quad x=(x_1,\dots,x_d)\in\RR^d,
\]
where $r\in[1,2]$, satisfies the assumptions of the above theorem with $p=2$ (formally it is not convex, but one can modify it to an equivalent convex function; we prefer this formulation as it corresponds to the way in which one usually presents inequalities between the exponential and Gaussian regimes).
Moreover,
 because of tensorization and characterizations on the real line from \cite{grsst-characterization,shu-strzelecki},
it is known when a product measure satisfies
 a weak transportation inequality $\bfTbar_c^-$ with such a cost.
 Theorem~\ref{thm:stability-of-Tbar_c^-} allows us to pass from product measures to mixtures of product measures. In a similar spirit one may consider the \emph{strongly subgaussian} regime, given by $c(x) = \sum_{i=1}^d \max\{\abs{x_i}^2,\abs{x_i}^p\}$ for $p > 2$.

Similarly as in the quadratic case,
the above theorem yields a corollary about certain \emph{modified} log-Sobolev inequalities for convex functions.
Recall that for a convex function $c\colon\RR^d \to [0,\infty)$,
one defines its convex conjugate (or Legendre transform)
by
\begin{equation}
	\label{eq:definition-Legendre-transform}
	c^*(y) = \sup_{x \in \RR^d}\{ \langle x,y \rangle - c(y)\}, \qquad y \in \RR^d.
\end{equation}
We say that $\mu$ satisfies the convex modified log-Sobolev inequality
with constant $C\in (0,\infty)$ if
\begin{equation}
\label{eq:definition-convex-modified-log-Sobolev}
\Ent_{\mu}(e^f) \le C \int_{\RR^d} c^*\bigl(\nabla f\bigr)e^f d\mu
\end{equation}
for all all convex, Lipschitz, $C^1$, and bounded from below functions $f\colon\RR^d\to\RR$.
Let us denote the best constant in~\eqref{eq:definition-convex-modified-log-Sobolev}
by $\ConstantConvexModifiedLogSobolev{\mu}$;
note that we suppress the dependence on the function $c^*$ appearing on the right-hand side of~\eqref{eq:definition-convex-modified-log-Sobolev} in the notation.
For a wide class of cost functions $c$ the inequality $\bfTbar_c^-$
is (up to constants) equivalent to the convex modified log-Sobolev inequality~\eqref{eq:definition-convex-modified-log-Sobolev}
(see Theorem~\ref{thm:equivalence} below).

\begin{corollary}
	\label{cor:stability-of-Tbar_c^--and-convex-mSLI-R^d}
	Let $p\geq 2$ and let $c\colon \RR^d\to [0,\infty)$
	be a convex function satisfying the assumptions of Theorem~\ref{thm:stability-of-Tbar_c^-}.
	Moreover, assume that $c^*$ is strictly convex
	and there are constants $1<a\le A<\infty$ such that
	$c^*(sx) \le s^{\frac{A}{A-1}}c^*(x)$ for any $x \in \RR^d, s \in [0,1]$ and
	$c^*(sx) \le s^{\frac{a}{a-1}}c^*(x)$ for any $x \in \RR^d, s > 1$.
	
	Let $\{\mu_\theta\}_{\theta \in \Theta}$
	be a family of probability measures on $\RR^d$,
	which satisfy the convex modified log-Sobolev inequality~\eqref{eq:definition-convex-modified-log-Sobolev}
	with uniformly bounded constants:
	\[
	C_\Theta \coloneqq \sup_{\theta\in\Theta} \ConstantConvexModifiedLogSobolev{\mu_\theta}<\infty.
	\]
	Suppose moreover that
	\begin{displaymath}
		D_{\Theta} \coloneqq \sup_{\theta\in\Theta, \theta'\in\Theta} \BarCostGeneral{2,p}{\mu_\theta}{\mu_{\theta'}} < \infty
	\end{displaymath}
and that $\int_{\RR^d} \exp(s\abs{x}) d\mu_{\theta}(x)<\infty$ for all $\theta\in\Theta$.
	Then the mixture $\mu_\mix$, defined as in~\eqref{eq:definition-of-mixture},
	satisfies the convex modified log-Sobolev inequality~\eqref{eq:definition-convex-modified-log-Sobolev}. Moreover,
	$\ConstantConvexModifiedLogSobolev{\mu_\mix}$
	can be bounded in terms of $C_\Theta$,  $D_\Theta$, $a$, $A$, and $p$.
\end{corollary}

In the case of measures on the real line
we are able to obtain the following result for general costs
using a characterization of Gozlan, Roberto, Samson, Shu, and Tetali~\cite{grsst-characterization}
(see Theorems~\ref{thm:GRSST} and Proposition~\ref{prop:GRSST-rephrased} below).
Note that, for a wide class of costs,
for measure on the real line the weak transportation inequalities
$\bfTbar_c$ and $\bfTbar_c^-$ are equivalent~\cite{shu-strzelecki};
moreover, in general the weak transportation inequality $\bfTbar_c^-$
is equivalent to an appropriate convex (modified) log-Sobolev inequality
(see, e.g., Theorem~\ref{thm:equivalence} below).
Thus, similarly as above, the next theorem also immediately implies
a corollary about stability under mixtures of modified convex log-Sobolev inequalities on the real line,
which however we will not state explicitly.

\begin{theorem}
\label{thm:stability-of-Tbar_c-real-line}
 Let $c\colon [0,\infty) \to [0,\infty)$ be a convex cost function such that $c(t) = t^2$ for $t \in[0,t_0]$ for some $t_0>0$.

Let $\{\mu_\theta\}_{\theta \in \Theta}$
be a family of probability measures on $\RR$,
satisfying the inequality $\bfTbar_{c(a|\cdot|)}$ with constant 1 for some $a > 0$.
Suppose moreover that
\begin{displaymath}
E_{\Theta} \coloneqq \sup_{\theta,\theta' \in \Theta}
\abs[\Big]{\int_\RR x \mu_{\theta}(dx) - \int_\RR x \mu_{\theta'}(dx)}  < \infty.
\end{displaymath}
Then the mixture $\mu_\mix$, defined as in~\eqref{eq:definition-of-mixture} (with $d=1$),
satisfies the inequality $\bfTbar_{c(a'|\cdot|)}$ with constant 1 for
$a' = \kappa\cdot (1/a +  E_\Theta)^{-1}$, where $\kappa > 0$ depends only on the cost $c$.
\end{theorem}

\begin{remark}
We stress that the definition of the constant $E_{\Theta}$ in the above theorem involves only distances between the expected values
of the mixed components. It turns out that also in Theorem \ref{thm:stability-of-Tbar_2^-} and the more general Theorem \ref{thm:stability-of-Tbar_c^-} one can replace the constants $D_\Theta$ by
\begin{displaymath}
  E_\Theta = \sup_{\theta,\theta' \in \Theta} \abs[\Big]{\int_{\RR^d} x \mu_{\theta}(dx) - \int_{\RR^d} x \mu_{\theta'}(dx)},
\end{displaymath}
where (recall) $\abs{\cdot}$ is now the standard Euclidean norm on $\RR^d$. However, for general dimension $d$, this happens at the cost of making the constant in the transport inequality for the measure $\mu_\mix$ dependent on $d$.

Let us explain how to bound $D_\Theta$ in terms of $E_\Theta$, $d$ and the constants $K, C_\Theta,\varepsilon$ appearing in Theorem \ref{thm:stability-of-Tbar_c^-}
(for simplicity we will not track the dependence on the latter three constants).
Since the weak transportation inequality $\bfTbar^-_c$ implies the Poincar\'e inequality (see Corollary \ref{cor:weak-transport-implies-convex-Poincare} below),
if all the measures $\mu_\theta$ satisfy $\bfTbar^-_c$ with the same constant,
then there exists a constant $c>0$, such that for every $\theta \in \Theta$, the random vector $X = (X_1,\ldots,X_d)$ with law $\mu_\theta$ satisfies
\begin{displaymath}
  \PP( \abs{X_i - \EE X_i} \ge t) \le 2\exp(-c t).
\end{displaymath}
As a consequence, for some constant $M$ and all $p\ge 2$, $\norm{X_i - \EE X_i}_p \le Mp$, which by the triangle inequality in $L_{p/2}$ implies that
\begin{align*}
  \norm{X - \EE X}_p &= \Bigl(\EE \Bigl(\sum_{i=1}^d \abs{X_i - \EE X_i}^2\Bigr)^{p/2}\Bigr)^{1/p}\\
  &\le  \Bigl(\sum_{i=1}^d \norm{X_i-\EE X_i}_p^2\Bigr)^{1/2}\le Mp\sqrt{d}.
\end{align*}
Thus, for any coupling $(X,Y)$ of the measures $\mu_\theta, \mu_\theta'$, we have
\begin{align*}
  \norm{ X - \conditionalEE{Y}{X}}_p &\le \norm{X - Y}_p \\
  &\le \norm{X - \EE X}_p + \abs{\EE X - \EE Y} + \norm{Y - \EE Y}_p
  \le 2Mp\sqrt{d} + E_\Theta,
\end{align*}
which in particular implies that $D_\Theta \le C_p(M^2d +M^p d^{p/2} + E_\Theta^2 + E_\Theta^p)$. This allows us to apply Theorems \ref{thm:stability-of-Tbar_2^-} or \ref{thm:stability-of-Tbar_c^-}.

Note that above we have in fact estimated $\norm{ X - \conditionalEE{Y}{X}}_p$ for any coupling and not just for the optimal one, so one may hope that in fact a better estimate on $D_\Theta$ holds.
However, an easy example shows that a term diverging with $d$ cannot be avoided.
Indeed, consider $\Theta = \{0,1\}$, and let $\mu_0$ be the Dirac delta at $0 \in \RR^d$ and $\mu_1$ be the uniform distribution on the discrete cube $\{-1,1\}^d$.
Then both $\mu_0$ and $\mu_1$ satisfy the inequality $\bfTbar_2^-$ with universal constants, and $E_\Theta = 0$, but of course $D_\Theta = d/2$. Moreover, for a random vector $X$ distributed according to the measure $\mu_\mix = \frac{1}{2}\mu_0 + \frac{1}{2}\mu_1$, we have
\begin{displaymath}
\EE \abs{X} \le (\EE \abs{X}^2)^{1/2} \le \sqrt{d/2},
\end{displaymath}
and so $\PP( \abs{X} \ge \EE \abs{X} + 0.2\sqrt{d}) \ge \PP(\abs{X} = \sqrt{d}) = 1/2$. Thus, the convex Poincar\'e constant of $\mu_\mix$ is of the order $d$.
In particular (again by Corollary \ref{cor:weak-transport-implies-convex-Poincare}), $\mu_\mix$ cannot satisfy the inequality $\bfTbar_2^-$ with a constant of order smaller than $d$ as $d\to \infty$.
\end{remark}

Let us close this subsection with the comment that
we do not know whether the inequality $\bfTbar_2^+$ is stable under mixtures.
Although this inequality is equivalent to a (modified) log-Sobolev inequality
for a certain subclass of \emph{concave} functions~\cite[Theorem~8.15]{grst-Kantorovich-duality},
the proofs cannot be easily adapted, since if $f$ is a concave function,
then $e^f$ is generally neither concave nor convex.
In our setting,
even under the assumptions that $\chi^2$-distances between the mixed components are uniformly bounded,
problems arise in the proofs,
since one is unable to estimate, e.g., $\Var_{\mu_{\mix}}(e^f)$
(cf.\ the proof of Theorem~\ref{thm:stability-of-covex-log-Sobolev-under-chi-square}
on page~\pageref{thm:stability-of-covex-log-Sobolev-under-chi-square}).
Similar difficulties with $\bfTbar^+$ (for the quadratic-linear cost)
and certain modified log-Sobolev inequalities for concave functions
arose in~\cite{adamczak-strzelecki}.
For a discussion of open questions connected to these issues
see~\cite[Section 5.4]{shu-strzelecki}.

%
%

\subsection{Stability of \texorpdfstring{$\bfT_1$}{T1} and related inequalities}
\label{sec:detailed-introduction-T_1}

We continue with results for the inequality $\bfT_1$.
In this part of the article we shall work in a more general setting,
with Borel probability measures on some Polish space $(E,d)$.
We remark that some parts of the arguments are true
on separable metric spaces (not necessarily complete)
or when $d\colon E\times E\to [0,\infty)$ is some lower semi-continuous metric on $E$,
not necessarily the one generating the topology on $E$.

	For a family $\{\mu_\theta\}_{\theta\in\Theta}$
	of Borel probability measures on $E$, their mixture $\mu_{\mix}$ is defined basically in the same way as in~\eqref{eq:definition-of-mixture},  by the formula
	\begin{equation}
		\label{eq:definition-of-mixture-general}
		\mu_{\mix}(A) \coloneqq \int_\Theta \mu_\theta(A) dm(\theta), \quad A\in\Borel{E}
	\end{equation}
	(under analogous assumptions on $\Theta$, the measure $m$, and measurability as in~\eqref{eq:definition-of-mixture}).
The set of all (Borel) probability
measures on $E$ is denoted
$\Probab{E}$,
while $\ProbabOne{E}$
is the class of all probability measures on $\mu$ on $E$
such that $\int_{E} d(x, x_0) \mu(dx) < \infty$ for some (equivalently: for all) $x_0\in E$.

Recall that for $p\in[1,\infty)$
the $L^p$-Wasserstein distance (or Kantorovich--Rubinstein metric)
between two probability measures $\mu$, $\nu$ on $E$
is defined as
$W_p(\mu, \nu) \coloneqq (\calT_{d^p})^{1/p}$,
where here
\[
\calT_{d^p}(\mu, \nu) \coloneqq
\inf_\pi \Bigl\{ \int_{E\times E} d(x,y)^p d\pi(x,y) \Bigr\},
\]
with the infimum taken over all couplings $\pi$ of $\mu$ and $\nu$.

We say that a probability measure $\mu$ on $E$
satisfies the transport--entropy inequality $\bfT_1$
with constant $C\in(0,\infty)$
if for all $\nu\in\Probab{E}$,
\begin{equation}
	\label{eq:definition-T_1-inequality}
	\calT_{d}(\mu,\nu) \leq \sqrt{2C \RelativeEntropy{\nu}{\mu}},
\end{equation}
where the relative entropy $\RelativeEntropy{\nu}{\mu}$ is defined similarly as in \eqref{eq:definition-relative-entropy}.
We denote the smallest constant $C$ in~\eqref{eq:definition-T_1-inequality}
by $\ConstantTransportTOne{\mu}$.

We start with the following observation
(note that $B_\theta(x_0) <\infty$,
because $\bfT_1$ implies integrability of Lipschitz functions,
cf.\ Section~\ref{sec:proofs-T_1}).

\begin{theorem}[main result: stability of $\bfT_1$ under mixtures]
\label{thm:stability-of-T_1-improved}
Let $\{\mu_\theta\}_{\theta \in \Theta}$ be a family of  probability measures on $E$
which satisfy the $\bfT_1$ inequality
 with uniformly bounded constants:
\[
C_\Theta \coloneqq \sup_{\theta\in\Theta} \ConstantTransportTOne{\mu_\theta} < \infty.
\]
Take any $x_0\in E$, denote
\[
B_{\theta}(x_0) \coloneqq \int_E d(x,x_0) d\mu_\theta(x), \quad \theta\in\Theta,
\]
and suppose that
\[
N_\Theta \coloneqq
\inf\Bigl\{ t>0 : \int_{\Theta} \exp\bigl( B_{\theta}(x_0)^2/t^2\bigr) dm(\theta) \leq 2 \Bigr\} < \infty.
\]
Then the mixture $\mu_\mix$,
defined as in~\eqref{eq:definition-of-mixture-general},
satisfies the $\bfT_1$ inequality
and
\[
\ConstantTransportTOne{\mu_\mix}
\leq
 \bigl( N_\Theta^2 + 4C_\Theta\bigr)\bigl(1 + \frac{3}{2}\log2 \bigr).
\]
\end{theorem}

Let us comment that
$N_\Theta$ is the $\psi_2$ Orlicz norm of the function $\theta\mapsto B_\theta(x_0)$
and that the assumption $N_\Theta<\infty$ is necessary
for the mixture $\mu_\mix$ to satisfy the $\bfT_1$ inequality,
see Remark~\ref{rem:on-assumptions-in-T1-result} below.

\begin{remark}
The proof of Theorem~\ref{thm:stability-of-T_1-improved} yields
a slightly more general estimate of the constant $\ConstantTransportTOne{\mu_\mix}$,
which for given values of $N_\Theta$, $C_\Theta$ may lead to slightly better numerical coefficients.
Moreover,  if it so happens that
\[
B_\Theta\coloneqq \sup_{\theta\in\Theta} B_{\theta}(x_0) < \infty,
\]
then $N_\Theta^2 \leq B_\Theta^2/\log 2$, so we obtain
\[
\ConstantTransportTOne{\mu_\mix}
\leq
\bigl( \frac{1}{\log2} B_\Theta^2 + C_\Theta\bigr)\bigl(1 + \frac{3}{2}\log2 \bigr).
\]
\end{remark}

One can consider also more general inequalities.
Suppose that $\alpha\colon[0,\infty)\to[0,\infty)$
is an increasing function with $\alpha(0) = 0$.
We say that a probability measure $\mu$ on $E$
satisfies the transport--entropy inequality
$\bfT_{1, \alpha}$
if for all $\nu\in\Probab{E}$,
\begin{equation}
\label{eq:definition-alpha-T_1-inequality}
\alpha\bigl(\calT_{d}(\mu,\nu)\bigr) \leq  \RelativeEntropy{\nu}{\mu}.
\end{equation}
Note that $\alpha(t) = t^2/(2C)$ corresponds to the $\bfT_1$ inequality.

We remark that one can consider even more general inequalities,
of the form
\begin{equation}
\label{eq:definition-general-alpha-T-inequality}
\alpha\bigl(\calT_{q(d)}(\mu,\nu)\bigr) \leq  \RelativeEntropy{\nu}{\mu},
\end{equation}
where $\alpha$ is as above,
and $q\colon[0,\infty) \to [0,\infty)$ is convex increasing
and satisfies some growth condition,
see \cite{gozlan-integral-criteria} and \cite[Chapitre VII]{gozlan-phd-thesis}.
Choosing $\alpha(t) = t/(2C)$  and $q(t) = t^2$
would correspond to the $\bfT_2$ inequality,
but in the sequel we will need additional assumptions on the function $\alpha$
(which preclude taking linear $\alpha$).

Assume that $\alpha\colon[0,\infty)\to[0,\infty)$
is convex increasing and $\alpha(0) = 0$.
For $t\geq 0$ define
\[
\alpha^*(t) \coloneqq \sup_{s\geq 0} \{ st - \alpha(s)\}
\]
(we use the same notation as in~\eqref{eq:definition-Legendre-transform},
but here the functions are defined on the half-line).
We shall assume that the effective domain of $\alpha^*$ is open on the right, i.e.,
\begin{equation}
\label{eq:gozlan-assumption-a1}
\{t\in[0,\infty) : \alpha^*(t) < \infty\} = [0,b_0)
\end{equation}
for some $b_0 \in (0,\infty]$,
and that $\alpha^*$ is super-quadratic near $0$, i.e.,
there exist $t_0 >0$
and $c_0 >0$ such that
\begin{equation}
\label{eq:gozlan-assumption-a2}
\alpha^*(t) \geq c_0 t^2 \quad \text{ for } t\in [0,t_0].
\end{equation}

\begin{theorem}[stability of $\bfT_{1,\alpha}$ under mixtures]
\label{thm:stability-of-T_1_alpha}
Let $\alpha\colon[0,\infty)\to[0,\infty)$ be an increasing convex function
satisfying the assumptions \eqref{eq:gozlan-assumption-a1} and \eqref{eq:gozlan-assumption-a2}.

Let $\{\mu_\theta\}_{\theta \in \Theta}$ be a family of  probability measures on $E$
which satisfy the inequality
$\bfT_{1, \alpha}$.
Take any $x_0\in E$, denote
\[
B_{\theta}(x_0) \coloneqq \int_E d(x,x_0) d\mu_\theta(x), \quad \theta\in\Theta,
\]
and suppose that there exists
$\lambda>0$ and $\delta \in (0,1/(\lambda+1))$
such that
\[
I_\Theta(x_0,\delta,\lambda)
\coloneqq
\int_{\Theta} \exp\bigl(\delta\lambda \alpha( B_{\theta}(x_0)/\lambda) \bigr) dm(\theta) < \infty.
\]
Then the mixture $\mu_\mix$,
defined as in~\eqref{eq:definition-of-mixture-general},
satisfies the $\bfT_{1,\alpha(\cdot/a)}$ inequality
with
\begin{equation*}
	a = \frac{2\sqrt{2} m_\alpha}{\delta}
	\Bigl( 1 + \frac{1}{\log 2} \log\Bigl( \frac{1+\delta}{1-\delta} I_\Theta(x_0,\delta,\lambda)
	\Bigr)\Bigr),
\end{equation*}
where  $m_\alpha$ is a constant which depends only on the function $\alpha$
(i.e., on the constants $c_0$, $t_0$ in~\eqref{eq:gozlan-assumption-a2}).
\end{theorem}

We note that Ma, Shen, Wang and Wu~\cite[Proposition~3.2]{MR2797986}
also proved a result concerning stability under mixtures of the $\bfT_{1,\alpha}$ inequality,
but in their result they assume that $\Theta$ is a metric space
and that the mixing measure $m$ also satisfies a transportation inequality.

%
%

\section{Proof of stability of \texorpdfstring{$\bfT_2$}{T2} and related transportation inequalities}

\label{sec:proofs-T_2}

\subsection{Preliminaries: restricted log-Sobolev inequalities}

Let us recall a characterization of the $\bfT_c$ inequality
due to Gozlan, Roberto, and Samson~\cite{gozlan-roberto-samson-new-characterization}.
Consider a cost function $c\colon\RR^d\to[0,\infty)$
of the form
\begin{equation}
\label{eq:special-form-of-the-cost-function-dulicate}
c(x) = \sum_{i=1}^d \alpha(x_i), \quad  x=(x_1,\dots,x_d)\in\RR^d,
\end{equation}
for some symmetric convex function $\alpha\colon\RR\to[0,\infty)$  of class $C^1$,
such that $\alpha(0) = \alpha'(0) = 0$ and $\alpha'$ is concave on $[0,\infty)$,
and $\alpha(t) = \frac{1}{2}t^2$ for $t\in[0,t_0]$ for some $t_0>0$.

Let $K\in\RR$.
A function $f\colon\RR^d\to\RR$ is called
$K$-semi-convex for the cost function $c$
if for all $\lambda\in[0,1]$
and all $x,y\in\RR^d$,
\begin{align*}
\label{eq:definition-K-semin-convex}
f(\lambda x + (1-\lambda)y)
&\leq \lambda f(x) + (1-\lambda)f(y)\\
&\quad
+ \lambda K c((1-\lambda)(y-x))
+ (1-\lambda) K c(\lambda(y-x)).
\end{align*}
For differentiable functions this is actually equivalent to
\[
f(y) \geq f(x) + \nabla f(x) \cdot(y-x) - Kc(y-x)
\]
for all $x,y\in\RR^d$, see~\cite[Proposition~5.1]{gozlan-roberto-samson-new-characterization}.
In the special case of the quadratic cost $c(x) = \abs{x}^2/2$,
the function $f$ is $K$-semi-convex
if and only if
$x\mapsto f(x) + K\abs{x}^2/2$ is convex.
We remark that in the literature there are also other definitions of $K$-semi-convexity, which are not equivalent to the one above. They will not be used in this article.

We say that
a probability measure $\mu$ on $\RR^d$
satisfies the restricted modified log-Sobolev inequality
(for the cost $c$)
with constant $C\in(0,\infty)$
if for all $K\geq 0$, $\eta >0$
with $\eta + K <1/C$
and for all differentiable $K$-semi-convex (for the cost $c$) functions $f\colon\RR^d\to\RR$,
\begin{equation}
\label{eq:definition-restricted-modified-log-Sobolev}
\Ent_\mu(e^f)
\leq
\frac{\eta}{1- (\eta + K) C}
\int_{\RR^d} c^*\Bigl( \frac{\nabla f}{\eta}\Bigr) e^f d\mu.
\end{equation}

In the quadratic case,
when $c(x) = c^*(x) = \abs{x}^2/2$,
one can optimize over $\eta$ and it is easy to see that
for fixed $K\in[0,1/C)$ the value $\eta = (1-KC)/(2C)$ is optimal.
Therefore, we say that
a probability measure $\mu$ on $\RR^d$
satisfies the restricted log-Sobolev inequality
(for the quadratic cost $c(x) = \abs{x}^2/2$)
with constant $C\in(0,\infty)$
if for all $K\in [0,1/C)$
and for all differentiable $K$-semi-convex (for the quadratic cost) functions $f\colon\RR^d\to\RR$,
\begin{equation}
\label{eq:definition-restricted-log-Sobolev}
\Ent_\mu(e^f)
\leq
\frac{2C}{(1-KC)^2}
\int_{\RR^d} |\nabla f|^2 e^f d\mu.
\end{equation}

Our main tool for proving stability of the inequalities $\bfT_c$
is their characterization in terms of restricted log-Sobolev inequalities, due to Gozlan, Roberto, and Samson.
We remark that while the proof in~\cite{gozlan-roberto-samson-new-characterization}
is presented under some regularity assumptions,
standard smoothing arguments ensure that the results hold in whole generality
and with no changes in the constants.

\begin{theorem}[{{\cite[Theorem~1.5]{gozlan-roberto-samson-new-characterization}}}]
\label{thm:a-new-characterization-of-T_c}
Let $\mu$ be a probability measure on $\RR^d$
and let $c\colon\RR^d\to[0,\infty)$
be a cost function as in~\eqref{eq:special-form-of-the-cost-function-dulicate}

If $\mu$ satisfies the transport--entropy inequality $\bfT_c$~\eqref{eq:definition-T_c-inequality}
with constant $C_1$,
then $\mu$ satisfies the restricted modified log-Sobolev inequality~\eqref{eq:definition-restricted-modified-log-Sobolev} with constant $C_1$.

If $\mu$ satisfies the restricted modified log-Sobolev inequality~\eqref{eq:definition-restricted-modified-log-Sobolev} with constant $C_2$
then $\mu$ satisfies the transport--entropy inequality $\bfT_c$~\eqref{eq:definition-T_c-inequality}
with constant $8C_2$.
\end{theorem}

We will also  need the following auxiliary lemma about the Legendre transform $\alpha^*(t) = \sup_{s\in\RR} \{ st - \alpha(s)\}$.

\begin{lemma}[{{cf.\ \cite[Lemma~5.4]{gozlan-roberto-samson-new-characterization}}}]
\label{lem:convex-conjugate-elementary-lemma}
Let $\alpha\colon\RR\to[0,\infty)$ be a symmetric convex function  of class $C^1$,
such that $\alpha'$ is concave on $[0,\infty)$.
Then, for $t\geq 0$  and $\delta\in (0,1)$ we have
$\delta^2 \alpha(t) \leq \alpha(\delta t)$
and $\alpha^*(\delta t) \leq \delta^2 \alpha^*(t)$;
in particular, $\alpha(t) \leq 4\alpha(t/2)$.
If moreover $\alpha(t) = \frac{1}{2}t^2$ for $t\in[0,t_0]$ for some $t_0>0$,
then $\alpha(t) \leq \frac{1}{2} t^2 \leq \alpha^*(t)$ for $t\geq 0$.
\end{lemma}

\begin{proof}
Take $u\in(0,1)$.
By concavity of $\alpha'$,
\[ \alpha'(s) \geq u \alpha'(s/u) + (1-u) \alpha'(0) = u\alpha'(s/u). \]
Integrating over $s\in[0,s_0]$ yields
$\alpha(s_0) \geq u^2 \alpha(s_0/u)$.
Taking $s_0=\delta t$, $u=\delta\in(0,1)$ yields the first claim;
the second one follows by scaling properties of the convex conjugate.

Moreover, if $\alpha(t) = \frac{1}{2} t^2$ for $t\in[0,t_0]$,
then choosing $s_0 =  t_0$, $u=t_0/t$, $t\geq t_0$,
yields that $\alpha(t) \leq \frac{1}{2}t^2$ for $t\geq t_0$.
Since the convex conjugation is order-reversing, this completes the proof of the last claim.
\end{proof}

\subsection{Lemmas}

We start with bounds on the entropy with respect to the mixture.
The next result is an enhanced version of \cite[Lemma~1]{chen-et-al}.

\begin{lemma}
\label{lem:entropy-donsker-varadhan-type-estimate}
Let $\nu$ and $\rho$ be two probability measures on $\RR^d$
such that $\nu$ is absolutely continuous with respect to $\rho$.
Let $f\colon\RR^d\to[0,\infty)$ be a non-negative function
such that $\EE_\nu f \in (0,\infty)$.
Then
\begin{align}
\label{eq:1st-estimate}
(\EE_\nu f - \EE_\rho f)_+ \log\Bigl(\frac{\EE_\nu f}{\EE_\rho f}\Bigr)
\leq \Ent_\nu(f) + (\EE_\nu f - \EE_\rho f)_+ \log\bigl(1+\chisquare{\nu}{\rho}\bigr).
\end{align}
\end{lemma}

\begin{proof}
Note that the assumptions imply that $\EE_\rho f >0$
(since otherwise we would have $f=0$ almost surely with respect to $\rho$, so also almost surely with respect to $\nu$).
Assume first that $\EE_{\nu} f = 1$.
Chen, Chewi, and Niles-Weed \cite[Lemma~1]{chen-et-al}
proved that
\begin{displaymath}
\log\Bigl(\frac{\EE_\nu f}{\EE_\rho f}\Bigr)
- \log\bigl(1+\chisquare{\nu}{\rho}\bigr) \le \Ent_\nu(f).
\end{displaymath}
Multiplying both sides by $(\EE_\nu f - \EE_\rho f)_+ \ge 0$, we get
\begin{displaymath}
(\EE_\nu f - \EE_\rho f)_+ \Bigl(\log\Bigl(\frac{\EE_\nu f}{\EE_\rho f}\Bigr) - \log\bigl(1+\chisquare{\nu}{\rho}\bigr)\Bigr) \le (\EE_\nu f - \EE_\rho f)_+\Ent_\nu(f) \le \Ent_\nu(f),
\end{displaymath}
where in the last inequality we used that $(\EE_\nu f - \EE_\rho f)_+ \le \EE_\nu f = 1$ and $\Ent_\nu(f) \ge 0$.
Thus,
\begin{align*}
(\EE_\nu f - \EE_\rho f)_+ \log\Bigl(\frac{\EE_\nu f}{\EE_\rho f}\Bigr)
\leq \Ent_\nu(f) + (\EE_\nu f - \EE_\rho f)_+ \log\bigl(1+\chisquare{\nu}{\rho}\bigr).
\end{align*}
This inequality is homogeneous in $f$, so we do not need to assume that $\EE_\nu f = 1$, it holds for all nonnegative $\nu$-integrable functions.
\end{proof}

\begin{lemma}
\label{lem:entropy-of-mixture-chisquared-estimate}
Let
$\{\mu_\theta\}_{\theta\in\Theta}$
be a family of probability measures on $\RR^d$
such that
\[
M_{\Theta}\coloneqq
\sup_{\theta,\theta' \in \Theta} \chisquare{\mu_\theta}{\mu_{\theta'}} < \infty.
\]
Let  $\mu_\mix$ be the mixture defined as in~\eqref{eq:definition-of-mixture}
and let
$f\colon\RR^d\to\RR$ be a fixed function.
Then
\begin{equation}\label{eq:Lemma-4.5-assertion}
\Ent_{\mu_\mix}(f^2)
\leq
2\int_\Theta \Ent_{\mu_\theta}(f^2)dm(\theta)
+ \log(1+M_{\Theta}) \Var_{\mu_\mix}(f).
\end{equation}	
\end{lemma}

\begin{proof}
Denote  $M \coloneqq M_\Theta$.
The assumption $M<\infty$ implies that all the mixed components have densities with respect to each other, so we can assume that $\EE_{\mu_\theta} f^2 > 0$ for every $\theta\in\Theta$
(since otherwise $\EE_{\mu_\theta} f^2 = 0$  for every $\theta\in\Theta$
and the assertion holds trivially). Moreover, if the right-hand side of \eqref{eq:Lemma-4.5-assertion} is finite,
then $\EE_{\mu_\theta} f^2 < \infty$ $m$-a.s.
By the decomposition~\eqref{eq:entropy-of-mixture},
\begin{align*}
\Ent_{\mu_\mix} (f^2)
& = \int_\Theta \Ent_{\mu_\theta} (f^2) dm(\theta)
+ \Ent_m (\theta\mapsto\EE_{\mu_\theta}f^2),
\end{align*}
where the second summand is harder to deal with.
By Jensen's inequality we have
\begin{align*}
\MoveEqLeft[2]
\Ent_m (\theta\mapsto\EE_{\mu_\theta}f^2) \\
&\leq\frac{1}{2}
\int_\Theta\int_\Theta (\EE_{\mu_\theta} f^2 - \EE_{\mu_{\theta'}}f^2)(\log \EE_{\mu_\theta} f^2 - \log \EE_{\mu_{\theta'}} f^2)dm(\theta)dm(\theta')\\
&=
\int_\Theta\int_\Theta (\EE_{\mu_\theta} f^2 - \EE_{\mu_{\theta'}}f^2)_+(\log \EE_{\mu_\theta} f^2 - \log \EE_{\mu_{\theta'}} f^2)dm(\theta)dm(\theta')\\
&\leq
\int_\Theta\int_\Theta \Bigl(\Ent_{\mu_\theta} (f^2) + (\EE_{\mu_\theta} f^2 - \EE_{\mu_{\theta'}}f^2)_+ \log(1+M)\Bigr)dm(\theta)dm(\theta')\\
&= \int_\Theta \Ent_{\mu_\theta} (f^2) dm(\theta) + \log(1+M)\int_\Theta\int_\Theta (\EE_{\mu_\theta} f^2 - \EE_{\mu_{\theta'}}f^2)_+ dm(\theta) dm(\theta'),
\end{align*}
where in the second inequality
we used \eqref{eq:1st-estimate} for $f \coloneqq f^2,
\nu\coloneqq\mu_\theta, \rho \coloneqq \mu_{\theta'}$.
The second summand can be estimated as follows:
\begin{align*}
\MoveEqLeft[2]
\int_\Theta\int_\Theta (\EE_{\mu_\theta} f^2 - \EE_{\mu_{\theta'}}f^2)_+ dm(\theta) dm(\theta')\\
& = \frac{1}{2}\int_\Theta\int_\Theta |\EE_{\mu_\theta} f^2 - \EE_{\mu_{\theta'}}f^2| dm(\theta) dm(\theta')\\
&= \frac{1}{2}\int_\Theta\int_\Theta
\Bigl\lvert
\bigl(\EE_{\mu_\theta} f^2 - (\EE_{\mu_\theta} f)^2\bigr)
+\bigl((\EE_{\mu_\theta} f)^2 - (\EE_{\mu_\mix} f)^2\bigr) \\
&\qquad\qquad   + \bigl((\EE_{\mu_\mix} f)^2 - (\EE_{\mu_{\theta'}} f)^2\bigr)
+ \bigl((\EE_{\mu_{\theta'}}f)^2 - \EE_{\mu_{\theta'}}f^2\bigr)
\Bigr\rvert dm(\theta) dm(\theta')\\
&\leq \frac{1}{2} \int_\Theta \bigl(\EE_{\mu_\theta} f^2 - (\EE_{\mu_\theta} f)^2\bigr)dm(\theta) + \frac{1}{2}\int_\Theta \bigl((\EE_{\mu_\theta} f)^2 - (\EE_{\mu_\mix} f)^2\bigr)dm(\theta) \\
&\qquad + \frac{1}{2} \int_\Theta \bigl((\EE_{\mu_{\theta'}} f)^2 - (\EE_{\mu_\mix} f)^2\bigr)dm(\theta')
+ \frac{1}{2} \int_\Theta  \bigl(\EE_{\mu_{\theta'}} f^2 - (\EE_{\mu_{\theta'}} f)^2\bigr)dm(\theta')\\
=& \int_\Theta \Var_{\mu_\theta}(f)dm(\theta) + \Var_m (\theta\mapsto\EE_{\mu_\theta} f) = \Var_{\mu_\mix}(f),
\end{align*}
where the last equality follows from~\eqref{eq:variance-of-mixture}.
Combining all inequalities we arrive at the assertion.
\end{proof}

We also observe that by minor changes in the proof of \cite[Theorem~1 (1)]{chen-et-al}
one gets the following improvement of~\eqref{eq:chen-et-al-Poincare-main-result}.

\begin{proposition}[numerical constants in~\eqref{eq:chen-et-al-Poincare-main-result}]
\label{prop:stability-of-Poincare-with-improved-constants}
Let
$\{\mu_\theta\}_{\theta\in\Theta}$
be a family of probability measures on $\RR^d$
which satisfy the Poincar\'e inequality with uniformly bounded constants:
\[
C_\Theta \coloneqq \sup_{\theta\in\Theta} \ConstantPoincare{\mu_\theta} < \infty.
\]
Suppose moreover that
\[
M_{\Theta}\coloneqq \sup_{\theta,\theta' \in \Theta} \chisquare{\mu_\theta}{\mu_{\theta'}} < \infty.
\]
Then the mixture $\mu_\mix$,
defined as in~\eqref{eq:definition-of-mixture},
satisfies the Poincar\'e inequality
with
\[
\ConstantPoincare{\mu_\mix} \leq
C_\Theta
  \cdot \Bigl(1  + \frac{1}{2} M_\Theta\Bigr).
\]
\end{proposition}

\begin{proof}
Denote $C\coloneqq C_\Theta$ and $ M\coloneqq M_\Theta $.
Fix a function $f\colon\RR^d\to\RR$.
By the decomposition~\eqref{eq:variance-of-mixture},
\begin{align}
\label{eq:vdfp-chi-square}
\MoveEqLeft[1]
\Var_{\mu_{\mix}} (f)
= \int_{\Theta} \Var_{\mu_\theta} (f)  dm(\theta)  + \Var_m\Bigl(\theta \mapsto \int_{\RR^d} f(x) d\mu_\theta(x) \Bigr)\\
&= \int_\Theta \Var_{\mu_\theta} (f)  dm(\theta)
+ \frac{1}{2} \int_\Theta\int_\Theta  \Bigl( \int_{\RR^d} f d\mu_{\theta_1} - \int_{\RR^d}f d\mu_{\theta_2} \Bigr)^2 dm(\theta_1) dm(\theta_2) \nonumber
\end{align}
and the first summand can be estimated by
\[
C \int_{\RR^d} |\nabla f|^2 d\mu_{\mix}.
\]
In order to deal with the second term observe that
\begin{align*}
\Bigl( \int_{\RR^d} f d\mu_{\theta_1} - \int_{\RR^d}f d\mu_{\theta_2} \Bigr)^2
&= \Bigl( \int_{\RR^d} f \bigl( 1- \frac{d\mu_{\theta_2}}{d\mu_{\theta_1}} \bigr)d\mu_{\theta_1}\Bigr)^2\\
&= \Cov_{\mu_{\theta_1}}\Bigl(f, 1- \frac{d\mu_{\theta_2}}{d\mu_{\theta_1}} \Bigr)^2\\
&\leq \Var_{\mu_{\theta_1}}(f)
\Var_{\mu_{\theta_1}}\Bigl(1- \frac{d\mu_{\theta_2}}{d\mu_{\theta_1}}\Bigr)\\
&= \Var_{\mu_{\theta_1}}(f) \chisquare{\mu_{\theta_2}}{\mu_{\theta_1}}\\
&\leq CM \int_{\RR^d} \abs{\nabla f}^2 d\mu_{\theta_1}.
\end{align*}
After multiplying by $1/2$, integrating over $\theta_1\in\Theta$, $\theta_2\in \Theta$,
and recalling~\eqref{eq:vdfp-chi-square}, we arrive at the assertion of the proposition.
\end{proof}

\subsection{Proof of Proposition~\ref{prop:stability-of-log-Sobolev-with-improved-constants} and Theorems~\ref{thm:stability-of-T_2} and~\ref{thm:stability-of-T_c}}

We start with the proof of Proposition~\ref{prop:stability-of-log-Sobolev-with-improved-constants}
as it is the simplest one.

\begin{proof}[Proof of Proposition~\ref{prop:stability-of-log-Sobolev-with-improved-constants}]
Denote $C\coloneqq C_\Theta$ and $M \coloneqq M_\Theta$.
By assumption, the mixed component $\mu_\theta$
satisfies the log-Sobolev inequality with constant at most $C$,
so it also satisfies the Poincar\'e inequality with constant $C$.
Hence, by Proposition~\ref{prop:stability-of-Poincare-with-improved-constants},
the mixture $\mu_\mix$ satisfies the Poincar\'e inequality with constant $(1+M/2)C$.

Thus, using Lemma~\ref{lem:entropy-of-mixture-chisquared-estimate} in the first inequality
and then the log-Sobolev inequality for the mixed components $\mu_{\theta}$
and the Poincar\'e inequality for the mixture $\mu_{\mix}$,
we can write
\begin{align*}
\Ent_{\mu_\mix}(f^2)
&\leq
2\int_\Theta \Ent_{\mu_\theta}(f^2)dm(\theta)
+ \log(1+M) \Var_{\mu_\mix}(f)\\
&\leq
4C \int_\Theta \int_{\RR^d} \abs{\nabla f}^2 d\mu_\theta dm(\theta)
+ \bigl(1+\frac{1}{2} M\bigr)\log(1+M)C  \int_{\RR^d} \abs{\nabla f}^2 d\mu_\mix\\
&= 2\cdot \Bigl(2 + \frac{1}{4}(2+M)\log(1+M)\Bigr) C \int_{\RR^d} \abs{\nabla f}^2 d\mu_\mix.
\end{align*}	
This completes the proof.
\end{proof}

In order to prove the stability under mixtures of the $\bfT_2$ inequality,
we follow a similar strategy and use the characterization in terms of
the restricted log-Sobolev inequality~\eqref{eq:definition-restricted-log-Sobolev}.

\begin{proof}[Proof of Theorem~\ref{thm:stability-of-T_2}]
Denote $C\coloneqq C_\Theta =  \sup_{\theta\in\Theta} \ConstantTransportTTwo{\mu_\theta}$ and $M \coloneqq M_\Theta$.
By assumption, for every $\theta\in\Theta$
the mixed component $\mu_\theta$
satisfies the $\bfT_2$ inequality with constant at most $C$,
so by \eqref{eq:implication-chain} it also satisfies the Poincar\'e inequality with constant at most $C$.
Hence, by Proposition~\ref{prop:stability-of-Poincare-with-improved-constants},
the mixture $\mu_\mix$ satisfies the Poincar\'e inequality with constant $(1+M/2)C$.
Moreover, by the characterization of Gozlan, Robert, and Samson~\cite{gozlan-roberto-samson-new-characterization}
(see Theorem~\ref{thm:a-new-characterization-of-T_c}),
the mixed component $\mu_\theta$
satisfies the restricted log-Sobolev inequality~\eqref{eq:definition-restricted-log-Sobolev}
with constant $C$.

Take any $K\in[0,1/C)$ and any differentiable
$K$-semi-convex (for the quadratic cost) function $f\colon\RR^d\to\RR$.
Applying Lemma~\ref{lem:entropy-of-mixture-chisquared-estimate} to the function $e^{f/2}$ yields
\begin{displaymath}
\Ent_{\mu_\mix}(e^f)
\leq
2\int_\Theta \Ent_{\mu_\theta}(e^f)dm(\theta)
+ \log(1+M) \Var_{\mu_\mix}(e^{f/2}).
\end{displaymath}
Hence, using the restricted log-Sobolev inequality for the mixed components $\mu_\theta$
as well as the Poincar\'e inequality for the mixture $\mu_{\mix}$, we can write
\begin{align*}
\Ent_{\mu_\mix}(e^f)
&\leq
\frac{4C}{(1-KC)^2} \int_\Theta \int_{\RR^d} \abs{\nabla f}^2 e^f d\mu_\theta dm(\theta)\\
&\quad +\bigl(1 +  \frac{1}{2}M\bigr)\log(1+M) C  \int_{\RR^d} \frac{1}{4}\abs{\nabla f}^2 e^f d\mu_\mix\\
&= \frac{4C + \frac{1}{8}(1-KC)^2(2+M)\log(1+M) C }{(1-KC)^2}
\int_{\RR^d} \abs{\nabla f}^2 e^f d\mu_\mix\\
&\leq\frac{2\cdot C\cdot\bigl(2 + \frac{1}{16}(2+M)\log(1+M)\bigr)}{(1-KC)^2}
\int_{\RR^d} \abs{\nabla f}^2 e^f d\mu_\mix.
\end{align*}
This implies that $\mu_\mix$ satisfies
restricted log-Sobolev inequality~\eqref{eq:definition-restricted-log-Sobolev}
with constant
\[
\widetilde{C} = C\cdot\Bigl(2 + \frac{1}{16}(2+M)\log(1+M)\Bigr);
\]
note that $\widetilde{C}\geq C$,
so if $K\in[0,1/\widetilde{C})$, then $K\in[0,1/C)$
and
\[
\frac{2\widetilde{C}}{(1-KC)^2} \leq \frac{2\widetilde{C}}{(1-K\widetilde{C})^2}.
\]
Thus, again by the characterization from  Theorem~\ref{thm:a-new-characterization-of-T_c},
we conclude that $\mu_\mix$ satisfies the $\bfT_2$ inequality with constant
\[
8\widetilde{C} = C\cdot\bigl(16 + \frac{1}{2} (2+M)\log(1+M)\bigr),
\]
This completes the proof.
\end{proof}

The proof of the stability under mixtures
of the general $\bfT_c$ inequality
is very similar,
but uses restricted \emph{modified} log-Sobolev inequalities
(see~\eqref{eq:definition-restricted-modified-log-Sobolev}).

\begin{proof}[Proof of Theorem~\ref{thm:stability-of-T_c}]
Denote $C\coloneqq C_\Theta$ and $M \coloneqq M_\Theta$.
By assumption,
for every $\theta\in\Theta$ the mixed component $\mu_\theta$
satisfies the $\bfT_c$ inequality with constant at most $C$
(where the cost function $c$ is quadratic near zero:
$c(x) = \abs{x}^2/2$ for $x\in[-t_0,t_0]^d$),
so it also satisfies the Poincar\'e inequality with constant at most~$C$,
see, e.g., \cite[Section~7]{MR1760620} or \cite[Proposition~8.4]{gozlan-leonard-survey}.
Hence, by Proposition~\ref{prop:stability-of-Poincare-with-improved-constants},
the mixture $\mu_\mix$ satisfies the Poincar\'e inequality with constant $(1+M/2)C$.
Moreover, by the characterization of Gozlan, Robert, and Samson~\cite{gozlan-roberto-samson-new-characterization}
(see Theorem~\ref{thm:a-new-characterization-of-T_c}),
the mixed component $\mu_\theta$
satisfies the restricted modified log-Sobolev inequality~\eqref{eq:definition-restricted-modified-log-Sobolev}
with constant $C$.

Take any $K\geq 0$, $\eta>0$ with $\eta + K < 1/(C+B)$,
where $B \geq 0$ is a constant (depending only on $C$ and $M$)
which will be determined later.
Let $f\colon\RR^d\to\RR$ be differentiable
and $K$-semi-convex with respect to the cost $2c(\cdot/2)$;
note that $f$ is also $K$-semi-convex with respect to the larger cost $c(\cdot)$
and that $\eta + K < 1/C$.
Applying Lemma~\ref{lem:entropy-of-mixture-chisquared-estimate} to the function $e^{f/2}$ yields
\begin{displaymath}
\Ent_{\mu_\mix}(e^f)
\leq
2\int_\Theta \Ent_{\mu_\theta}(e^f)dm(\theta)
+ \log(1+M) \Var_{\mu_\mix}(e^{f/2}).
\end{displaymath}
Hence, using the restricted modified log-Sobolev inequality for the mixed components $\mu_\theta$
as well as the Poincar\'e inequality for the mixture $\mu_{\mix}$, we can write
\begin{align*}
\Ent_{\mu_\mix}(e^f)
&\leq
\frac{2\eta}{1-(\eta + K)C} \int_\Theta \int_{\RR^d} c^*\Bigl(
\frac{\nabla f}{\eta}\Bigr) e^f d\mu_\theta dm(\theta)\\
&\quad +\frac{1}{2} (2+M)\log(1+M) C  \int_{\RR^d} \frac{1}{4}\abs{\nabla f}^2 e^f d\mu_\mix\\
&\leq \Bigl(\frac{2\eta}{1-(\eta+K)C} +\frac{\eta^2}{4}(2+M)\log(1+M) C\Bigr)
\int_{\RR^d} c^*\Bigl(
\frac{\nabla f}{\eta}\Bigr) e^f d\mu_\mix\\
&\leq\frac{\eta\cdot\bigl(1 + \frac{\eta}{8}(2+M)\log(1+M) C\Bigr)}{1-(\eta+K)C}
\int_{\RR^d} 2c^*\Bigl(
\frac{\nabla f}{\eta}\Bigr)  e^f d\mu_\mix,
\end{align*}
where in the second inequality we used Lemma~\ref{lem:convex-conjugate-elementary-lemma}.
If $B>0$
and $\eta + K< 1/(C+B)$, then
\[
\frac{\eta (1+\eta B)}{1-(\eta+K)C} \leq
\frac{\eta}{1-(\eta+K)(C+B)},
\]
since
\[
1+\eta B \leq 1 + \frac{(\eta+K)B}{1-(\eta+K)(C+B)} = \frac{1 -(\eta+K) C}{1-(\eta+K)(C+B)}.
\]
Applying this inequality with $B \coloneqq \frac{1}{8}(2+M)\log(1+M) C$, we conclude that
\begin{equation*}
\Ent_{\mu_\mix}(e^f)
\leq\frac{\eta}{1-(\eta+K)\widetilde{C}}
\int_{\RR^d} 2c^*\Bigl(
\frac{\nabla f}{\eta}\Bigr)  e^f d\mu_\mix,
\end{equation*}
where
\[
\widetilde{C} \coloneqq C + B =  C\cdot\Bigl(1 + \frac{1}{8}(2+M)\log(1+M)\Bigr).
\]
This means  that $\mu_\mix$ satisfies
restricted modified log-Sobolev inequality~\eqref{eq:definition-restricted-modified-log-Sobolev}
with constant $\widetilde{C}$
and cost function $2c(\cdot/2)$
(for which the convex conjugate is equal to $(2c(\cdot/2))^* = 2c^*(\cdot)$).

Thus,
again by the characterization from  Theorem~\ref{thm:a-new-characterization-of-T_c},
we conclude that $\mu_\mix$ satisfies the transportation inequality with constant
\[
8\widetilde{C} = C\cdot\bigl(8 + (2+M)\log(1+M)\bigr)
\]
and cost function
$(2 c^*)^*(\cdot) = 2c(\cdot/2)$.
By Lemma~\ref{lem:convex-conjugate-elementary-lemma}, $2c(\cdot/2) \geq c(\cdot)/2$,
so this translates into the inequality $\bfT_c$
(with the initial cost function $c$) and constant $16\widetilde{C}$.
This completes the proof.
\end{proof}

%
%

\section{Proofs of stability of weak transportation inequalities}

\label{sec:proofs-weak}

\subsection{Preliminaries: weak transportation inequalities}

Given a convex cost function $c\colon\RR^d\to [0,\infty)$
and $t>0$ we define the infimum convolution operator $Q_t^c$
acting on Lipschitz functions $f\colon\RR^d\to\RR$ by the formula
\begin{equation}
\label{eq:definition-of-Q_t}
Q_t^cf(x)
= \inf_{y \in \RR^d} \Big \{f(y) + t c\Bigl(\frac{x-y}{t}\Bigr) \Big\},
\quad x \in \RR^d.
\end{equation}
We will often suppress the dependence on the cost function $c$
and omit the superscript in the notation.

The following dual formulation of the weak transportation inequalities
is very handy;
the assumption  $c(x) \geq a\abs{x}+b$ ensures that the duality from \cite[Theorem~2.11]{grst-Kantorovich-duality} holds.

\begin{theorem}[{{\cite[Proposition~4.5]{grst-Kantorovich-duality}}}]
\label{thm:duality-from-Kantorovich-duality}
Let $c\colon\RR^d\to [0,\infty)$ be a continuous convex function such that $c(0)=0$
and $c(x) \geq a\abs{x}+b$
for some $a>0$, $b\in\RR$ and all $x\in\RR^d$.
Let $\mu$ be a probability measure on $\RR^d$.
Then the following conditions are equivalent:
\begin{enumerate}[label=(\roman*)]
\item
The measure $\mu$ satisfies the weak transportation inequality $\bfTbar_c^-$
with constant~$C$.
\item For every convex, Lipschitz, and bounded from below function $f\colon\RR^d\to\RR$
we have
\begin{equation}
\label{eq:dual-formulation}
\int_{\RR^d} \exp\bigl(C^{-1} Q_1^c f(x)\bigr)d\mu(x)
\leq \exp\Bigl(  C^{-1} \int_{\RR^d} f(x) d\mu(x) \Bigr).
\end{equation}
\end{enumerate}
\end{theorem}

Let us also recall the following standard result
(see Appendix~\ref{sec:appendix-a} for the proof).

\begin{corollary}
\label{cor:weak-transport-implies-convex-Poincare}
Let $c\colon\RR^d\to [0,\infty)$ be a continuous convex function such that $c(x)=\abs{x}^2/2$ for $\abs{x}\leq \delta$ (for some $\delta>0$).
Suppose that $\mu$ is a probability measure on $\RR^d$
which satisfies the weak transportation inequality $\bfTbar_c^-$ with constant $C$.
Then $\mu$ satisfies
the convex Poincar\'e inequality~\eqref{eq:definition-convex-Poincare}
with the same constant $C$.
\end{corollary}

Let us now explain why~\eqref{eq:implication-chain-convex-setting} holds
and in what sense the weak transportation $\bfTbar_2^-$
and the convex log-Sobolev inequality are equivalent.
Some extensions of the next theorem
(to  other cost functions and certain convex modified log-Sobolev inequalities)
can be found in~\cite{MR3456588,shu-strzelecki};
we provide yet another fairly general result of this type in Theorem~\ref{thm:equivalence} below.
Below the assumption about the exponential integrability of the norm
is introduced in order to exclude heavy-tailed measures
for which the only exponentially integrable convex functions are constants.

\begin{theorem}[{{\cite[Theorem~8.8]{grst-Kantorovich-duality}}}]
\label{thm:characterization-from-Kantorovich-duality}
Let $\mu$ be a probability measure on $\RR^d$.

If $\mu$ satisfies the transport--entropy inequality $\bfTbar_2^-$
with constant $C_1$,
then we have $\int_{\RR^d} \exp(s\abs{x}) d\mu(x) <\infty$ for all $s>0$
and
for every $C^1$-smooth function $f\colon \RR^d\to\RR$
convex, Lipschitz, and bounded from below,
the convex log-Sobolev inequality~\eqref{eq:definition-convex-log-Sobolev} holds with constant $4C_1$.

If $\int_{\RR^d} \exp(s\abs{x}) d\mu(x) <\infty$ for all $s>0$
and $\mu$ satisfies the convex log-Sobolev inequality~\eqref{eq:definition-convex-log-Sobolev} with constant $C_2$ for every convex function which is $C^1$-smooth, Lipschitz, and bounded from below,
then $\mu$ satisfies $\bfTbar_2^-$ with constant $C_2$.
\end{theorem}

We will also need the following technical result;
we defer the proof to Appendix~\ref{sec:appendix-a}.

\begin{lemma}
\label{lem:aux-exponential-moment-of-mixture}
Let
$\{\mu_\theta\}_{\theta\in\Theta}$
be a family of probability measures on $\RR^d$
which satisfy the $\bfTbar_2^-$ inequality with uniformly bounded constants:
\[
C_\Theta \coloneqq \sup_{\theta\in\Theta} \ConstantTransportTbarTwoMinus{\mu_\theta} < \infty.
\]
Suppose moreover that either
\[
D_\Theta \coloneqq \sup_{\theta\in\Theta, \theta'\in\Theta}
\BarCostQuadratic{\mu_\theta}{\mu_{\theta'}} < \infty
\quad \text{ or } \quad
M_\Theta\coloneqq \sup_{\theta,\theta' \in \Theta} \chisquare{\mu_\theta}{\mu_{\theta'}} < \infty.
\]
Then the mixture $\mu_\mix$,
defined as in~\eqref{eq:definition-of-mixture},
has all exponential moments:
\[
\int_{\RR^d} \exp(s\abs{x}) d\mu_{\mix}(x) <\infty
\quad \text{ for all } s>0.
\]
\end{lemma}

\subsection{Proof of Proposition~\ref{prop:stability-of-convex-Poincare-improved}}

The proof of Proposition~\ref{prop:stability-of-convex-Poincare-improved}
is similar to the one of \cite[Proposition 5.3]{adamczak-strzelecki},
but uses more subtle estimates,
which allow us to bound the constant in the convex Poincar\'e inequality for the mixture
in terms of the weak transport cost between mixed components
(which is smaller than the classical transport cost).

\begin{proof}[Proof of Proposition~\ref{prop:stability-of-convex-Poincare-improved}]
Denote $C\coloneqq C_\Theta$.
By the decomposition~\eqref{eq:variance-of-mixture},
\begin{align}
\label{eq:vdfcp}
\MoveEqLeft[1]
\Var_{\mu_{\mix}} (f)
= \int_\Theta \Var_{\mu_\theta} (f)  dm(\theta)  + \Var_m\Bigl(\theta \mapsto \int_{\RR^d} f(x) d\mu_\theta(x) \Bigr)\\
&= \int_\Theta \Var_{\mu_\theta} (f)  dm(\theta)
+ \frac{1}{2} \int_\Theta\int_\Theta  \Bigl( \int_{\RR^d} f d\mu_{\theta_1} - \int_{\RR^d}f d\mu_{\theta_2} \Bigr)^2 dm(\theta_1) dm(\theta_2). \nonumber
\end{align}
If $f\colon\RR^d\to\RR$ is convex, then the first summand can be estimated by
\[
C \int_{\RR^d} |\nabla f|^2 d\mu_{\mix},
\]
so it remains to deal with the second term.

Fix a convex function $f\colon\RR^d\to\RR$.
For $x\in\RR^d$ let
\[
\partial f(x) \coloneqq\{ v\in\RR^d : f(x) + \langle v, y-x\rangle \leq f(y) \text{ for every } y\in\RR^d\}
\]
be the subdifferential of $f$ at $x$;
this is a convex compact set.
Moreover, it is known that there exists a Borel measurable map
$g\colon\RR^d\to\RR^d$ such that $g(x)\in \partial f(x)$ for every $x\in\RR^d$
(see, e.g., \cite{mathoverflow-453991} (cf.\ \cite[6.9.7. Theorem]{bogachev})
or \cite[Proposition~3.1]{bobkov2026subgradientsconvexfunctionsorlicz}).

For fixed $\theta_1,\theta_2 \in \Theta$  let $X$, $Y$ be random vectors in $\RR^d$ with laws $\mu_{\theta_1}$, $\mu_{\theta_2}$, respectively.
By convexity of $f$ (and our choice of the subgradient),
\begin{align*}
\EE f(X) -\EE f(Y)
&\leq \EE \langle g(X), X-Y \rangle
= \EE \bigl\langle g(X), X-\conditionalEE{Y}{X} \bigr\rangle\\
&\leq \sqrt{\EE \abs{g(X)}^2} \sqrt{\EE \abs{X-\conditionalEE{Y}{X}}^2}\\
&\leq \sqrt{\EE \abs{\nabla f (X)}^2} \sqrt{\EE \abs{X-\conditionalEE{Y}{X}}^2}.
\end{align*}
Taking the infimum over all realizations of $X$ and $Y$, we conclude that
\begin{equation*}
\int_{\RR^d} f d\mu_{\theta_1} - \int_{\RR^d}f d\mu_{\theta_2}
\leq
\Bigl(\int_{\RR^d} |\nabla f|^2 d\mu_{\theta_1}\Bigr)^{1/2}
\bigl( \BarCostQuadratic{\mu_{\theta_2}}{\mu_{\theta_1}} \bigr)^{1/2}.
\end{equation*}
Similarly,
\begin{equation*}
\int_{\RR^d} f d\mu_{\theta_2} - \int_{\RR^d}f d\mu_{\theta_1}
\leq
\Bigl(\int_{\RR^d} \abs{\nabla f}^2 d\mu_{\theta_2}\Bigr)^{1/2}
\bigl( \BarCostQuadratic{\mu_{\theta_1}}{\mu_{\theta_2}} \bigr)^{1/2}
\end{equation*}
(note that we do not claim that $ \BarCostQuadratic{\mu_{\theta_2}}{\mu_{\theta_1}}$
and $\BarCostQuadratic{\mu_{\theta_1}}{\mu_{\theta_2}}$ are realized by the same coupling).

Since for real numbers $s, a, b \in\RR$ the inequalities $s\leq |a|$ and $-s \leq |b|$ imply that $|s| \leq \max\{|a|, |b|\}$ and so $s^2 \leq a^2 + b^2$,
we can conclude that
\begin{equation*}
\Bigl(\int_{\RR^d} f d\mu_{\theta_1} - \int_{\RR^d}f d\mu_{\theta_2} \Bigr)^2
\leq
\Bigl( \int_{\RR^d} \abs{\nabla f}^2 d\mu_{\theta_1}
 + \int_{\RR^d} \abs{\nabla f}^2 d\mu_{\theta_2} \Bigr)
\cdot
\sup_{\theta, \theta' \in \Theta} \BarCostQuadratic{\mu_{\theta} }{\mu_{\theta'}}.
\end{equation*}
After multiplying by $1/2$, integrating over $\theta_1\in\Theta$, $\theta_2\in \Theta$,
and recalling~\eqref{eq:vdfcp}, we arrive at the assertion of the proposition.
\end{proof}

\subsection{Proofs of results with assumptions about the \texorpdfstring{$\chi^2$}{chi-squared}-distance}

\begin{proof}[Proof of Proposition~\ref{prop:stability-of-convex-Poincare-under-chi-square}]
This follows exactly from the same calculations
as in the proof of Proposition~\ref{prop:stability-of-Poincare-with-improved-constants}.
\end{proof}

\begin{theorem}
\label{thm:stability-of-covex-log-Sobolev-under-chi-square}
Let
$\{\mu_\theta\}_{\theta\in\Theta}$
be a family of probability measures on $\RR^d$
which satisfy the convex log-Sobolev inequality with uniformly bounded constants:
\[
C_\Theta\coloneqq \sup_{\theta\in\Theta} \ConstantConvexLogSobolev{\mu_\theta} < \infty.
\]
Suppose moreover that
\[
M_\Theta\coloneqq \sup_{\theta,\theta' \in \Theta} \chisquare{\mu_\theta}{\mu_{\theta'}} < \infty.
\]
Then the mixture $\mu_\mix$,
defined as in~\eqref{eq:definition-of-mixture},
satisfies the convex log-Sobolev inequality
and
\[
\ConstantConvexLogSobolev{\mu_\mix} \leq
\sup_{\theta\in\Theta} \ConstantConvexLogSobolev{\mu_\theta}
\cdot
\bigl(2 + \frac{1}{4}(2+M_{\Theta})\log(1+M_{\Theta})\bigr).
\]
\end{theorem}

\begin{proof}
Denote $C\coloneqq C_\Theta$ and $M \coloneqq M_\Theta$.
By assumption, for every $\theta\in\Theta$
the mixed component $\mu_\theta$
satisfies the convex log-Sobolev inequality with constant at most $C$,
so by \eqref{eq:implication-chain-convex-setting} it also satisfies
the convex Poincar\'e inequality with constant at most $C$
(cf.\ Theorem~\ref{thm:characterization-from-Kantorovich-duality}
and Corollary~\ref{cor:weak-transport-implies-convex-Poincare}).
Hence, by Proposition~\ref{prop:stability-of-convex-Poincare-under-chi-square},
the mixture $\mu_\mix$ satisfies the convex Poincar\'e inequality
with constant $(1+M/2)C$.

Take any $C^1$-smooth, convex, Lipschitz, and bounded from below
 function $f\colon\RR^d\to\RR$;
 note that then $e^{f/2}$ is also convex.
Applying Lemma~\ref{lem:entropy-of-mixture-chisquared-estimate} to the function $e^{f/2}$ yields
\begin{displaymath}
\Ent_{\mu_\mix}(e^f)
\leq
2\int_\Theta \Ent_{\mu_\theta}(e^f)dm(\theta)
+ \log(1+M) \Var_{\mu_\mix}(e^{f/2}).
\end{displaymath}
Hence, using the convex log-Sobolev inequality~\eqref{eq:definition-convex-log-Sobolev}
for the mixed components $\mu_\theta$
as well as the convex Poincar\'e inequality for the mixture $\mu_{\mix}$, we can write
\begin{align*}
\Ent_{\mu_\mix}(e^f)
&\leq
C \int_\Theta \int_{\RR^d} \abs{\nabla f}^2 e^f d\mu_\theta dm(\theta)\\
&\quad + \frac{1}{2} (2+M)\log(1+M) C  \int_{\RR^d} \frac{1}{4}\abs{\nabla f}^2 e^f d\mu_\mix\\
&\leq    \frac{1}{2} \cdot C\cdot\bigl(2 + \frac{1}{4} (2+M)\log(1+M)\bigr)
\int_{\RR^d} \abs{\nabla f}^2 e^f d\mu_\mix,
\end{align*}
that is, $\mu_\mix$ satisfies
the convex log-Sobolev inequality with constant
\[
 C\cdot\bigl(2 + \frac{1}{4} (2+M)\log(1+M)\bigr).
\]
This completes the proof.
\end{proof}

\begin{corollary}
\label{cor:stability-of-Tbar_2^--under-chi-square}
Let
$\{\mu_\theta\}_{\theta\in\Theta}$
be a family of probability measures on $\RR^d$
which satisfy the $\bfTbar_2^-$ inequality with uniformly bounded constants:
\[
C_\Theta \coloneqq \sup_{\theta\in\Theta} \ConstantTransportTbarTwoMinus{\mu_\theta} < \infty.
\]
Suppose moreover that
\[
M_\Theta\coloneqq \sup_{\theta,\theta' \in \Theta} \chisquare{\mu_\theta}{\mu_{\theta'}} < \infty.
\]
Then the mixture $\mu_\mix$,
defined as in~\eqref{eq:definition-of-mixture},
satisfies the $\bfTbar_2^-$ inequality
and
\[
\ConstantTransportTbarTwoMinus{\mu_\mix} \leq
\sup_{\theta\in\Theta} \ConstantTransportTbarTwoMinus{\mu_\theta}
\cdot
\bigl(8 + \frac{1}{4}(2+M_{\Theta})\log(1+M_{\Theta})\bigr).
\]
\end{corollary}

\begin{proof}
Denote $C\coloneqq C_\Theta$ and $M \coloneqq M_\Theta$.
By assumption, for every $\theta\in\Theta$
the mixed component $\mu_\theta$
satisfies the $\bfTbar_2^-$ inequality with constant at most $C$,
so by \eqref{eq:implication-chain-convex-setting} it also satisfies
the convex Poincar\'e inequality with constant at most $C$
(cf.\ Corollary~\ref{cor:weak-transport-implies-convex-Poincare}).
Hence, by Proposition~\ref{prop:stability-of-convex-Poincare-under-chi-square},
the mixture $\mu_\mix$ satisfies the convex Poincar\'e inequality
with constant $(1+M/2)C$.
Moreover, by Theorem~\ref{thm:characterization-from-Kantorovich-duality},
the mixed component $\mu_\theta$
satisfies the convex log-Sobolev inequality~\eqref{eq:definition-convex-log-Sobolev}
with constant~$4C$
and, by Lemma~\ref{lem:aux-exponential-moment-of-mixture}, the mixture $\mu_\mix$ has all exponential moments.
We proceed as in the proof of Theorem~\ref{thm:stability-of-covex-log-Sobolev-under-chi-square}
and conclude that $\mu_\mix$ satisfies
the convex log-Sobolev inequality with constant
\[
 C\cdot\bigl(8 + \frac{1}{4}(2+M)\log(1+M)\bigr)
\]
(note that the second summand in the parentheses does not change,
since it corresponds to the part controlled by the convex Poincar\'e inequality,
which holds with the same constant as in the proof of Theorem~\ref{thm:stability-of-covex-log-Sobolev-under-chi-square}).
Thus, again by the characterization from  Theorem~\ref{thm:characterization-from-Kantorovich-duality},
we conclude that $\mu_\mix$
satisfies the $\bfTbar_2^-$ inequality with the same constant.
This completes the proof.
\end{proof}

\subsection{Proof of Theorems~\ref{thm:stability-of-Tbar_2^-} and~\ref{thm:stability-of-Tbar_c^-}}

Let us now pass to the proof of Theorem \ref{thm:stability-of-Tbar_2^-}.
In the proof we will be using couplings of measures on $\Theta$,
so in order to avoid measurability issues
it is convenient to assume that the set $\Theta$ is finite.
We therefore start with the following simple lemma.
We note that an alternative approach would be
to assume without loss of generality that $\Theta$ is a Polish space
with the Borel $\sigma$-field
and then use appropriate measurable selection theorems.
Indeed, we may always replace $\Theta$ with the space of all Borel probability measures on $\RR^d$,
equipped with the smallest $\sigma$-field for which the evaluation maps $\mu \mapsto \mu(A)$,
where $A$ is a Borel subset of $\RR^d$, are measurable.
This $\sigma$-field coincides with the Borel $\sigma$-field generated by the topology of weak convergence, under which the space of probability measures is Polish (see, e.g., \cite[Theorem 1.12.4]{MR4628026}).

\begin{lemma}
\label{lem:approximation-Tbar}
Let $c\colon\RR^d\to[0,\infty)$ be a convex cost function
such that $c(0)=0$
and let $\mu$ and $\mu_n$, $n\in\NN$, be probability measures on $\RR^d$.
Suppose that for each $n\in\NN$ the measure $\mu_n$
satisfies the weak transportation inequality $\bfTbar_c^-$,
\begin{gather*}
\sup_{n\in\NN} \ConstantTransportTbarGeneralMinus{c}{\mu_n} <\infty,\\
\sup_{n\in\NN} \int_{\RR^d} \abs{x}^2 d\mu_n(x) <\infty,
\end{gather*}
and $\mu_n\to\mu$ weakly.
Then $\mu$ satisfies the weak transportation inequality $\bfTbar_c^-$
and
\[
\ConstantTransportTbarGeneralMinus{c}{\mu}
\leq \sup_{n\in\NN} \ConstantTransportTbarGeneralMinus{c}{\mu_n} <\infty.
\]
\end{lemma}

\begin{proof}
Denote $C\coloneqq \sup_{n\in\NN} \ConstantTransportTbarGeneralMinus{c}{\mu_n}$.
Let $f\colon\RR^d\to\RR$ be convex, Lipschitz, and bounded from below.
Since $\mu_n\to\mu$ weakly, $f$ is Lipschitz, and the second moments of the measures
$\mu_n$ are uniformly bounded, we obtain
\[
\lim_{n\to\infty} \int_{\RR^d} f(x) d\mu _n (x) = \int_{\RR^d} f(x) d\mu(x).
\]
Thus, since $\mu_n\to\mu$ weakly and the function $\exp(Q_1^c f)$ is continuous and nonnegative
and by the dual formulation~\eqref{eq:dual-formulation} of $\bfTbar_c^-$ for the measures $\mu_n$,
we see that
\begin{align*}
\int_{\RR^d} \exp\bigl(C^{-1} Q_1^c f(x)\bigr)d\mu(x)
&\leq \liminf_{n\to\infty}  \int_{\RR^d} \exp\bigl(C^{-1} Q_1^c f(x)\bigr)d\mu_n(x)\\
&\leq \liminf_{n\to\infty}  \exp\Bigl(  C^{-1} \int_{\RR^d} f(x) d\mu_n(x) \Bigr)\\
&= \exp\Bigl(  C^{-1} \int_{\RR^d} f(x) d\mu(x) \Bigr).
\end{align*}
By the dual formulation~\eqref{eq:dual-formulation} this is equivalent to the fact
that $\mu$ satisfies $\bfTbar_c^-$ with constant~$C$.
\end{proof}

\begin{proof}[Proof of Theorem~\ref{thm:stability-of-Tbar_2^-}]
\textsc{Step 0:} Approximation.
Note that the collection $\{\mu_\theta\}_{\theta \in \Theta}$ is tight.
Indeed, the  weak transportation inequality $\bfTbar_2^-$ implies convex Poincar\'e inequality and hence convex concentration.
Moreover, due to $D_\Theta < \infty$, the barycenters of $\mu_\theta$ form a bounded set in $\RR^d$.
In particular, by the convex Poincar\'e inequality, the  measures $\mu_\theta$ have uniformly bounded second moments and the same is true for every mixture measure formed from (a subset of) $\{\mu_\theta\}_{\theta \in \Theta}$.

Let $d_{BL}$ be the bounded Lipschitz metric on $\mathcal P(\RR^d)$.
By Prokhorov's theorem,
for any $n \in \NN$ let us choose a finite $\frac{1}{n}$-net in $d_{BL}$ distance,
$\mathcal E_n \subset \{\mu_\theta\}_{\theta \in \Theta}$.
For every $\theta \in \Theta$ we choose some $\mu_\theta^n \in \mathcal E_n$ such that $d_{BL}(\mu_\theta,\mu_\theta^n) < \frac{1}{n}$.
The map $\mu_\theta \mapsto \mu_\theta^n$
may be chosen to be Borel measurable.  Moreover, $\mathcal B(\mathcal P(\RR^d))$ is generated by evaluations $\mu \mapsto \mu(B)$ for every $B \in \mathcal B(\RR^d)$, and so the map $\theta \mapsto \mu_\theta^n$ is measurable.

Now take any $1$-Lipschitz function $f\colon\RR^d \to \RR$ with $\|f\|_\infty \le 1$ and note that
\[
\int_{\Theta}\Big| \int_{\RR^d}f(x)d\mu_\theta(x) - \int_{\RR^d}f(x)d\mu_\theta^n(x)\Big|dm(\theta) \le \frac{1}{n}.
\]
Hence, the mixtures
\[
\mu_\mix^n(\cdot) \coloneqq \int_{\Theta} \mu_\theta^n(\cdot)dm(\theta)
\]
converge weakly to $\mu_\mix$.
Moreover, by the  above observations,
the measures $\mu_\mix^n$ have bounded second moments.
By notion of Lemma~\ref{lem:approximation-Tbar} we may and will therefore assume that the set $\Theta$ is finite.
Without loss of generality,
we also assume that $\support m = \Theta$
and that $\Theta$ is equipped with the $\sigma$-algebra of all subsets of $\Theta$.

\textsc{Step 1a:} Goal and setting.
Denote $C \coloneqq C_{\Theta}$,
$D\coloneqq D_{\Theta}$,
and fix $\nu\in\mathcal{P}_1(\RR^d)$ which is absolutely continuous with respect to $\mu_{\mix}$
and such that $\RelativeEntropy{\nu}{\mu_{\mix}}<\infty$. Our goal is to show that
\begin{equation}
\label{eq:goal-T-bar-minus-Bedlewo}
\BarCostQuadratic{\mu_{\mix}}{\nu} \leq \max\{ 2C,4 D\} \RelativeEntropy{\nu}{\mu_{\mix}}.
\end{equation}
To this end, we will define an appropriate coupling of $\mu_{\mix}$ and $\nu$.
Let $f$ be the density of $\nu$ with respect to $\mu$,	
\[
\nu(dx) = f(x) \mu_{\mix}(dx).
\]
Let
\begin{equation*}
g(\theta)\coloneqq \int_{\RR^d} f(x) \mu_\theta(dx), \quad \theta\in\Theta,
\end{equation*}
and define $f_\theta \colon \RR^d \to [0,\infty)$ with the formula $f_\theta(x) = \frac{f(x)}{g(\theta)}$ if $g(\theta) > 0$ and as $f_\theta \equiv 1$ otherwise.
Then, for $A\in\Borel{\RR^d}$,
\begin{align*}
\nu(A)
= \int_A f(x) \mu_{\mix}(dx)
&= \int_{\Theta}\int_A  f(x) \mu_\theta(dx) m(d\theta) \\
&=  \int_{\Theta}\int_A f_\theta(x) \mu_\theta(dx) g(\theta) m(d\theta)
=  \int_{\Theta} \nu_\theta(A)  m'(d\theta),
\end{align*}
where we denoted
\begin{equation*}
\nu_\theta(dx) \coloneqq f_{\theta}(x) \mu_\theta(dx),
\quad 	 m'(d\theta) \coloneqq g(\theta) m(d\theta).
\end{equation*}
Note that the measures $\nu_\theta$, $\theta\in\Theta$,
are probability measures on $\RR^d$, 	
$m'$ is a probability measure on $\Theta$,
and $\nu$ is a mixture of the measures $\nu_\theta$, $\theta\in\Theta$, with the mixing measure being $m'$.

\textsc{Step 1b:} Initial observations.
By convexity of the function $t\mapsto t\log t$ and Jensen's inequality,
\begin{align}
\label{eq:observation-entropy-m'-wrt-m}
\RelativeEntropy{m'}{m}
&= \int_{\Theta} g(\theta)\log g(\theta) m(d\theta)\\
&= \int_{\Theta} \int_{\RR^d} f(x) \mu_\theta(dx) \log \Bigl( \int_{\RR^d} f(x) \mu_\theta(dx)  \Bigr) m(d\theta)\nonumber\\
&\leq \int_{\Theta} \int_{\RR^d} f(x)  \log \bigl( f(x)   \bigr) \mu_\theta(dx) m(d\theta)\nonumber\\
&=  \int_{\RR^d} f(x)  \log \bigl( f(x)   \bigr) \mu_{\mix}(dx)
=\RelativeEntropy{\nu}{\mu_{\mix}} < \infty.\nonumber
\end{align}
Moreover, recalling that if $g(\theta')=0$,
then $f = 0$ $\mu_{\theta'}$-a.s.\ and $g(\theta')\log g(\theta') = 0$,
yields
\begin{align}
\label{eq:observation-integrate-relative-entropy-wrt-m'}
\MoveEqLeft[4]
\int_{\Theta} \RelativeEntropy{\nu_{\theta'}}{\mu_{\theta'}} m'(d\theta')\\
&= \int_{\Theta}
\int_{\RR^d} f_{\theta'}(x) \log \bigl( f_{\theta'}(x) \bigr) \mu_{\theta'}(dx)
g(\theta')m(d\theta')\nonumber\\
&= \int_{\Theta}
\int_{\RR^d} \frac{f(x)}{g(\theta')} \log \Bigl( \frac{f(x)}{g(\theta')} \Bigr) \mu_{\theta'}(dx)
\indicatorbraces{g(\theta') > 0} g(\theta')m(d\theta')\nonumber\\
&=  \int_{\Theta}
\int_{\RR^d} f(x) \log\bigl( f(x)\bigr) \mu_{\theta'}(dx)\indicatorbraces{g(\theta') > 0}
m(d\theta') \nonumber\\
&\quad
- \int_{\Theta}
\int_{\RR^d} f(x)\log \bigl( g(\theta')\bigr) \mu_{\theta'}(dx)
\indicatorbraces{g(\theta') > 0} m(d\theta')\nonumber\\
&=
\RelativeEntropy{\nu}{\mu_{\mix}}
- \int_{\Theta} g(\theta')\log\bigl( g(\theta')\bigr) m(d\theta')\nonumber\\
&= \RelativeEntropy{\nu}{\mu_{\mix}} - \RelativeEntropy{m'}{m}.\nonumber
\end{align}
In particular,
it follows that $m'$-a.s.\ $\RelativeEntropy{\nu_{\theta'}}{\mu_{\theta'}}$ is finite.
Since the measure $\mu_{\theta'}$ satisfies the weak transportation inequality $\bfTbar_2^-$, we also conclude that $\BarCostQuadratic{\nu_{\theta'}}{\mu_{\theta'}}$ is finite ($m'$-a.s.).

\textsc{Step 1c:} Notation.
In what follows, it will be convenient to use the following notational convention.
For random variables $U$, $V$ we will denote by $p_U(du)$
and $p_{U,V}(du,dv)$
the law of $U$
and the joint law of the pair $(U,V)$, respectively.
Moreover,
$p_{U | V}(du|v)$ will stand for the regular conditional distribution of $U$ given $V=v$. Thus, $p_{U,V}(du,dv) = p_{U|V}(du|v) p_V(dv)$.

In a similar spirit, if $\pi(du,dv)$ is some coupling of the probability measures $\alpha(du)$ and $\beta(dv)$, we shall write
\[
\pi(du,dv) = \pi(du|v) \beta(dv) = \pi(dv|u) \alpha(du)
\]
with $\pi(du|v)$ and $\pi(dv|u)$ being the disintegration kernels
of $\pi$
with respect to its second and first marginal,
respectively.
This is a minor abuse of notation,
but the meaning of the conditioning will always be clear from the context and the choice of variables.

\textsc{Step 2:} Construction of the couplings.
We are ready to define a certain coupling $(X,Y)$ of $\mu_{\mix}(dx)$
and $\nu(dy)$, which will be a witness to \eqref{eq:goal-T-bar-minus-Bedlewo}.
We will define first a coupling $(\capitalPhiTheta,\capitalPhiTheta')$ of $m$ and $m'$.
Then, we will consecutively introduce an auxiliary random variable $X'$ and our target variables $X,Y$, each time by specifying their conditional distribution given the already defined variables.
The conditional distribution of the variable $X'$ given $\capitalPhiTheta,\capitalPhiTheta'$ will depend only on $\Phi'$,
i.e., the variables $\capitalPhiTheta,\capitalPhiTheta',X'$ will form a Markov chain.
The conditional distribution of $X$ given $\capitalPhiTheta,\capitalPhiTheta',X'$ will truly depend on all the three variables,
while the conditional distribution of $Y$ given $\capitalPhiTheta,\capitalPhiTheta',X,X'$ will depend only on $\capitalPhiTheta',X'$, i.e.,
the three random elements $(\capitalPhiTheta,X),(\capitalPhiTheta',X'),Y$ will form a Markov chain.
We will also rely on the elementary fact that a process obtained from a Markov chain by reversing the time remains a Markov chain (see Lemma \ref{lem:Markov-reverse} below).
In particular, $X',\capitalPhiTheta',\capitalPhiTheta$ and $Y,(\capitalPhiTheta',X'),(\capitalPhiTheta,X)$ will have the Markov property.


Having introduced the intuition, let us pass to the details:
\begin{itemize}
\item
Due to a result of Marton \cite[Lemma 3.2]{MR1392329} (see also, e.g., \cite[Theorem 8.5]{boucheron-lugosi-massart-2012}) there exists a coupling $(\capitalPhiTheta,\capitalPhiTheta')$ of $m$, $m'$ such that
\begin{align}\label{eq:Marton-coupling}
\EE\Bigl( \conditionalPP[\big]{\capitalPhiTheta' \neq \capitalPhiTheta}{\capitalPhiTheta'}^2 \Bigr) \leq 2 H(m'|m).
\end{align}
\item Let $X'$ be an $\RR^d$-valued random variable which conditionally on $\capitalPhiTheta = \theta$, $\capitalPhiTheta' = \theta'$ is distributed according
to $\mu_{\theta'}$, i.e., \[
p_{X'|\capitalPhiTheta,\capitalPhiTheta'}(dx'|\theta,\theta') = \mu_{\theta'}(dx').
\]
It is easy to check that
\begin{equation}
\label{eq:law-of-X'-given-Phi'}
p_{X'|\capitalPhiTheta'}(dx' | \theta') = \mu_{\theta'}(dx'),
\qquad
p_{X'}(dx') = \int_{\Theta}\mu_{\theta'} (dx') m'(d\theta').
\end{equation}
Indeed, to formally justify the first equality we write, for a Borel set $B'\subset\RR^d$ and a measurable set $A'\subset \Theta$,
\begin{align*}
\EE \indicatorbraces{X'\in B'} \indicatorbraces{\capitalPhiTheta'\in A'} &= \EE\conditionalEE[\big]{\indicatorbraces{X'\in B'} \indicatorbraces{\capitalPhiTheta'\in A'}}{\capitalPhiTheta,\capitalPhiTheta'} \\
& = \EE\Bigl[ \conditionalEE[\big]{\indicatorbraces{X'\in B'} }{\capitalPhiTheta,\capitalPhiTheta'}\indicatorbraces{\capitalPhiTheta'\in A'} \Bigr]\\
& = \EE\bigl[ \mu_{\capitalPhiTheta'}(B') \indicatorbraces{\capitalPhiTheta'\in A'} \bigr],
\end{align*}
i.e., $\conditionalEE{\indicatorbraces{X'\in B'}}{\capitalPhiTheta'} = \mu_{\capitalPhiTheta'}(B')$, which implies that
$p_{X'|\capitalPhiTheta'}(dx' | \theta') = \mu_{\theta'}(dx')$. The expression for $p_{X'}(dx')$ then follows by integrating over $m'(d\theta')$.

\item For $\theta, \theta'\in\Theta$ let $\pi_{\theta',\theta}(dx',dx)$
be a coupling of $\mu_{\theta'}(dx')$ and $\mu_{\theta}(dx)$
such that $\calTbar_2(\mu_\theta|\mu_{\theta'})$ is minimal
(it exists by Lemma \ref{lem:minimizer-existence});
using our notational convention we can write
\[
\pi_{\theta',\theta}(dx',dx) = \pi_{\theta',\theta}(dx|x') \mu_{\theta'}(dx').
\]

It follows from the finiteness of $\Theta$ and measurability of $x' \mapsto \pi_{\theta',\theta}(\cdot|x')$
 that the map $(\theta,\theta',x')\mapsto \pi_{\theta',\theta}(\cdot|x')$
is measurable (with respect to the product $\sigma$-field on $\Theta \times \Theta \times \RR^d$ and the Borel $\sigma$-field on the space of probability measures; equivalently for all Borel subsets $A$ of $\RR^d$, the function $(\theta,\theta',x) \mapsto \pi_{\theta',\theta}(A|x')$ is Borel measurable).

\item Now let $X$ be an $\RR^d$-valued random variable, which conditionally on $\capitalPhiTheta = \theta$, $\capitalPhiTheta' = \theta'$, $X'=x'$,
is distributed according to $\pi_{\theta',\theta}(dx|x')$, i.e.,
\begin{align}
\label{definition-of-X}
p_{X|\capitalPhiTheta,\capitalPhiTheta', X'}(dx|\theta,\theta',x') = \pi_{\theta',\theta}(dx|x').
\end{align}
It follows that
\begin{align}
\label{eq:law-of-X-X'-given-Phi-Phi'}
p_{X,X'|\capitalPhiTheta,\capitalPhiTheta'}(dx,dx'|\theta, \theta') &= \pi_{\theta',\theta}(dx',dx),\\
p_{X|\capitalPhiTheta,\capitalPhiTheta'}(dx|\theta,\theta') &= \mu_{\theta}(dx).\nonumber
\end{align}
Indeed, if $A,B$ are Borel subsets of $\RR^d$, then
\begin{align*}
\conditionalEE{\indicatorbraces{X \in A}\indicatorbraces{X' \in B}}{\Phi',\Phi}
&= \conditionalEE[\Big]{
\conditionalEE[\big]{\indicatorbraces{X \in A}}{X',\Phi,\Phi'}
\indicatorbraces{X' \in B}}{ \Phi',\Phi}\\
&= \conditionalEE[\Big]{\pi_{\Phi',\Phi}(A|X')\indicatorbraces{X'\in B}}{\Phi',\Phi} \\
&= \int_{B} \pi_{\Phi',\Phi}(A|x') \mu_{\Phi'}(dx') = \pi_{\Phi',\Phi}(B\times A),
\end{align*}
which proves the first equality above by Dynkin's $\pi$-$\lambda$ lemma. The second equality follows from the first one by substituting $B = \RR^d$ and the definition of $\pi_{\theta',\theta}$.

Moreover, we can check as in the case of~\eqref{eq:law-of-X'-given-Phi'} that
\begin{equation*}
p_{X|\capitalPhiTheta}(dx|\theta) = \mu_{\theta}(dx),
\qquad
p_{X}(dx) = \int_{\Theta}\mu_{\theta} (dx) m(d\theta) = \mu_{\mix}(dx).
\end{equation*}

Let us also note that if $\theta = \theta'$,
then clearly $\BarCostQuadratic{\mu_\theta}{\mu_{\theta'}} = 0$
and this is attained by the coupling $\delta_{x'}(dx)\mu_\theta(dx')$, i.e., on the set $\{\Phi = \Phi'\}$ we can put $X = X'$.\footnote{This is the only  coupling attaining $\BarCostQuadratic{\mu_\theta}{\mu_{\theta'}} = 0$,
but in our argument we will not exploit uniqueness.}
\item Let $\pi_{\theta'}(dy,dx')$ be a coupling of $\nu_{\theta'}(dy)$ and $\mu_{\theta'}(dx')$
such that $\BarCostQuadratic{\mu_{\theta'}}{\nu_{\theta'}}$ is minimal
(the existence follows again by Lemma \ref{lem:minimizer-existence});
using our notational convention we can write
\[
\pi_{\theta'}(dy,dx') = \pi_{\theta'}(dy|x') \mu_{\theta'}(dx') = \pi_{\theta'}(dx'|y) \nu_{\theta'}(dy).
\]
Note that due to the finiteness of $\Theta$ the map $(\theta',y) \mapsto \pi_{\theta'}(\cdot|y)$ is measurable.
\item Let finally $Y$ be a random variable, which conditionally on  $\capitalPhiTheta = \theta$, $\capitalPhiTheta' = \theta'$, $X'=x', X = x$,
is distributed according to $\pi_{\theta'}(dy|x')$, i.e.,
\[
p_{Y|\capitalPhiTheta,\capitalPhiTheta', X, X'}(dy|\theta,\theta',x,x') = \pi_{\theta'}(dy|x').
\]
Again, as in the case of \eqref{eq:law-of-X'-given-Phi'} and~\eqref{eq:law-of-X-X'-given-Phi-Phi'}, it follows that
\begin{align}
p_{Y|\capitalPhiTheta',X'}(dy|\theta',x')
& = \pi_{\theta'}(dy|x'),\nonumber\\
\label{eq:law-of-Y-X'-given-Phi-Phi'}
p_{Y,X'|\capitalPhiTheta,\capitalPhiTheta'}(dy,dx'|\theta, \theta')
= p_{Y,X'|\capitalPhiTheta'}(dy,dx'|\theta')
& = \pi_{\theta'}(dy,dx'),\\
p_{Y|\capitalPhiTheta,\capitalPhiTheta'}(dy|\theta,\theta') & = \nu_{\theta'}(dy),
\nonumber
\end{align}
and
\begin{equation*}
p_{Y|\capitalPhiTheta'}(dy|\theta') = \nu_{\theta'}(dy),
\qquad
p_{Y}(dy) = \int_{\Theta}\nu_{\theta'} (dy) m'(d\theta') = \nu(dy).
\end{equation*}
Moreover,
\begin{equation}
\label{eq:law-of-X'-given-Y-Phi'}
p_{X'|Y,\capitalPhiTheta'}(dx'|y,\theta') = \pi_{\theta'}(dx'|y).
\end{equation}
Indeed, for Borel sets $B', C\subset\RR^d$ and a measurable set $A'\subset \Theta$,
\begin{align*}
\EE \indicatorbraces{X'\in B'}\indicatorbraces{Y\in C} \indicatorbraces{\capitalPhiTheta'\in A'} &= \EE\Bigl[\conditionalEE[\big]{\indicatorbraces{X'\in B'}\indicatorbraces{Y\in C} }{\capitalPhiTheta'} \indicatorbraces{\capitalPhiTheta'\in A'}\Bigr]\\
& = \EE\bigl[ \pi_{\capitalPhiTheta'}(C\times B') \indicatorbraces{\capitalPhiTheta'\in A'} \bigr]\\
& = \EE\Bigl[ \int_{\RR^d} \pi_{\capitalPhiTheta'}(B'|y) \indicatorbraces{y\in C}\nu_{\capitalPhiTheta'}(dy) \indicatorbraces{\capitalPhiTheta'\in A'} \Bigr]\\
& = \EE\Bigl[
\conditionalEE[\big]{\pi_{\capitalPhiTheta'}(B'|Y) \indicatorbraces{Y\in C}}{\capitalPhiTheta'}
\indicatorbraces{\capitalPhiTheta'\in A'} \Bigr]\\
& = \EE\bigl[ \pi_{\capitalPhiTheta'}(B'|Y) \indicatorbraces{Y\in C} \indicatorbraces{\capitalPhiTheta'\in A'}\bigr],
\end{align*}
which by Dynkin's $\pi$-$\lambda$ lemma (or finiteness of $\Theta$) proves that
we have
$\conditionalEE{\indicatorbraces{X'\in B'}}{Y, \capitalPhiTheta'} =  \pi_{\capitalPhiTheta'}(B'|Y)$, which implies the claim.
\end{itemize}


\textsc{Step 3:} Estimation of the transport cost.
Since $X \sim \mu_{\mix}$ and $Y \sim \nu$, to prove \eqref{eq:goal-T-bar-minus-Bedlewo}  it is enough to show that
\begin{align}\label{eq:local-goal-weak-transport}
\EE \|Y - \EE[X|Y]\|^2 \le \max\{ 2C,4 D\} \RelativeEntropy{\nu}{\mu_{\mix}}.
\end{align}
Note first that by Lemma \ref{lem:Markov-reverse}, for random variables constructed as above we have
\begin{equation}
\label{eq:law-of-Phi-given-Phi'-X'}
p_{\capitalPhiTheta|\capitalPhiTheta',X'}(d\theta | \theta',x')
=
p_{\capitalPhiTheta|\capitalPhiTheta'}(d\theta | \theta').
\end{equation}

We are ready to prove the estimate~\eqref{eq:goal-T-bar-minus-Bedlewo}.
Below we shall denote the Euclidean norm by $\lVert\cdot\rVert$ for visual clarity.
Since $(X,Y)$ is a coupling of $\mu_{\mix}$ and $\nu$, we have
\begin{equation}
\label{eq:goal-T-bar-minus-Bedlewo-two-summands}
\BarCostQuadratic{\mu_{\mix}}{\nu}
\leq \frac{1}{2}\EE\big\lVert Y - \conditionalEE{X}{Y}\big\rVert^2
\leq \EE\big\lVert Y - \conditionalEE{X'}{Y}\big\rVert^2
+ \EE\big\lVert \conditionalEE{X'-X}{Y}\big\rVert^2.
\end{equation}
We shall estimate the two summands on the right-hand side separately.

Recall
from~\eqref{eq:law-of-Y-X'-given-Phi-Phi'}
and~\eqref{eq:law-of-X'-given-Y-Phi'}
that conditionally on $\capitalPhiTheta'=\theta'$
the pair
$(Y,X')$ is distributed according to a coupling $\pi_{\theta'}(dy,dx')$
of $\nu_{\theta'}(dy)$ and $\mu_{\theta'}(dx')$
which minimizes
$\calTbar_2(\mu_{\theta'}|\nu_{\theta'})$
and
$p_{X'|Y,\capitalPhiTheta'}(dx'|y,\theta') = \pi_{\theta'}(dx'|y)$.
Thus, using Jensen's inequality for the conditional expectation,
we can write
\begin{align*}
\EE\big\lVert Y - \conditionalEE{X'}{Y}\big\rVert^2
&= \EE\Big\lVert \conditionalEE[\Big]{ Y-\conditionalEE[\big]{X'}{Y,\capitalPhiTheta'}}{Y} \Big\rVert^2\\
&\leq \EE\Big\lVert  Y-\conditionalEE[\big]{X'}{Y,\capitalPhiTheta'} \Big\rVert^2\\
&= \int_{\Theta}
\int_{\RR^d}
\Big\lVert   y - \int_{\RR^2} x' \pi_{\theta'}(dx'|y)\Big\rVert^2
\nu_{\theta'}(dy)
m'(d\theta')\\
&= 2\int_{\Theta}
\BarCostQuadratic{\mu_{\theta'}}{\nu_{\theta'}}
m'(d\theta')\\
&\leq 2C
\int_{\Theta}
\RelativeEntropy{\nu_{\theta'}}{\mu_{\theta'}}
m'(d\theta'),
\end{align*}
where in the last step we used
the assumption that $\mu_{\theta'}$ satisfies $\bfTbar_2^-$.
Hence, by recalling~\eqref{eq:observation-integrate-relative-entropy-wrt-m'},
we can estimate the first summand from the right-hand side of~\eqref{eq:goal-T-bar-minus-Bedlewo-two-summands} as
\begin{equation}
\label{eq:goal-T-bar-minus-Bedlewo-first-summand-done}
\EE\big\lVert Y - \conditionalEE{X'}{Y}\big\rVert^2
\leq 2C \RelativeEntropy{\nu}{\mu_{\mix}} - 2C \RelativeEntropy
{m'}{m}.
\end{equation}

Let us now proceed to the second summand from the right-hand side of~\eqref{eq:goal-T-bar-minus-Bedlewo-two-summands}.
Recall from~\eqref{eq:law-of-X-X'-given-Phi-Phi'} that conditionally on $\capitalPhiTheta=\theta$, $\capitalPhiTheta'=\theta'$
the pair
$(X',X)$ is distributed according to a coupling  $\pi_{\theta',\theta}(dx',dx)$
of $\mu_{\theta'}(dx')$ and $\mu_{\theta}(dx)$ which minimizes $\calTbar_2(\mu_\theta|\mu_{\theta'})$.

Moreover, using the fact that $(\capitalPhiTheta,X), (\capitalPhiTheta',X'), Y$ forms a Markov chain and Lemma~\ref{lem:Markov-reverse},
we obtain
$p_{X|X',\Phi',Y}(dx|x',\theta',y) = p_{X|X',\Phi'}(dx|x',\theta')$ and so
\begin{displaymath}
\EE[X|X',\Phi',Y] = \EE[X|X',\Phi'].
\end{displaymath}
Thus, using Jensen's inequality for the conditional expectation again,
we can write
\begin{align*}
\EE\big\lVert \conditionalEE{X'-X}{Y}\big\rVert^2
&= \EE\Big\lVert
\conditionalEE[\Big]{
\conditionalEE[\big]{X'-X}{X',\capitalPhiTheta',Y}
}{Y}
\Big\rVert^2\\
&\leq
\EE\Big\lVert
\conditionalEE[\big]{X'-X}{X',\capitalPhiTheta',Y}
\Big\rVert^2\\
&=
\EE\Big\lVert
\conditionalEE[\big]{X'-X}{X',\capitalPhiTheta'}
\Big\rVert^2\\
&=
\EE\Big\lVert
\conditionalEE[\big]{(X'-X)\indicatorbraces{\capitalPhiTheta\neq\capitalPhiTheta'}}{X',\capitalPhiTheta'}
\Big\rVert^2,
\end{align*}
where the last equality follows from the fact that, by construction, on the event $\{\capitalPhiTheta=\capitalPhiTheta'\}$ we have $X'=X$. 	
Recall also from~\eqref{eq:law-of-Phi-given-Phi'-X'} that
$p_{\capitalPhiTheta|X',\capitalPhiTheta'}(d\theta|x',\theta') =
p_{\capitalPhiTheta|\capitalPhiTheta'}(d\theta|\theta')$.
Thus,
\begin{align*}
\MoveEqLeft[2]
\conditionalEE[\big]{(X'-X)\indicatorbraces{\capitalPhiTheta\neq\capitalPhiTheta'}}{X'=x',\capitalPhiTheta'=\theta'}\\
&=
\int_{\Theta}
\int_{\RR^d}
(x'-x)\indicatorbraces{\theta\neq\theta'}
p_{X|X',\capitalPhiTheta',\capitalPhiTheta}(dx|x',\theta',\theta)
p_{\capitalPhiTheta|X',\capitalPhiTheta'}(d\theta|x',\theta')\\
&=
\int_{\Theta}
\indicatorbraces{\theta'\neq\theta}
I(x',\theta',\theta)
p_{\capitalPhiTheta|\capitalPhiTheta'}(d\theta|\theta'),
\end{align*}	
where we denoted
\begin{align*}
I(x',\theta',\theta)
&\coloneqq
\int_{\RR^d}
(x'-x)
p_{X|X',\capitalPhiTheta',\capitalPhiTheta}(dx|x',\theta',\theta).
\end{align*}
If we also define
\begin{equation*}
Z(\theta')
\coloneqq
\int_{\Theta}
\indicatorbraces{\theta'\neq\theta}
p_{\capitalPhiTheta|\capitalPhiTheta'}(d\theta|\theta')
=
\conditionalPP{\capitalPhiTheta\neq\capitalPhiTheta'}{\capitalPhiTheta'=\theta'},
\end{equation*}
then Jensen's inequality implies that
\[
\Big\lVert
\int_{\Theta}
\indicatorbraces{\theta'\neq\theta}
I(x',\theta',\theta)
p_{\capitalPhiTheta|\capitalPhiTheta'}(d\theta|\theta')
\Big\lVert^2
\leq
Z(\theta')
\int_{\Theta}
\indicatorbraces{\theta'\neq\theta}
\Big\lVert I(x',\theta',\theta) \Big\lVert^2
p_{\capitalPhiTheta|\capitalPhiTheta'}(d\theta|\theta').
\]
Hence, by Fubini's theorem, we obtain
\begin{align*}
\MoveEqLeft[2]
\EE\Big\lVert
\conditionalEE[\big]{(X'-X)\indicatorbraces{\capitalPhiTheta\neq\capitalPhiTheta'}}{X',\capitalPhiTheta'}
\Big\rVert^2\\
&=
\int_{\Theta} \int_{\RR^d}
\Big\lVert
\int_{\Theta}
\indicatorbraces{\theta'\neq\theta}
I(x',\theta',\theta)
p_{\capitalPhiTheta|\capitalPhiTheta'}(d\theta|\theta')
\Big\lVert^2
\mu_{\theta'}(dx')
m'(d\theta')\\
&\leq
\int_{\Theta}
\int_{\RR^d}
Z(\theta')
\int_{\Theta}
\indicatorbraces{\theta'\neq\theta}
\Big\lVert I(x',\theta',\theta) \Big\lVert^2
p_{\capitalPhiTheta|\capitalPhiTheta'}(d\theta|\theta')
\mu_{\theta'}(dx') m'(d\theta')\\
&=
\int_{\Theta}
Z(\theta')
\int_{\Theta}
\indicatorbraces{\theta'\neq\theta}
\int_{\RR^d} \Big\lVert I(x',\theta',\theta) \Big\lVert^2 \mu_{\theta'}(dx')
p_{\capitalPhiTheta|\capitalPhiTheta'}(d\theta|\theta')
m'(d\theta').
\end{align*}
In this expression
the integral over $\RR^d$
can be estimated by a constant. Indeed, by \eqref{definition-of-X} and the definition of $\pi_{\theta',\theta}$,
\begin{align*}
\int_{\RR^d}
\Big\lVert
I(x',\theta',\theta)
\Big\lVert^2
\mu_{\theta'}(dx')
&=
\int_{\RR^d}
\Big\lVert
x' - \int_{\RR^d}
x
p_{X|X',\capitalPhiTheta',\capitalPhiTheta}(dx|x',\theta',\theta)
\Big\lVert^2
\mu_{\theta'}(dx')  \\
&=
\int_{\RR^d}
\Big\lVert
x' - \int_{\RR^d}
x
\pi_{\theta',\theta}(dx|x')
\Big\lVert^2
\mu_{\theta'}(dx')  \\
&=
2\BarCostQuadratic{\mu_\theta}{\mu_{\theta'}} \leq 2 D.
\end{align*}
Recalling the definition of $Z(\theta')$,
we see that in total
the second summand on the right-hand side of~\eqref{eq:goal-T-bar-minus-Bedlewo-two-summands}
can be estimated by
\begin{align*}
\MoveEqLeft[4]
\EE\Big\lVert
\conditionalEE[\big]{(X'-X)\indicatorbraces{\capitalPhiTheta\neq\capitalPhiTheta'}}{X',\capitalPhiTheta'}
\Big\rVert^2\\
&\leq
2D
\int_{\Theta}
Z(\theta')
\int_{\Theta}
\indicatorbraces{\theta'\neq\theta}
p_{\capitalPhiTheta|\capitalPhiTheta'}(d\theta|\theta')
m'(d\theta')\\
&=
2D
\int_{\Theta}
Z(\theta')^2
m'(d\theta')\\
&=
2D
\int_{\Theta}
\conditionalPP{\capitalPhiTheta\neq\capitalPhiTheta'}{\capitalPhiTheta'=\theta'}^2
m'(d\theta')
= 2D \EE \conditionalPP{\capitalPhiTheta\neq\capitalPhiTheta'}{\capitalPhiTheta'}^2\\
&\leq 4 D \RelativeEntropy{m'}{m},
\end{align*}
where in the last step we used that by definition the coupling $(\capitalPhiTheta,\capitalPhiTheta')$ satisfies \eqref{eq:Marton-coupling}.

Together with~\eqref{eq:goal-T-bar-minus-Bedlewo-two-summands}  and~\eqref{eq:goal-T-bar-minus-Bedlewo-first-summand-done},
this yields
\[
\BarCostQuadratic{\mu_{\mix}}{\nu}
\leq
2C \RelativeEntropy{\nu}{\mu_{\mix}}
+ (-2C + 4D) \RelativeEntropy{m'}{m}.
\]
By~\eqref{eq:observation-entropy-m'-wrt-m} we have
$0\leq \RelativeEntropy{m'}{m}
\leq \RelativeEntropy{\nu}{\mu_{\mix}} < \infty$.
This implies~\eqref{eq:goal-T-bar-minus-Bedlewo} and completes the proof of the theorem.
\end{proof}

\begin{proof}[Proof of Corollary~\ref{cor:stability-of-convex-log-Sobolev}]
This follows directly from Theorem~\ref{thm:stability-of-Tbar_2^-}
and the implication chain~\eqref{eq:implication-chain-convex-setting}
(cf.\ Theorem~\ref{thm:characterization-from-Kantorovich-duality}).
\end{proof}

Let us also explain the modification needed in the above argument in order to prove Theorem \ref{thm:stability-of-Tbar_c^-};
Corollary~\ref{cor:stability-of-Tbar_c^--and-convex-mSLI-R^d}
is proved in the next subsection.

\begin{proof}[Sketch of proof of Theorem \ref{thm:stability-of-Tbar_c^-}]
We only indicate the most important changes one needs to make
to the proof of Theorem~\ref{thm:stability-of-Tbar_2^-}. Thanks to the lower bound on the cost $c$, the $\bfTbar^-_c$ inequality implies the convex Poincar\'e inequality, so just as in the proof of Theorem~\ref{thm:stability-of-Tbar_2^-} we may reduce the problem to the case of finite $\Theta$.
Denote $C \coloneqq \sup_{\theta\in\Theta}
\ConstantTransportTbarGeneralMinus{c}{\mu_\theta}$,
$D\coloneqq D_{\Theta}$,
fix $\nu\in\mathcal{P}_1(\RR^d)$ which is absolutely continuous with respect to $\mu_{\mix}$,
and introduce $f$, $f_\theta$, $g$, $\nu_{\theta}$, and $m'$
as in the proof of Theorem~\ref{thm:stability-of-Tbar_2^-}.

Our goal is to show that
\begin{equation}
\label{eq:goal-T-bar-minus-Bedlewo-general-cost}
\BarCostGeneral{2c(\cdot/2)}{\mu_{\mix}}{\nu} \leq \max\{ C, 4K D\} \RelativeEntropy{\nu}{\mu_{\mix}}.
\end{equation}
To this end we define an appropriate coupling of $\mu_{\mix}$ and $\nu$,
by introducing the random variables $\capitalPhiTheta, \capitalPhiTheta', X, X', Y$
similarly as in the proof of Theorem~\ref{thm:stability-of-Tbar_2^-},
with the caveats that now we demand that:
\begin{itemize}
\item $\pi_{\theta',\theta}(dx',dx)$ is a coupling of $\mu_{\theta'}(dx')$ and $\mu_{\theta}(dx)$
minimizing the weak cost $\BarCostGeneral{2,p}{\mu_\theta}{\mu_{\theta'}}$
associated with the cost function $x\mapsto \abs{x}^2+\abs{x}^p$;
\item $\pi_{\theta'}(dy,dx')$ is a coupling of $\nu_{\theta'}(dy)$ and $\mu_{\theta'}(dx')$
such that $\BarCostGeneral{c}{\mu_{\theta'}}{\nu_{\theta'}}$ is minimal;
\end{itemize}
the existence of such couplings
follows by Lemma \ref{lem:minimizer-existence}.

Since $(X,Y)$ is a coupling of $\mu_{\mix}$ and $\nu$, we have
(by convexity and the assumption about the cost function $c$)
\begin{align}
\label{eq:goal-T-bar-minus-Bedlewo-two-summands-general-cost}
 \BarCostGeneral{2c(\cdot/2)}{\mu_{\mix}}{\nu}
&\leq 2\EE c\Bigl( \frac{1}{2} \bigl( Y - \conditionalEE{X}{Y}\bigr) \Bigr)\\
& \leq  \EE c\bigl(  Y - \conditionalEE{X'}{Y} \bigr)
+  \EE c\bigl( \conditionalEE{X' - X}{Y} \bigr)\nonumber\\
& \leq  \EE c\bigl(  Y - \conditionalEE{X'}{Y} \bigr)\nonumber\\
&\quad
+ K \EE \bigl\lVert \conditionalEE{X' - X}{Y} \bigr\rVert^2
+ K \EE \bigl\lVert \conditionalEE{X' - X}{Y} \bigr\rVert^p
\nonumber.
\end{align}
The first summand on the right-hand side can be estimated similarly as in the proof
of Theorem~\ref{thm:stability-of-Tbar_2^-}
(using the fact that
$\pi_{\theta'}(dy,dx')$ is a coupling of $\nu_{\theta'}(dy)$ and $\mu_{\theta'}(dx')$
minimizing $\BarCostGeneral{c}{\mu_{\theta'}}{\nu_{\theta'}}$):
\begin{align*}
 \EE c\bigl(  Y - \conditionalEE{X'}{Y} \bigr)\
&\leq \int_{\Theta}
\BarCostGeneral{c}{\mu_{\theta'}}{\nu_{\theta'}}
m'(d\theta')\\
&\leq
C \int_{\Theta} \RelativeEntropy{\nu_{\theta'}}{\mu_{\theta'}} m'(d\theta')
\leq
C \RelativeEntropy{\nu}{\mu_{\mix}} - C\RelativeEntropy{m'}{m}.
\end{align*}


As for the two other summands on the right-hand side of~\eqref{eq:goal-T-bar-minus-Bedlewo-two-summands-general-cost},
we also proceed similarly as the proof of Theorem~\ref{thm:stability-of-Tbar_2^-},
by first observing that by Jensen's inequality and Fubini's theorem, if $r\in\{2,p\}$, then
\begin{align*}
\MoveEqLeft[4]
\EE\big\lVert \conditionalEE{X'-X}{Y}\big\rVert^r
\leq
\EE\Big\lVert
\conditionalEE[\big]{(X'-X)\indicatorbraces{\capitalPhiTheta\neq\capitalPhiTheta'}}{X',\capitalPhiTheta'}
\Big\rVert^r\\
&\leq
\int_{\Theta}
Z(\theta')^{r-1}
\int_{\Theta}
\indicatorbraces{\theta'\neq\theta}
\int_{\RR^d} \Big\lVert I(x',\theta',\theta) \Big\lVert^r \mu_{\theta'}(dx')
p_{\capitalPhiTheta|\capitalPhiTheta'}(d\theta|\theta')
m'(d\theta')
\end{align*}
(with  $I(x',\theta',\theta)$, $Z(\theta')$
as in the proof of Theorem~\ref{thm:stability-of-Tbar_2^-}).
For $r\in\{2,p\}$,
\begin{align*}
\MoveEqLeft[4]\int_{\RR^d} \Big\lVert I(x',\theta',\theta) \Big\lVert^r \mu_{\theta'}(dx')\\
&\leq
\int_{\RR^d}
\Big\lVert x' - \int_{\RR^d} x \pi_{\theta',\theta}(dx|x') \Big\lVert^2
+ \Big\lVert x' - \int_{\RR^d} x \pi_{\theta',\theta}(dx|x') \Big\lVert^p
\mu_{\theta'}(dx')  \\
&=\BarCostGeneral{2,p}{\mu_\theta}{\mu_{\theta'}} \leq  D.
\end{align*}
Thus, after recalling the definition of $Z(\theta')$, we arrive at
\begin{align*}
 \MoveEqLeft[4]
 \EE \bigl\lVert \conditionalEE{X' - X}{Y} \bigr\rVert^2
+ \EE \bigl\lVert \conditionalEE{X' - X}{Y} \bigr\rVert^p  \\
&\leq
D \int_{\Theta} Z(\theta')^2 +Z(\theta')^p m'(d\theta')\\
&=
D
\int_{\Theta}
\conditionalPP{\capitalPhiTheta\neq\capitalPhiTheta'}{\capitalPhiTheta'=\theta'}^2
+ \conditionalPP{\capitalPhiTheta\neq\capitalPhiTheta'}{\capitalPhiTheta'=\theta'}^p
m'(d\theta')\\
&\leq
 2D \EE \conditionalPP{\capitalPhiTheta\neq\capitalPhiTheta'}{\capitalPhiTheta'}^2
 \leq 4 D \RelativeEntropy{m'}{m},
\end{align*}
where in the last two steps we used $p\geq 2$ and the definition
of the coupling $(\capitalPhiTheta,\capitalPhiTheta')$.

In total, \eqref{eq:goal-T-bar-minus-Bedlewo-two-summands-general-cost}
yields
\[
 \BarCostGeneral{2c(\cdot/2)}{\mu_{\mix}}{\nu}
\leq
C \RelativeEntropy{\nu}{\mu_{\mix}}
+ (-C + 4KD) \RelativeEntropy{m'}{m}.
\]
Recalling that
$\RelativeEntropy{m'}{m} \leq \RelativeEntropy{\nu}{\mu_{\mix}}$
implies~\eqref{eq:goal-T-bar-minus-Bedlewo-two-summands-general-cost}
and completes the proof of the theorem.
\end{proof}


\subsection{Equivalence of \texorpdfstring{$\bfTbar_c^-$}{bar(T)-theta-minus}
and convex log-Sobolev}

Throughout this subsection
we assume that $c\colon\RR^d \to [0,\infty)$
is a convex function such that $c(0)=0$
and $\lim_{\abs{x} \to \infty} c(x)/\abs{x} = \infty$.
Under these assumptions the convex conjugate $c^*\colon\RR^d\to [0,\infty)$
(which is is defined by~\eqref{eq:definition-Legendre-transform})
is finite everywhere,
$\lim_{\abs{x} \to \infty} c^*(x)/\abs{x} = \infty$,
and $c^{**} = c$.

Given $\lambda >0 $ and $\varphi\colon\RR^d \to \RR$, define
\begin{equation}
\label{eq:definition-R_c^lambda}
R_c^\lambda \varphi(x)
= \inf_{p \in \mathcal{P}_1(\RR^d)} \Bigl\{ \int_{\RR^d} \varphi(y)dp(y)
+ \lambda c\bigl( x - \int_{\RR^d}ydp(y) \bigr) \Bigr\}, \qquad x \in \RR^d.
\end{equation}
Moreover,
 given $t>0$ we define the  operator $Q_t = Q_t^c$
acting on Lipschitz functions $\varphi$ as in~\eqref{eq:definition-of-Q_t};
we also denote $Q_0 \varphi \coloneqq \varphi$.
It is well known (see~\cite[Chapter 3.3.2]{evans})
that if $\varphi$ is Lipschitz,
then the map $(t,x)\to Q_t \varphi(x)$ is Lipschitz on $[0,\infty)\times\RR^d$
and the Hamilton--Jacobi equation
\begin{equation}
\label{eq:Hamilton-Jacobi-equation}
 \frac{\partial}{\partial t}Q_t\varphi(x)
= -c^*\Bigl(\nabla_xQ_t\varphi(x)\Bigr)
\end{equation}
is satisfied almost everywhere with respect to the Lebesgue measure
 on $(0,\infty)\times \RR^d$.
Let us also observe that $(Q_t\varphi)^* = \varphi^* + tc^*$.
In particular, if $\varphi$ is convex and $c^*$ is strictly convex,
then so is $(Q_t\varphi)^*$.
Our assumptions on $c$
further imply
that $Q_t\varphi = (Q_t \varphi)^{**}$ is a $C^1$ function
in the $x$ variable for each $t>0$, see \cite[Theorem 26.3]{MR274683}.

We say that $\mu \in \mathcal{P}_1(\RR^d)$
satisfies the $(\tau)$-log-Sobolev inequality with constants $C, \lambda > 0$
and cost function $c$ (in short, $(\tau)-\mathrm{LSI}_c(\lambda,C)$)
if for all (not necessarily convex) $\varphi:\RR^d \to \RR$ with $\int_{\RR^d} \varphi\exp(\varphi) d\mu < \infty$,
\begin{equation}\label{eq: tau-log-sobolev-inequality}
\Ent_\mu(e^\varphi)
\le C \int_{\RR^d}(\varphi-R_c^\lambda \varphi)e^\varphi d\mu.
\end{equation}

The following theorem, obtained along the lines of \cite[Theorem 8.8]{grst-Kantorovich-duality} (see also \cite[Proposition 3.1]{MR3456588}), will be our main tool in the proof of Corollary~\ref{cor:stability-of-Tbar_c^--and-convex-mSLI-R^d}.

\begin{theorem}
\label{thm:equivalence}
Assume that $\mu \in \mathcal{P}_1(\RR^d)$
and let $c\colon \RR^d \to [0,\infty)$ be a continuous convex function
such that $c(0)=0$ and $c(x)/\abs{x} \to \infty$ as $\abs{x} \to \infty$. Assume that there exist $M,\delta > 0$ such that $c(x) \le M\abs{x}^2$ for $\abs{x} \le \delta$. Moreover, assume that $c^*$ is strictly convex and there are constants $1<a\le A<\infty$ such that
$c^*(sx) \le s^{\frac{A}{A-1}}c^*(x)$ for any $x \in \RR^d, s \in [0,1]$ and
$c^*(sx) \le s^{\frac{a}{a-1}}c^*(x)$ for any $x \in \RR^d, s > 1$.
Then the following conditions are equivalent:
\begin{enumerate}[label=(\roman*)]
\item\label{item:equiv-transport}
 $\mu$ satisfies $\bfTbar_c^-$ with some constant $b>0$,
\item\label{item:equiv-tau-lsi}
$\int_{\RR^d} \exp(s\abs{x}) d\mu(x)<\infty$ for every $s>0$ and $\mu$ satisfies $(\tau)$-$\mathrm{LSI}_c(\lambda,C)$ for some $\lambda,C>0$,
\item\label{item:equiv-convex-lsi}
$\int_{\RR^d}  \exp(s\abs{x})  d\mu(x)<\infty$ for every $s>0$ and there exists $C'>0$ such that for all convex, Lipschitz
and bounded from below $\varphi \in C^1(\RR^d,\RR)$,
\[
\Ent_\mu(e^\varphi) \le C'
\int_{\RR^d} c^*\bigl(\nabla \varphi\bigr)e^\varphi d\mu,
\]
\item\label{item:equiv-hypercontractive}
$\int_{\RR^d} \exp(s\abs{x}) d\mu(x) <\infty$ for every $s>0$ and there exists $C''>0$ such that for every $t \ge 0$ and $t_0 \le C''(A-1)$
and all convex, Lipschitz and bounded from below $\varphi \in C^1(\RR^d,\RR)$,
\[
\norm{e^{Q_t\varphi}}_{k(t)}
\le \|e^\varphi\|_{k(0)},
\]
where
$k(t) =
\big(1 + \frac{t-t_0}{C''(A - 1)}\big)^{A - 1} \indicatorbraces{t \le t_0}
+ \big(1+ \frac{t-t_0}{C''(a - 1)}\big)^{a-1}\indicatorbraces{t>t_0}$.
\end{enumerate}
The dependence of constants in the implications
\ref{item:equiv-transport} $\Rightarrow$
\ref{item:equiv-tau-lsi} $\Rightarrow$
\ref{item:equiv-convex-lsi} $\Rightarrow$
\ref{item:equiv-hypercontractive} $\Rightarrow$
\ref{item:equiv-transport}
is explicit in the proof.
\end{theorem}

\begin{proof}
We start with the proof of the implication~\ref{item:equiv-transport} $\Rightarrow$ \ref{item:equiv-tau-lsi}. As for the integrability condition, let us take any $s>0$ and define $f(x) = c\abs{x}$. By Theorem \ref{thm:duality-from-Kantorovich-duality} (and $\mu \in \mathcal P_1(\RR^d)$), the function $\exp(C^{-1} Q_1f)$ is $\mu-$integrable for some constant $C>0$.
Now, since
\[
Q_1f(x) = \inf_{y \in \RR^d} \big\{ s|y| + c(x-y) \big\}, \quad x \in \RR^d
\]
and $c(x)/\abs{x} \to \infty$, we see that for large $|x|$ (depending on $s$), the infimum is at least $\frac{s}{2}\abs{x}$. Since $s>0$ was arbitrary, it proves the asserted exponential integrability.
Now, let us fix some $\varphi\colon\RR^d \to \RR$
satisfying $\int_{\RR^d} \varphi e^\varphi d\mu < \infty$. If $\int_{\RR^d} \varphi d\mu = -\infty$, then $\varphi - R_c^\lambda \varphi \equiv \infty$ and the inequality \eqref{eq: tau-log-sobolev-inequality} is trivially satisfied. Hence, we may and will assume that $\int_{\RR^d}\varphi d\mu > -\infty$.
The measure $\nu_\varphi$ given by
\[
 d\nu_\varphi = \frac{e^\varphi}{\int_{\RR^d} e^\varphi d\mu}d\mu
\]
is well defined. Moreover, $\nu_\varphi \in \mathcal P_1(\RR^d)$. Indeed,
\[
\int_{\RR^d} |x| e^{\varphi(x)}\mu(dx)
\le \int_{\{\varphi < 0\}}|x|\mu(dx) + \int_{\{0\le\varphi(x) \le |x|\}} |x|e^{|x|}\mu(dx) + \int_{\{\varphi(x) > |x|\}} \varphi e^{\varphi} d\mu.
\]
The first integral is finite because $\mu \in \mathcal P_1(\RR^d)$;
the second one due to already proved exponential integrability of $\mu$
and the finiteness of the third one follows by our assumption on $\varphi$.
Now, note that
by Jensen's inequality,
\begin{align*}
\RelativeEntropy{\nu_\varphi}{\mu}
 &=
\int_{\RR^d} \log \Bigl( \frac{d\nu_\varphi}{d\mu}\Bigr)d\nu_\varphi
= \int_{\RR^d} \varphi d\nu_\varphi
- \int_{\RR^d} \log\Bigl(\int_{\RR^d}e^\varphi d\mu\Bigr)\frac{e^\varphi}{\int_{\RR^d}e^\varphi d\mu}d\mu \\
&= \int_{\RR^d}\varphi d\nu_\varphi - \log\Bigl(\int_{\RR^d} e^\varphi d\mu\Bigr)
\le \int_{\RR^d} \varphi d\nu_\varphi - \int_{\RR^d} \varphi d\mu.
\end{align*}
Hence, if $\pi$ is any coupling of $\nu_\varphi$ and $\mu$
given by $\pi(dx,dy) = p_x(dy)\nu_\varphi (dx)$,
we get
\[
\RelativeEntropy{\nu_\varphi}{\mu}
\le \int_{\RR^d} (\varphi(x)-\varphi(y))\pi(dx,dy)
= \int_{\RR^d}\Bigl(\int_{\RR^d}(\varphi(x)-\varphi(y))p_x(dy)\Bigr)\nu_\varphi(dx).
\]

From the definition of $R_c^\lambda$ (see~\eqref{eq:definition-R_c^lambda}),
 for any $\lambda > 0$ we have
\begin{align*}
\int_{\RR^d}\bigl(\varphi(x)-\varphi(y)\bigr)p_x(dy)
&= \varphi(x) - \int_{\RR^d}\varphi(y)p_x(dy)\\
&\le \varphi(x) - R_c^\lambda \varphi(x) + \lambda c\Bigl(x-\int_{\RR^d}y p_x(dy)\Bigr)
\end{align*}
and hence
\[
\RelativeEntropy{\nu_\varphi}{\mu}
\le \int_{\RR^d} (\varphi(x)-R_c^\lambda \varphi(x))\nu_\varphi(dx)
+ \lambda \int_{\RR^d} c\Bigl(x - \int_{\RR^d}y p_x(dy)\Bigr)\nu_\varphi(dx).
\]
 Optimizing over all $\pi$,
 we obtain from $\bfTbar_c^-$ (which holds with constant $b>0$),
\begin{align*}
\RelativeEntropy{\nu_\varphi}{\mu}
&\le \int_{\RR^d}(\varphi - R_c^\lambda \varphi)d\nu_\varphi
+ \lambda \calTbar_\theta(\mu|\nu_\varphi)\\
&\le \frac{1}{\int_{\RR^d}e^\varphi d\mu}\int_{\RR^d}(\varphi-R_c^\lambda \varphi )e^\varphi d\mu
+ \lambda b\RelativeEntropy{\nu_\varphi}{\mu}.
\end{align*}
It remains to note that
$\RelativeEntropy{\nu_\varphi}{\mu}\int_{\RR^d}e^\varphi d\mu
= \Ent_\mu(e^\varphi)$,
from which for any $\lambda \in (0,b^{-1})$,
 \[
\Ent_\mu(e^\varphi)
\le \frac{1}{1-\lambda b} \int_{\RR^d}(\varphi-R_c^\lambda \varphi)e^\varphi d\mu.
\]

We now prove \ref{item:equiv-tau-lsi}  $\Rightarrow$ \ref{item:equiv-convex-lsi}.
Let us fix any convex, Lipschitz and bounded from below function
 $\varphi \in C^1(\RR^d,\RR)$.
It is enough to prove $\varphi-R_c^\lambda \varphi \le C' c^*(\nabla \varphi)$
for some constant $C' = C'(\lambda, a, A) >0$.

Note that
$\varphi(x) - \varphi(y) \le \langle \nabla \varphi(x),x-y\rangle$
for any $x,y \in \RR^d$ due to convexity of $\varphi$,
so
\begin{align*}
\varphi(x)-R_c^\lambda\varphi(x)
&= \sup_{p \in \mathcal{P}_1(\RR^d)}\Bigg\{\int_{\RR^d}(\varphi(x)-\varphi(y))dp(y) - \lambda c \Bigl(
x-\int_{\RR^d}ydp(y)\Bigr)\Bigg\} \\ &\le \sup_{p \in \mathcal{P}_1(\RR^d)}
\Bigg\{ \int_{\RR^d} \langle \nabla \varphi(x),x-y\rangle dp(y) - \lambda c
\Bigl(x-\int_{\RR^d}ydp(y)\Bigr)\Bigg\} \\ &= \sup_{p \in \mathcal{P}_1(\RR^d)}
\Bigg\{\Big \langle \nabla \varphi(x), x - \int_{\RR^d}ydp(y)\Big\rangle -
\lambda c\Bigl(x - \int_{\RR^d}ydp(y)\Bigr)\Bigg\} \\ &\le \lambda \sup_{z \in
\RR}\Big\{ \Big \langle \frac{\nabla \varphi(x)}{\lambda}, z \Big \rangle -
c(z) \Big\} = \lambda c^*\Bigl( \frac{\nabla \varphi(x)}{\lambda}\Bigr)\\
&\le C'(\lambda,a,A) c^*(\nabla \varphi(x)),
\end{align*}
where in the last step we used the assumptions about scaling properties of $c^*$.
Hence,
\[
\Ent_\mu(e^\varphi) \le C' \int_{\RR^d} c^*\bigl(\nabla
\varphi\bigr)e^\varphi d\mu.
\]

We now prove \ref{item:equiv-convex-lsi} $\Rightarrow$ \ref{item:equiv-hypercontractive}
(with $C''=C'$).
The first scaling property of $c^*$ translates to
$c(x/s)\geq s^{A/(A-1)} c(x/s^{A/(A-1)})$ for $x\in\RR^d$, $s\in(0,1]$;
setting $t\coloneqq 1/s^{1/{A-1}}$ (and $x\coloneqq x/s$) yields
$t^A c(x) \geq c(tx)$ for $x\in\RR^d$, $t\geq 1$.
Similarly, $t^a c(x) \geq c(tx)$ for $x\in\RR^d$, $t\in [0,1]$.
Moreover, by setting $t\coloneqq 1/t$ in these two inequalities,
we obtain that
\begin{equation}
\label{eq:scaling-properties-in-terms-of-c}
\begin{cases}
    t^A c(x) \leq c(tx) \leq t^a c(x) & \quad \text{ for } x\in\RR^d, t\in[0,1],\\
    t^a c(x) \leq c(tx) \leq t^A c(x) & \quad \text{ for } x\in\RR^d, t\geq 1.
\end{cases}
\end{equation}
By convexity,
\[
c(x) \leq \frac{1}{d} \sum_{i=1}^d c(dx_i e_i) \leq d^{A-1} \sum_{i=1}^d c(x_i e_i),
\]
so due to~\eqref{eq:scaling-properties-in-terms-of-c}
and the assumption $c(x) \le M\abs{x}^2$ for $\abs{x}\le \delta$, we can find such $a_1,a_2>0$
(which depend on the dimension $d$ and the cost function $c$)
and $p>2$ that the function $\widetilde{c}\colon\RR \to [0,\infty)$ given by $\widetilde{c}(t)=\widetilde{c}(\abs{t}) = \max\{a_1\abs{t}^2, a_2\abs{t}^p\}$ satisfies $c(x) \le \sum_{i=1}^d\widetilde{c}(x_i)$.
From \cite[Theorem 1.2]{grsst-characterization} (see Theorem  \ref{thm:GRSST} below) we see that the uniform measure on $[-1,1]$ satisfies $\bfTbar_{\widetilde{c}}^-$ with some constant
and, by tensorization (see~\cite[Section~4]{grst-Kantorovich-duality}),
the uniform measure $\nu$ on $[-1,1]^d$ satisfies $\bfTbar_{c}^-$ with some constant.
Hence,
by the already proved implication \ref{item:equiv-transport} $\Rightarrow$ \ref{item:equiv-convex-lsi},
$\nu$ satisfies \ref{item:equiv-convex-lsi} on $\RR^d$ with some constant $\widetilde{C}'$ (which can depend also on $d$).
Again by tensorization of entropy (see, e.g., \cite[Proposition~5.6]{MR1849347}),
$\mu \otimes \nu$ satisfies the following version of \ref{item:equiv-convex-lsi} on $\RR^d \times \RR^d$:
\begin{align}
\label{eq:equivalence-tensorization-of-clsi}
    \Ent_{\mu\otimes\nu}(e^\Phi)
    &\le C'\int_{\RR^d\times\RR^d}  c^*\bigl(\nabla_x \Phi(x,y)\bigr)e^{\Phi(x,y)}d\mu\otimes\nu(x,y)\\
&\quad + \widetilde{C}' \int_{\RR^d\times\RR^d} c^*\bigl(\nabla_y \Phi(x)\bigr)e^{\Phi(x,y)} d\mu\otimes\nu(x,y),\nonumber
\end{align}
for appropriate $\Phi\colon\RR^d\to\RR$.
Take any $\varphi\colon\RR^d\to\RR$ as in assumptions of \ref{item:equiv-convex-lsi}.
Now, for $\varepsilon > 0$ consider function $\Phi(x,y) = \varphi(x+\varepsilon y)$, $x,y \in \RR$.
By applying \eqref{eq:equivalence-tensorization-of-clsi},
we see that the convolution $\mu * \nu_\varepsilon$, where $\nu_\varepsilon(\cdot) = \nu(\cdot/\varepsilon)$, satisfies \ref{item:equiv-convex-lsi} with some constant $\widetilde{C}'_\varepsilon \to C'$ as $\varepsilon \to 0^+$.
Thus, if we prove \ref{item:equiv-hypercontractive} for $\mu * \nu_\varepsilon$,
then by dominated convergence theorem we will also obtain \ref{item:equiv-hypercontractive} for $\mu$.
Hence, we may and will assume that $\mu$ is absolutely continuous (for similar argument, see \cite[Proposition 3.1]{MR3456588}).

Fix any $t_0 \in (0,C'(A -1)]$
and define
$ F_\varphi (t) = \int_{\RR^d} \exp(k(t) Q_t\varphi)d\mu$,
$H_\varphi(t) = k(t)^{-1}\log(F_\varphi(t))$. Using absolute continuity of $\mu$, together with integrability properties of Lipschitz functions, one can show that $F_\varphi$ is locally Lipschitz on $(0,\infty)$
and hence $H_\varphi'(t)$ exists almost everywhere for $t>0$.
Then, calculations similar to those in \cite[Proposition 4.1]{gozlan-roberto-samson-hj-metric} show that for almost all $t>0$,
\[
H_\varphi'(t)
= \frac{k'(t)}{k^2(t)} \frac{1}{\int_{\RR^d} e^{k(t)Q_t\varphi}d\mu}\Bigl( \Ent_\mu(e^{k(t)Q_t\varphi}) +
\frac{k^2(t)}{k'(t)} \int_{\RR^d} \big(\frac{\partial}{\partial t}Q_{t}f\big)
e^{k(t)Q_t\varphi}d\mu\Bigr).
\]
We bound the entropy term by applying~\ref{item:equiv-convex-lsi} to the function $k(t)Q_tf$,
\begin{align*}
\Ent(e^{k(t)Q_t\varphi})
& \le C' \int_{\RR^d} c^*\Bigl( k(t)\nabla_xQ_t\varphi\Bigr)e^{k(t)Q_t\varphi}d\mu
\\
&\le - C' \Bigl(k(t)^{\frac{A}{A - 1}}\indicatorbraces{t \le t_0}
+ k(t)^{\frac{a}{a - 1}}\indicatorbraces{t > t_0}\Bigr)
\int_{\RR^d} \Bigl(\frac{\partial}{\partial t}Q_{t}\varphi\Bigr)
e^{k(t)Q_t\varphi}d\mu .
\end{align*}
Hence, from the very definition of $k$,
\[
H_\varphi'(t)
\le \frac{\int_{\RR^d} \big(\frac{\partial }{\partial t}Q_{t}\varphi\big) e^{k(t)Q_t\varphi}d\mu}{\int e^{k(t)Q_t\varphi}}\Bigl( 1- C' k'(t)\big(k(t)^\frac{2-A}{A - 1}\indicatorbraces{t \le t_0}
+ k(t)^\frac{2-a}{a-1}\indicatorbraces{t > t_0}\big)\Bigr)
= 0,
\]
so the function $H_\varphi$ is non-increasing.
In particular,
using $Q_\varepsilon\varphi\leq\varphi$ we obtain
\begin{align*}
\log\bigl(\|e^{Q_t\varphi}\|_{k(t)} \bigr)
= H_\varphi(t)
&\le \liminf_{\varepsilon \to 0^+} H_\varphi(\varepsilon)\\
&\leq \lim_{\varepsilon\to 0^+} \frac{\log\bigl(\int_{\RR^d} \exp(k(\varepsilon) \varphi) d\mu \bigr)}{k(\varepsilon)}
= \log\bigl(\|e^{\varphi}\|_{k(0)}\bigr)
\end{align*}
which proves our claim.

It remains to show that
\ref{item:equiv-hypercontractive} $\Rightarrow$ \ref{item:equiv-transport}.
Fix $t=1$ and $t_0 = C''(A - 1)$.
Then our inequality \ref{item:equiv-hypercontractive} reads
\[
\Bigl( \int_{\RR^d} e^{k(1)Q_1\varphi}d\mu \Bigr)^{\frac{1}{k(1)}}
= \|e^{Q_1\varphi}\|_{k(1)} \le \|e^\varphi\|_{k(0)}
= \exp\bigl( \int_{\RR^d} \varphi d\mu\bigr),
\]
with $k(0)=0$ and appropriate $k(1)$.
From the dual formulation of the weak transportation inequality
(see Theorem~\ref{thm:duality-from-Kantorovich-duality} above),
we conclude that $\mu$ satisfies $\bfTbar_c^-$ with $b = 1/k(1)$.
\end{proof}

\begin{proof}[Proof of Corollary~\ref{cor:stability-of-Tbar_c^--and-convex-mSLI-R^d}]
Fix $p\geq 2$.
The assumptions of Theorem~\ref{thm:stability-of-Tbar_c^-}  imply that
$c(x)/\abs{x}\to\infty$ as $\abs{x}\to\infty$
and that
for some $M,\delta>0$ we have $c(x)\leq M\abs{x}^2$ for $\abs{x}\leq \delta$,
so the assumptions of Theorem~\ref{thm:equivalence} are satisfied.
By Theorem~\ref{thm:equivalence},
the mixed components satisfy the transportation inequality $\bfTbar_c^-$ (with some constant depending only on $C_\Theta$ and other parameters describing the behaviour of $c$, $c^*$),
by Theorem~\ref{thm:stability-of-Tbar_c^-},
the mixture $\mu_{\mix}$ satisfies the $\bfTbar_{2c(\cdot/2)}^-$ (with a constant depending additionally on $D_\Theta$),
and -- again by Theorem~\ref{thm:equivalence} -- the mixture $\mu_{\mix}$ satisfies the convex modified log-Sobolev inequality
(note that $(2c(\cdot/2))^* = 2c^*(\cdot)$,
so the scaling can be absorbed into the constant and the function $c^*$
on the right-hand side of~\eqref{eq:definition-convex-modified-log-Sobolev} does not change).
\end{proof}

\subsection{Preliminaries: weak transportation inequalities on the real line}

For a probability measure on $\RR$
with cumulative distribution function $F_\mu$
define the function
\begin{displaymath}
U_\mu(x) =
\begin{cases}
F_\mu^{-1}(\frac{1}{2}e^{-|x|}) & \text{ if }  x \le 0,\\
 F_\mu^{-1}(1 - \frac{1}{2}e^{-|x|}) & \text{ if }  x \ge 0.
\end{cases}
\end{displaymath}
Here
$F_{\mu}^{-1}(u) \coloneqq \inf\{ t\in\RR : F_{\mu}(t) \geq u\}$,
$u\in[0,1]$, is the generalized inverse of $F_{\mu}$;
note that $F_{\mu}^{-1}(u)\in\RR$ for $u\in (0,1)$.
Observe also that, by right continuity of $F_{\mu}$ and the definition of $F_{\mu}^{-1}$,
\begin{equation}
\label{eq:F-vs-F-inverse}
F_{\mu}^{-1}(F_\mu(s)) \leq s,
\qquad
F_{\mu}(F_\mu^{-1}(u)) \geq u
\end{equation}
for any $s\in\RR$ and any $u\in(0,1)$
(for $u\in\{0,1\}$ the second inequality also holds
(trivially) with natural conventions
for the meaning of $F_{\mu}(\pm \infty$)).

The map $U_\mu\colon\RR\to\RR$
provides the optimal transportation plan
between the standard two-sided exponential distribution
and the measure $\mu$.
The following result has been proved
by Gozlan, Roberto, Samson, Shu, and Tetali~\cite{grsst-characterization}.

\begin{theorem}
\label{thm:GRSST}
Let $\mu$ be a probability measure on $\RR$,
with finite first moment and let $c\colon [0,\infty)\to[0,\infty)$
be a convex cost function such that
$c(t) = t^2$ for $t \in [0,t_0]$ for some $t_0>0$.
The following conditions are equivalent:
\begin{enumerate}[label=(\roman*)]
\item There exists $a > 0$ such that $\mu$ satisfies $\bfTbar_{c(a\abs{\,\cdot\,})}$ with constant 1.
\item There exists $b > 0$ such that for all $u > 0$,
\begin{equation}
\label{eq:GRSST-condition}
\sup_{x\in\RR} \bigl(U_\mu(x+u) - U_\mu(x)\bigr) \le \frac{1}{b}c^{-1}(u + t_0^2).
\end{equation}
\end{enumerate}
Moreover, $(i)$ implies $(ii)$ with $b= a\kappa_1$,
and $(ii)$ implies $(i)$ with $a=b\kappa_2$,
where $\kappa_i$ depend only on the cost $c$.
\end{theorem}

Let $M_\mu = F_{\mu}^{-1}(1/2)$ be the smallest median of $\mu$.
We shall rephrase  the condition
\begin{align}
\label{eq:GRSST-single-x}
U_\mu(x+u) - U_\mu(x) \le \frac{1}{b}c^{-1}(u + t_0^2)
\end{align}
from Theorem~\ref{thm:GRSST}
in terms of $F_\mu$ and $G_\mu \coloneqq 1 - F_\mu$
rather than $U_\mu$ or $F_\mu^{-1}$,
thus obtaining the following reformulation of Theorem \ref{thm:GRSST}.
It is better suited for working with mixtures,
since one may integrate both sides of the inequalities~\eqref{eq:tail-above-median} and~\eqref{eq:tail-below-median}.

\begin{proposition}
\label{prop:GRSST-rephrased}
 Let $\mu$ be a probability measure on $\RR$,
with finite first moment and let $c\colon [0,\infty)\to [0,\infty)$
be a convex cost function such that $c(t) = t^2$ for $t \le t_0$ for some $t_0>0$.
The following conditions are equivalent:
\begin{enumerate}[label=(\roman*)]
\item \label{item:GRSST-rephrased-item-i} There exists $a > 0$ such that $\mu$ satisfies $\bfTbar_{c(a\abs{\,\cdot\,})}$ with constant 1.
\item \label{item:GRSST-rephrased-item-ii} There exists $b > 0$ such that for all $u > 0$,
\begin{align}
\label{eq:tail-above-median}
G_\mu(t + \frac{1}{b}c^{-1}(u + t_0^2)) &\le e^{-u}G_\mu(t)
\quad \text{ for all } t \ge M_\mu,\\
\label{eq:tail-below-median}
F_\mu\bigl(t - \frac{1}{b}c^{-1}(u + t_0^2)\bigr) &\le e^{-u} F_\mu(t)
\quad \text{ for all } t < M_\mu.
\end{align}
\end{enumerate}
Moreover, (i) implies (ii) with $b= a\kappa_1$,
and (ii) implies (i) with $a=b\kappa_2$,
 where $\kappa_i$ depend only on the cost $c$.
\end{proposition}

\begin{proof}
	\textsc{Step 1}.	
Clearly, if~\eqref{eq:GRSST-condition} holds for all $u>0$,
then, for all $u>0$,
\begin{align}
\label{eq:GRSST-condition-positive-x}
\sup_{x \ge 0} \bigl(U_\mu(x+u) - U_\mu(x)\bigr) &\le \frac{1}{b}c^{-1}(u + t_0^2),\\
\label{eq:GRSST-condition-negative-part}
\sup_{x \le -u} \bigl(U_\mu(x+u) - U_\mu(x)\bigr) &\le \frac{1}{b}c^{-1}(u + t_0^2).
\end{align}
On the other hand, if~\eqref{eq:GRSST-condition-positive-x}
and~\eqref{eq:GRSST-condition-negative-part} hold for all $u > 0$,
then for fixed $u>0$ and $x \in [-u,0)$
 we can write
\begin{align*}
U_\mu(x+u) - U_\mu(x)
&\le (U_{\mu}(x+u) - U_{\mu}(0))  + (U_\mu(0) - U_\mu(x))  \\
&\le \frac{1}{b}c^{-1}(x+u + t_0^2)+  \frac{1}{b}c^{-1}(|x| + t_0^2)
\le \frac{2}{b}c^{-1}(u + t_0^2).
\end{align*}
Thus,~\eqref{eq:GRSST-condition-positive-x} and~\eqref{eq:GRSST-condition-negative-part}
imply~\eqref{eq:GRSST-condition} with $b$ replaced by $b/2$.

\textsc{Step 2}. Let us rephrase~\eqref{eq:GRSST-condition-positive-x}
and show that it is equivalent with~\eqref{eq:tail-above-median}.
Assume first that~\eqref{eq:tail-above-median} holds for a given $u>0$
(and for all $t\geq M_\mu$).
Fix any $x\geq 0$ and define  $t \coloneqq U_\mu(x) = F_{\mu}^{-1}(1-\frac{1}{2}e^{-x})$.
In particular, \eqref{eq:F-vs-F-inverse} yields
$F_\mu(t) \ge 1 - \frac{1}{2}e^{-x} \ge 1/2$,  i.e., $t \ge M_\mu$.
For our fixed $x\geq 0$, $u>0$,
the condition~\eqref{eq:GRSST-single-x}
holds if and only if
\begin{align*}
F_\mu^{-1}\bigl(  1 - \frac{1}{2}e^{-(x+u)} \bigr)
&\le t + \frac{1}{b}c^{-1}(u + t_0^2);\\
\intertext{by the definition of $F_{\mu}^{-1}$ and~\eqref{eq:F-vs-F-inverse}
this happens if and only if}
 1 - \frac{1}{2}e^{-(x+u)}
&\leq F_\mu\bigl(t + \frac{1}{b}c^{-1}(u + t_0^2)\bigr),\\
\shortintertext{i.e.,}
G_\mu\bigl(t + \frac{1}{b}c^{-1}(u + t_0^2)\bigr)
&\le \frac{1}{2}e^{-(x+u)}.
\end{align*}
Since $G_\mu(t) \leq \frac{1}{2}e^{-x} $,
it follows that for $x \ge 0$ and $u>0$, the condition
\eqref{eq:GRSST-single-x} is implied by
\begin{equation}
\label{eq:GRSST-positive-x-rephrased}
G_\mu\bigl(t + \frac{1}{b}c^{-1}(u + t_0^2)\bigr) \le e^{-u}G_\mu(t),
\end{equation}
where $t = U_\mu(x) \ge M_\mu$.
In particular, for given $u>0$, \eqref{eq:tail-above-median} (for all $t\geq M_\mu$) implies
that~\eqref{eq:GRSST-single-x} holds for all $x\geq 0$,
i.e., that~\eqref{eq:GRSST-condition-positive-x} holds.

On the other hand,
fix $u>0$ and assume that~\eqref{eq:GRSST-condition-positive-x} holds, i.e.,
\eqref{eq:GRSST-single-x} holds for all $x \ge 0$.
Take any $t \ge M_\mu$.
Assume for now that $F_{\mu}(t) < 1$
and define $x \ge 0$ by the condition
$F_\mu(t) = 1 - \frac{1}{2}e^{-x}$;
let also $t_\ast \coloneqq U_\mu(x)$.
Then $t_\ast \le t$ and $F_\mu(t_\ast) = F_\mu(t)$.
Using \eqref{eq:GRSST-single-x} and \eqref{eq:F-vs-F-inverse}
 gives, for any $u>0$,
\begin{align*}
G_\mu\bigl(t+ \frac{1}{b}c^{-1}(u + t_0^2)\bigr)
& \le G_\mu\bigl(t_\ast + \frac{1}{b}c^{-1}(u + t_0^2)\bigr)
 = G_\mu\bigl(U_{\mu}(x) + \frac{1}{b}c^{-1}(u + t_0^2)\bigr)\\
& \le G_{\mu}\bigl(U_{\mu}(x+u)\bigr) = 1 - F_{\mu}\bigl(F_{\mu}^{-1}(1-\frac{1}{2}e^{-x-u}) \bigr)\\
& \leq  \frac{1}{2}e^{-x-u} = e^{-u}G_\mu(t).
\end{align*}
Since the inequality
$G_\mu\bigl(t+ \frac{1}{b}c^{-1}(u + t_0^2)\bigr)
 \leq  e^{-u}G_\mu(t)$
also holds (trivially) if $F_{\mu}(t) = 1$,
it follows that~\eqref{eq:tail-above-median} holds (for all $t\geq M_{\mu}$ and for our fixed $u>0$).

Summarizing, we have thus proved that~\eqref{eq:GRSST-condition-positive-x} holds for all $u > 0$,
if and only if~\eqref{eq:tail-above-median} holds for all $u>0$.

\textsc{Step 3}.
Let us show that~\eqref{eq:tail-below-median}
implies~\eqref{eq:GRSST-condition-negative-part}.
Assume first that~\eqref{eq:tail-below-median} holds for a given $u>0$ (and for all $t<M_{\mu}$).
Fix any $x< 0$ such that $x + u \le 0$.
Let $s\coloneqq  U_\mu(x) = F_{\mu}^{-1}(\frac{1}{2}e^{x})$.
By~\eqref{eq:F-vs-F-inverse} we have $F_\mu(s) \ge \frac{1}{2}e^x$;
 moreover, $s \le F_{\mu}^{-1}(\frac{1}{2}) = M_\mu$.
For our fixed $x$, $u$, the condition~\eqref{eq:GRSST-single-x} holds if and only if
\begin{equation*}
F_{\mu}^{-1}\bigl(\frac{1}{2} e^{x+u}\bigr)
 \leq s + \frac{1}{b}c^{-1}(u + t_0^2);
\end{equation*}
by the definition of $F_{\mu}^{-1}$ and~\eqref{eq:F-vs-F-inverse}
this happens if and only if
\begin{displaymath}
\frac{1}{2} e^{x+u} \leq F_\mu\bigl(s + \frac{1}{b}c^{-1}(u + t_0^2)\bigr).
\end{displaymath}
Note that if $x$, $u$ are such that $s +\frac{1}{b}c^{-1}(u + t_0^2) \ge M_\mu$,
 then the above inequality is satisfied trivially
since its left-hand side is smaller than $1/2$.
If $s +\frac{1}{b}c^{-1}(u + t_0^2) < M_\mu$,
then the above inequality is implied by
\begin{displaymath}
 F_\mu(s) \leq e^{-u} F_\mu\bigl(s + \frac{1}{b}c^{-1}(u + t_0^2)\bigr),
\end{displaymath}
since $F_\mu(s) \geq \frac{1}{2}e^{x} $ by~\eqref{eq:F-vs-F-inverse}.
Thus, by a change of variables, the condition~\eqref{eq:tail-below-median}
(for fixed $u>0$ and for all $t <  M_\mu$) implies that~\eqref{eq:GRSST-single-x} holds for all $x\leq -u$, i.e., that~\eqref{eq:GRSST-condition-negative-part} holds (for our given $u>0$).

Therefore, if~\eqref{eq:tail-below-median} holds for all $u>0$, then~\eqref{eq:GRSST-condition-negative-part} holds for all $u>0$.

\textsc{Step 4}.
Next, we show that~\eqref{eq:GRSST-condition}
implies~\eqref{eq:tail-below-median}.
Fix $u>0$ and assume that~\eqref{eq:GRSST-single-x} holds for all $x\in\RR$.
Take any $t < M_\mu$.
Let $s \coloneqq t - \frac{1}{b}c^{-1}(u + t_0^2)$.
Assume for now that $F_{\mu}(s) > 0 $
and define $x$  by the condition $F_\mu(s) = \frac{1}{2}e^{x}$;
note that $x < 0$.
Let $s_\ast \coloneqq U_\mu(x)$.
Then $s_\ast \le s$,
$F_{\mu}(s_\ast) = F_{\mu}(s)$,
and
\[
U_\mu(x+u) \le U_\mu(x) + \frac{1}{b}c^{-1}(u + t_0^2)
= s_\ast + \frac{1}{b}c^{-1}(u + t_0^2)
\le t.
\]
In combination with $t < M_\mu$
this implies that
$x+u \le 0$ and $F_{\mu}^{-1}(\frac{1}{2} e^{x+u}) \leq t$.
By~\eqref{eq:F-vs-F-inverse}, this implies in turn
that $\frac{1}{2} e^{x+u} \leq F_\mu(t) $,
i.e., $F_\mu(s)\leq e^{-u} F_\mu(t)$;
in the case $F_\mu(s) = 0$ such an inequality holds trivially.
This means that~\eqref{eq:tail-below-median} holds (for all $t<M_{\mu}$ and our given $u>0$).

Thus, \eqref{eq:GRSST-condition} implies that \eqref{eq:tail-below-median} holds for all $u>0$.

\textsc{Step 5}.
We are ready to summarize everything.
By Step~1,
the condition~\eqref{eq:GRSST-condition} implies~\eqref{eq:GRSST-condition-positive-x},
which is equivalent to~\eqref{eq:tail-above-median} (by Step~2);
by Step~4,
the condition~\eqref{eq:GRSST-condition} implies~\eqref{eq:tail-below-median}.
This completes the proof of the implication \ref{item:GRSST-rephrased-item-i}~$\implies$~\ref{item:GRSST-rephrased-item-ii}.

By Steps~2 and~3,
\eqref{eq:tail-above-median} and~\eqref{eq:tail-below-median}
imply~\eqref{eq:GRSST-condition-positive-x} and~\eqref{eq:GRSST-condition-negative-part}, respectively.
By Step~1, these two conditions imply~\eqref{eq:GRSST-condition} (with $b/2$ in place of $b$).
This completes the proof of the implication \ref{item:GRSST-rephrased-item-ii}~$\implies$~\ref{item:GRSST-rephrased-item-i}.
\end{proof}

\subsection{Proof of stability on the real line for general costs}

We can now prove the stability result of Theorem~\ref{thm:stability-of-Tbar_c-real-line}.

\begin{proof}[Proof of Theorem~\ref{thm:stability-of-Tbar_c-real-line}]
Denote $E \coloneqq E_\Theta$ and $\mu\coloneqq \mu_{\mix}$.
Below the constants $\kappa_1, \kappa_2,  \dots$
depend only on the cost $c$,
while $C$, $C'$, etc.\
are allowed to depend on $E$, $a$, and the cost $c$;
 $M_\nu = F_{\nu}^{-1}(1/2)$ stands for the smallest median of $\nu$.
By our assumptions and Proposition~\ref{prop:GRSST-rephrased},
for every $\theta\in\Theta$
the mixed component $\mu_\theta$ satisfies the conditions~\eqref{eq:tail-above-median}
and~\eqref{eq:tail-below-median} with constant $b= \kappa_1 a$.
It suffices to verify that the mixture
 satisfies the conditions~\eqref{eq:tail-above-median}
and~\eqref{eq:tail-below-median} from Proposition~\ref{prop:GRSST-rephrased}.
We will prove \eqref{eq:tail-above-median},
the other inequality is analogous.

Fix an arbitrary $\theta \in \Theta$.
By Corollary~\ref{cor:weak-transport-implies-convex-Poincare},
all mixed components satisfy the convex Poincar\'e inequality with constant $C = 1/(2a^2)$.
Consequently,
using the definition of $E$ and \cite[Lemma~2.1]{adamczak-strzelecki},
we see that for any $\theta'\in \Theta$,
\begin{equation*}
|M_{\mu_\theta} - M_{\mu_{\theta'}}|
\le |M_{\mu_\theta} - \int_\RR x \mu_{\theta}(dx)|
+ E
+  |M_{\mu_{\theta'}} - \int_\RR x \mu_{\theta'}(dx)|
\le 2\sqrt{2C} + E .
\end{equation*}
Setting $C' \coloneqq 2 \sqrt{2C} + E= 2/a + E$, we can rewrite this as
$M_{\mu_{\theta}} - C' \le M_{\mu_\theta'} \le  M_{\mu_{\theta}} + C'$, which means exactly that
\[
F_{\mu_\theta'}(M_{\mu_\theta} - C' - \varepsilon) < 1/2 \leq F_{\mu_\theta'}(M_{\mu_\theta} + C')
\]
for any $\varepsilon >0$.
By integration over $\theta'\in\Theta$ with respect to the mixing measure $m$,
it follows that also $M_{\mu_{\theta}} - C' \le M_{\mu} \le  M_{\mu_{\theta}} + C'$.

Now, take $C''\coloneqq 1 + C'b/c^{-1}(t_0^2)$, so that $\frac{C''-1}{b}c^{-1}(t_0^2) \ge C'$.
Then, for $t \ge M_{\mu}$ and $u>0$,
\begin{align*}
G_{\mu}(t + \frac{C''}{b}c^{-1}(u + t_0^2))
&= \int_\Theta G_{\mu_\theta}\bigl(t+\frac{C''}{b}c^{-1}(u + t_0^2)\bigr)m(d\theta) \\
&=  \int_\Theta G_{\mu_\theta}\bigl(t+\frac{C''-1}{b}c^{-1}(u + t_0^2) + \frac{1}{b}c^{-1}(u + t_0^2) \bigr)m(d\theta) \\
&\le e^{-u} \int_\Theta G_{\mu_\theta}\bigl(t+\frac{C''-1}{b}c^{-1}(u + t_0^2)\bigr) m(d\theta)\\
&\le e^{-u} \int_\Theta G_{\mu_\theta}(t) m(d\theta) = e^{-u} G_{\mu}(t),
\end{align*}
where in the first inequality we used that
$t +  \frac{C''-1}{b}c^{-1}(u+t_0^2) \ge M_{\mu} + C' \ge M_{\mu_\theta}$,
 which allows us to use \eqref{eq:tail-above-median} with $\mu_\theta$ instead of $\mu$
and $t +  \frac{C''-1}{b}c^{-1}(u+t_0^2)$ instead of $t$.
This proves \eqref{eq:tail-above-median} with $b$ replaced by $b/C''$.
Hence, $\mu$ satisfies $\bfTbar_{c(a'\abs{\cdot})}$
with $a' \coloneqq \kappa_2 b/C'' = \kappa_1\kappa_2 a/C''$,
where
\[
C''  = 1 + bC'/c^{-1}(t_0^2) = 1 + \kappa_3 a (2/a + E) = \kappa_4 + \kappa_3 aE.
\]
This completes the proof.
\end{proof}

%
%

\section{Proofs of stability of \texorpdfstring{$\bfT_1$}{T1} and related transportation inequalities}

\label{sec:proofs-T_1}

Throughout this section $(E,d)$ is a Polish space.

\subsection{Proof of Theorem~\ref{thm:stability-of-T_1-improved}}
The proof of Theorem~\ref{thm:stability-of-T_1-improved}
is based on a characterization of the $\bfT_1$ transportation inequality
due to Bobkov and Götze~\cite[Theorem~3.1]{bobkov-goetze}
and integral criteria obtained
by Djellout, Guillin, and Wu~\cite[Theorem~2.3]{djellout-guillin-wu}
(cf.\ \cite[Corollary 2.4]{bolley-villani}).
They state that the following conditions are equivalent
(cf.\ \cite[ Lemma~2.8]{bolley-villani}):
\begin{enumerate}[1.]
\item There exists $C>0$ such that $\mu$ satisfies the $\bfT_1$ inequality.
\item There exist $\delta>0$  and $x_0\in E$
such that $\int_E \exp(\delta d(x, x_0)^2 )d\mu(x) < \infty$.
\item There exists $\delta'>0$
such that $\int_{E}\int_E \exp(\delta' d(x,y)^2 )d\mu(x) d\mu(y) < \infty$.
\end{enumerate}
(In particular, if $\mu$ satisfies $\bfT_1$, then $\int_E d(x,x_0) d\mu(x) < \infty$ for any $x_0\in E$).
We provide the precise formulations with explicit constant dependence below.

\begin{fact}[{\cite[Theorem~2.1]{bobkov-goetze}}]
\label{lem:bobkov-goetze-exponential-integrability}
Let $\mu$ be  a probability measure on $E$
such that $\int_E d(x,x_0) d\mu(x) < \infty$
for some $x_0\in E$ (equivalently: for all $x_0\in E$).
Then $\mu$ satisfies $\bfT_1$ with constant $C$
if and only if
for every
$1$-Lipschitz function $F\colon E\to\RR$ with $\int_E F d\mu = 0$,
\begin{equation}
	\label{eq:bobkov-goetze-exponential-integrability}
	\int_E \exp(tF) d\mu \leq \exp(C t^2/2) \quad \text { for } t\in\RR.
\end{equation}
\end{fact}

\begin{lemma}[folklore]
\label{lem:DGW-Gaussian-trick-upgraded}
Let $\mu$ be  a probability measure on $E$
such that
\[
B(x_0) \coloneqq \int_E d(x,x_0) d\mu(x) < \infty
\]
for some $x_0\in E$ (equivalently: for all $x_0\in E$).
If $\mu$ satisfies
$\bfT_1$ with some constant $C\in(0,\infty)$,
then, for every $x_0\in E$ and $\lambda>0$,
\begin{equation*}
	\int_E \exp\bigl(\delta d(x,x_0)^2\bigr) d\mu(x)
	\leq \frac{\exp\bigl((1+\lambda)\delta B(x_0)^2\bigr)}{\sqrt{1- 2(1+\lambda^{-1})\delta C}}
	\quad \text{ for } 0\leq \delta < \frac{1}{2(1+\lambda^{-1})C}.
\end{equation*}
\end{lemma}

\begin{proof}
We follow the reasoning from \cite[p.~2704]{djellout-guillin-wu}.
Let $\gamma$ be the standard Gaussian measure on $\RR$
(with mean $0$ and variance $1$).
Take a $1$-Lipschitz function
$F\colon E\to\RR$ with $\int_E F d\mu = 0$.
and $a\in\RR$ such that $ a^2 C < 1$.
By Fubini's theorem and~\eqref{eq:bobkov-goetze-exponential-integrability},
\begin{align}
\label{eq:gaussian-trick}	
\int_E \exp(a^2 F^2/2) d\mu
& =  \int_E \int_{\RR}\exp(a t F) d\gamma(t) d\mu \\
&\leq  \int_{\RR} \exp(C a^2 t^2/2)  d\gamma(t) =
\frac{1}{\sqrt{1- a^2C}}.
\nonumber
\end{align}
Fix $x_0\in E$ and let
\[
F(x) \coloneqq d(x,x_0) - B(x_0) = d(x,x_0) - \int_E d(y,x_0) d\mu(y), \quad x\in E;
\]
then $F\colon E\to\RR$ is $1$-Lipschitz and has mean $0$.
Since for any parameter $\lambda >0$,
\begin{align*}
d(x,x_0)^2 = \bigl(F(x) + B(x_0)\bigr)^2
&=F(x)^2 + 2F(x)B(x_0) + B(x_0)^2\\
&\leq (1+\lambda^{-1})F(x)^2 + (1+\lambda) B(x_0)^2,	
\end{align*}
we see that \eqref{eq:gaussian-trick} yields
\begin{align*}
\int_E \exp\bigl(\delta d(x,x_0)^2\bigr) d\mu(x)
& \leq  \exp\bigl((1+\lambda)\delta B(x_0)^2\bigr)
\int_E \exp\bigl((1+\lambda^{-1})\delta F(x)^2\bigr) d\mu(x)\\
& \leq \frac{\exp\bigl((1+\lambda)\delta B(x_0)^2\bigr)}{\sqrt{1- 2(1+\lambda^{-1})\delta C}}
\end{align*}
provided that $0\leq \delta < (2(1+\lambda^{-1})C)^{-1}$.
This proves the assertion.
\end{proof}

The implication in the reverse direction is due to
Djellout, Guillin, and Wu~\cite{djellout-guillin-wu}.
We recall a version of their result
with sharper constants proved by  Bolley and~Villani~\cite{bolley-villani}.

\begin{lemma}[{\cite[Corollary~2.4]{bolley-villani}}]
\label{lem:BV-improved}
Let $\mu$ be  a probability measure on $E$.
If for some $x_0\in E$ and $\delta>0$,
\begin{equation*}
\int_E \exp(\delta d(x,x_0)^2) d\mu(x)  <\infty,
\end{equation*}
then $\mu$ satisfies $\bfT_1$ with constant
\begin{equation*}
\ConstantTransportTOne{\mu} \leq
\frac{1}{\delta} \Bigl( 1 + \log \int_E \exp(\delta d(x,x_0)^2) d\mu(x) \Bigr) .
\end{equation*}
\end{lemma}

\begin{proof}[Proof of Theorem~\ref{thm:stability-of-T_1-improved}]
Denote $C\coloneqq C_\Theta $ and $N \coloneqq N_\Theta$.
Fix any $x_0\in E$.
Take $s>1$ and denote
$\lambda \coloneqq 2sC/N^2$, $\delta \coloneqq 1/(N^2+2sC)$,
so that
\[
(1+\lambda) \delta = \frac{1}{N^2},
\qquad
2(1+\lambda^{-1})\delta C = \frac{1}{s}<1.
\]
Since by assumption $\mu_\theta$ satisfies $\bfT_1$
with constant $\ConstantTransportTOne{\mu_\theta} \leq C$,
Lemma~\ref{lem:DGW-Gaussian-trick-upgraded} implies that
\begin{equation*}
\int_E \exp\bigl(\delta d(x,x_0)^2\bigr) d\mu_\theta(x)
\leq \frac{\exp\bigl((1+\lambda)\delta B_{\theta}(x_0)^2\bigr)}{\sqrt{1- 2(1+\lambda^{-1})\delta C}}
=\frac{\exp\bigl( B_{\theta}(x_0)^2/N^2\bigr)}{\sqrt{1- 1/s}}.
\end{equation*}
Integrating this inequality over $\Theta$,
we obtain that
\begin{equation*}
\int_E  \exp\bigl(\delta d(x,x_0)^2\bigr) d\mu_{\textup{mix}}(x)
\leq \frac{1}{\sqrt{1-1/s}}
\int_{\Theta} \exp(B_{\theta}(x_0)^2/N^2) dm(\theta)
\leq \frac{2}{\sqrt{1-1/s}},
\end{equation*}
where in the last inequality we used the definition of $N=N_\Theta$.
By Lemma~\ref{lem:BV-improved},
this immediately implies
that the mixture $\mu_\mix$ satisfies the $\bfT_1$ inequality
with constant
\[
\frac{1}{\delta}\Bigl(1 + \log\Bigl(\frac{2}{\sqrt{1-1/s}}\Bigr) \Bigr) = (N^2 + 2sC)\Bigl(1 + \log 2 + \frac{1}{2} \log\Bigl(\frac{1}{1-1/s}\Bigr) \Bigr).
\]
Taking $s=2$ completes the proof.
\end{proof}

\begin{remark}
\label{rem:on-assumptions-in-T1-result}
Suppose that $\mu_{\mix}$ satisfies the $\bfT_1$ inequality
with  constant $K\coloneqq \ConstantTransportTOne{\mu_\mix}<\infty$. Then $B_{\mix}(x_0) \coloneqq \int_E d(x,x_0) d\mu_{\mix}(x) <\infty$.
Jensen's inequality and Lemma~\ref{lem:DGW-Gaussian-trick-upgraded} imply that for sufficiently small $\delta >0$,
\begin{align*}
\int_\Theta \exp\bigl(\delta B_\theta(x_0)^2\bigr) dm(\theta)
&=\int_\Theta  \exp\Bigl(\delta \bigl( \int_E d(x,x_0) d\mu_\theta(x) \bigr)^2\Bigr)  dm(\theta)\\
&\leq \int_E \exp\bigl(\delta  d(x,x_0)^2\bigr) d\mu_{\mix}(x) <\infty,
\end{align*}
so necessarily the assumption $N_\Theta<\infty$ from  Theorem~\ref{thm:stability-of-T_1-improved} is satisfied.
More precisely,
one can use Lemma~\ref{lem:DGW-Gaussian-trick-upgraded} with, e.g., $\lambda\coloneqq 1$, $\delta \coloneqq 1/(8(K + B_\mix(x_0)^2))$
and get
\begin{align*}
\int_\Theta \exp\bigl(\delta B_\theta(x_0)^2\bigr) dm(\theta)
&\leq \int_E \exp\bigl(\delta  d(x,x_0)^2\bigr) d\mu_{\mix}(x)\\
&\leq \frac{\exp(2\delta B_\mix(x_0)^2)}{\sqrt{1- 4\delta K}}
\leq \sqrt{2}\exp(1/4) \leq 2,
\end{align*}
so $N_\Theta^2 \leq 8(K + B_\mix(x_0)^2)$.
\end{remark}

\subsection{Proof of Theorem~\ref{thm:stability-of-T_1_alpha}}

Throughout this subsection we assume that the function $\alpha\colon[0,\infty)\to[0,\infty)$ is fixed
and satisfies the assumptions of Theorem~\ref{thm:stability-of-T_1_alpha}.
We will need the following versions of Lemmas~\ref{lem:DGW-Gaussian-trick-upgraded}
and~\ref{lem:BV-improved}.
They follow from general results of Gozlan~\cite{gozlan-integral-criteria}
and Gozlan and L\'eonard~\cite{gozlan-leonard},
but we provide brief sketches of proofs since their statements are less canonical.

\begin{lemma}
\label{lem:GL-integrability}
Let $\mu$ be  a probability measure on $E$
such that
\[
B(x_0) \coloneqq \int_E d(x,x_0) d\mu(x) < \infty
\]
for some $x_0\in E$ (equivalently: for all $x_0\in E$).
If $\mu$ satisfies $\bfT_{1,\alpha}$,
then, for every $x_0\in E$, $\delta\in [0,1)$, and $\lambda>0$,
\begin{equation*}
\int_E \exp\Bigl(\delta(\lambda+1)\alpha\bigl(d(x,x_0)/(\lambda+1)\bigr)\Bigr) d\mu(x)
\leq \frac{1+\delta}{1-\delta} \exp\Bigl(\delta\lambda \alpha\bigl(B(x_0)/\lambda \bigr)\Bigr).
\end{equation*}
\end{lemma}

\begin{proof}[Sketch of the proof]
This follows directly from the results of Gozlan and L\'eonard~\cite{gozlan-leonard}.
Indeed, let us denote $\alpha(t) \coloneqq 0$ for $t<0$.
By \cite[Proposition~5 (b), Corollary~6 (b)]{gozlan-leonard} (and their proofs),
we know that
\begin{equation*}
\int_E \exp\bigl(\delta \alpha\bigl(F - \int_E Fd\mu\bigr) \Bigr) d\mu
\leq \frac{1+\delta}{1-\delta}
\end{equation*}
for any bounded $1$-Lipschitz function $F\colon E\to\RR$.
Applying this to functions $F_n(x) = d(x,x_0)\land n$, $x\in E$, $n\in\NN$,
and using monotone convergence together with the facts that
$-B(x_0) \leq -\int_E F_n d\mu$ and that $\alpha$ is nondecreasing on $\RR$,
we conclude that
\begin{equation*}
\int_E \exp\bigl(\delta \alpha\bigl(d(x,x_0) - B(x_0)\bigr) \Bigr) d\mu
\leq \frac{1+\delta}{1-\delta}.
\end{equation*}
The assertion follows since by convexity of $\alpha$ we have, for $K>1$,
\begin{align*}
\MoveEqLeft[4]
\delta(\lambda+1) \alpha\bigl(d(x,x_0)/(\lambda+1)\bigr)\\
&= \delta(\lambda+1) \alpha\Bigl( \frac{1}{\lambda+1} \cdot \bigl( d(x,x_0)-B(x_0)\bigr) + \frac{\lambda}{\lambda+1} \cdot  B(x_0)/\lambda\Bigr)\\
&\leq \delta \alpha\bigl(d(x,x_0)-B(x_0) \bigr)
+ \delta\lambda \alpha\bigl( B(x_0)/\lambda\bigr).
\qedhere
\end{align*}
\end{proof}

\begin{lemma}
\label{lem:G-reverse}
Let $\mu$ be  a probability measure on $E$.
If for some  $x_0\in E$ and  $\delta>0$,
\begin{equation}
\label{eq:GL-integrability-condition}
\int_E \exp\bigl(\alpha(\delta d(x,x_0)) \bigr) d\mu(x)  <\infty,
\end{equation}
then $\mu$ satisfies $\bfT_{1,\alpha(\cdot/a)}$, where
one can take
\begin{equation}
\label{eq:G-constant-estimate}
a = \frac{2\sqrt{2} m_\alpha}{\delta}
\Bigl( 1 + \frac{1}{\log 2} \log \int_E \exp\bigl(\alpha\bigl(\delta d(x,x_0)\bigr)\bigr) d\mu(x)
\Bigr) .
\end{equation}
Here $m_\alpha$ is a constant which depends only on the function $\alpha$
(i.e., on the constants $c_0$, $t_0$ in~\eqref{eq:gozlan-assumption-a2}).
\end{lemma}

\begin{proof}[Sketch of the proof]
This follows from \cite[Theorem 1.13]{gozlan-integral-criteria};
the definition of $m_\alpha$ is given in \cite[Theorem 1.7]{gozlan-integral-criteria}.
\end{proof}

\begin{proof}[Proof of Theorem~\ref{thm:stability-of-T_1_alpha}]
Fix any $x_0\in E$.	
Since by assumption $\mu_\theta$ satisfies $\bfT_{1,\alpha}$,
Lemma~\ref{lem:GL-integrability} implies that
\begin{equation*}
\int_E \exp\Bigl(\delta(\lambda+1) \alpha\bigl( d(x,x_0)/(\lambda+1)\bigr)\Bigr) d\mu_\theta(x)
\leq \frac{1+\delta}{1-\delta} \exp\Bigl(\lambda\delta\alpha\bigl( B_{\theta}(x_0)/\lambda\bigr)\Bigr)
\end{equation*}
for any $0\leq \delta < 1$ and $\lambda>0$.	
Suppose from now on that $\lambda>0$ and $\delta\in(0,1/(1+\lambda)]$	
are as in the statement of the theorem, i.e.,
such that $I_\Theta(x_0,\delta,\lambda)<\infty$.		
Integrating the preceding inequality over $\Theta$,
we obtain that
\begin{align*}
\label{eq:proof-mix-T-1-alpha-step-integrate-theta}
\MoveEqLeft[4]
\int_E \exp\Bigl(\delta(\lambda+1) \alpha\bigl( d(x,x_0)/(\lambda+1)\bigr)\Bigr) d\mu_\mix(x)\\
&\leq
 \frac{1+\delta}{1-\delta}
\int_{\Theta} \exp\Bigl(\lambda\delta\alpha\bigl( B_{\theta}(x_0)/\lambda\bigr)\Bigr) dm(\theta)
\leq
\frac{1+\delta}{1-\delta} I_\Theta(x_0,\delta, \lambda)
\nonumber
\end{align*}
(in particular, since the right-hand side is finite,
 it follows that $\int_E  d(x,x_0) d\mu_{\textup{mix}}(x) <\infty$).
By $\delta(\lambda+1)\leq 1$ and convexity,
\[
\alpha(\delta t) = \alpha\bigl(\delta(\lambda+1)\cdot t/(\lambda+1) + (1-\delta(\lambda+1)) \cdot 0\bigr)
\leq \delta(\lambda+1)\alpha\bigl(t/(\lambda+1)\bigr).
\]
Thus,
\begin{align*}
	\int_E \exp\bigl(\alpha\bigl( \delta d(x,x_0)\bigr)\bigr) d\mu_\mix(x)
	\leq
	\frac{1+\delta}{1-\delta} I_\Theta(x_0,\delta, \lambda).
\end{align*}
By Lemma~\ref{lem:G-reverse},
this immediately implies
that the mixture $\mu_\mix$ satisfies the $\bfT_{1,\alpha(\cdot/a)}$ inequality
with constant $a$  as in the assertion.
This completes the proof.
\end{proof}

\begin{remark}
	\label{rem:on-assumptions-in-T1-alpha-result}
	Suppose that $\mu_{\mix}$ satisfies the $\bfT_{1,\alpha(\cdot/a)}$ inequality
	for some  constant $a>0$;
	then $B_{\mix}(x_0) \coloneqq \int_E d(x,x_0) d\mu_{\mix}(x) <\infty$.
	Denote $\widetilde{\alpha}(\cdot) = \alpha(\cdot/a)$.
	Jensen's inequality and Lemma~\ref{lem:GL-integrability} imply that
	\begin{align*}
		\MoveEqLeft[4]
		\int_\Theta \exp\Bigl(\delta (\lambda+1)\widetilde{\alpha}\bigl(B_\theta(x_0)/(\lambda+1)\bigr)\Bigr) dm(\theta)
		\\
		&=\int_\Theta  \exp\Bigl(\delta(\lambda+1)\widetilde{\alpha} \bigl( \int_E d(x,x_0) d\mu_\theta(x)/(\lambda+1) \bigr)\Bigr)  dm(\theta)\\
		&\leq \int_E  \exp\Bigl(\delta(\lambda+1)\widetilde{\alpha} \bigl(  d(x,x_0)/(\lambda+1) \bigr)\Bigr)d\mu_{\mix}(x) <\infty
	\end{align*}
	for any $\delta \in(0,1)$ and $\lambda>0$.
	In particular,  $I_\Theta(x_0,\delta',\lambda')<\infty$ if $\lambda'\coloneqq (\lambda+1)a$ and $0<\delta' <\min\{1/a,1/(\lambda'+1)\} = 1/(\lambda'+1)$.
\end{remark}


\appendix

\section{Some additional remarks}

\label{sec:appendix-a}

\subsection{Variance and entropy decompositions}

Recall that we have the following classical decompositions
of variance and entropy with respect to a mixture.

\begin{lemma}
\label{lem:decompositions}
We have
\begin{align}
\label{eq:variance-of-mixture}
\Var_{\mu_{\mix}} (f)
&=
\int_\Theta \Var_{\mu_\theta} (f)  dm(\theta)  + \Var_m\Bigl(\theta \mapsto \int_{\RR^d} f(x) d\mu_\theta(x) \Bigr)\\
&= \int_\Theta \Var_{\mu_\theta} (f)  dm(\theta)\nonumber\\
&\qquad + \frac{1}{2} \int_\Theta\int_\Theta  \Bigl( \int_{\RR^d} f d\mu_{\theta_1} - \int_{\RR^d}f d\mu_{\theta_2} \Bigr)^2 dm(\theta_1) dm(\theta_2),\nonumber\\
\label{eq:entropy-of-mixture}
\Ent_{\mu_\mix} (f^2)
&=
\int_\Theta \Ent_{\mu_\theta} (f^2)  dm(\theta)  + \Ent_m\Bigl(\theta \mapsto \int_{\RR^d} f(x)^2 d\mu_\theta(x) \Bigr).
\end{align}
\end{lemma}

\begin{proof}
Indeed,
\begin{align*}
\MoveEqLeft[2]
\Var_{\mu_{\mix}} (f)
=
\int_{\RR^d} f(x)^2 d\mu_{\mix}(x)
- \Bigl( \int_{\RR^d} f(x) d\mu_{\mix}(x)\Bigr)^2\\
&=
\int_{\Theta}\int_{\RR^d} f(x)^2 d\mu_\theta(x) dm(\theta)
-	\int_{\Theta}\Bigl(\int_{\RR^d} f(x) d\mu_\theta(x) \Bigr)^2dm(\theta)\\
&\quad
+\int_{\Theta}\Bigl(\int_{\RR^d} f(x) d\mu_\theta(x) \Bigr)^2dm(\theta)
- \Bigl( 	\int_{\Theta}\int_{\RR^d} f(x) d\mu_\theta(x) dm(\theta) \Bigr)^2\\
&= \int_{\Theta} \Var_{\mu_\theta} (f)  dm(\theta)  + \Var_m\Bigl(\theta \mapsto \int_{\RR^d} f(x) d\mu_\theta(x) \Bigr)\\
&= \int_\Theta \Var_{\mu_\theta} (f)  dm(\theta)
+ \frac{1}{2} \int_\Theta\int_\Theta  \Bigl( \int_{\RR^d} f d\mu_{\theta_1} - \int_{\RR^d}f d\mu_{\theta_2} \Bigr)^2 dm(\theta_1) dm(\theta_2).
\end{align*}
Similarly,
(we use the notation $\mu_\theta(h) = \int_{\RR^d} h(x) d\mu_\theta(x)$),
\begin{align*}
\MoveEqLeft[2]
\Ent_{\mu_{\mix}} (f^2)
= \int_{\RR^d} f^2\log(f^2) d\mu_{\mix}	- \int_{\RR^d} f^2 d\mu_{\mix}
	\log\Bigl( \int_{\RR^d} f^2  d\mu_{\mix}\Bigr)\\
&= \int_{\Theta} \mu_\theta\bigl(f^2\log(f^2) \bigr) dm(\theta)
	 - \int_{\Theta}  \mu_\theta(f^2 ) dm(\theta)
	\log\Bigl( \int_{\Theta}  \mu_\theta(f^2 ) dm(\theta)\Bigr)\\
&= \int_{\Theta} \mu_\theta\bigl(f^2\log(f^2) \bigr) dm(\theta)
-\int_{\Theta} \mu_\theta(f^2 )
	\log\bigl( \mu_\theta(f^2 )\bigr)  dm(\theta)\\
&\quad +\int_{\Theta} \mu_\theta(f^2 )
\log\bigl(\mu_\theta(f^2 )\bigr)  dm(\theta)
	 - \int_{\Theta}  \mu_\theta(f^2 ) dm(\theta)
	\log\Bigl( \int_{\Theta}  \mu_\theta(f^2 ) dm(\theta)\Bigr)\\
&= \int_{\Theta} \Ent_{\mu_\theta} (f^2)  dm(\theta)
+ \Ent_m\Bigl(\theta \mapsto \int_{\RR^d} f^2 d\mu_\theta \Bigr).
\qedhere
\end{align*}
\end{proof}

\subsection{Approximation arguments}

Below, on the example of the convex Poincar\'e inequality,
we explain that inequalities for convex functions
can be considered in the class of smooth convex functions
or general convex functions (with the length of the gradient)
and the constant does not change.

\begin{lemma}
\label{lem:convex-poincare-class-of-funtions}
Let $\mu$ be a probability measure on $\RR^d$.
Suppose that for some $C\in(0,\infty)$,
\begin{equation}
\label{eq:convex-poincare-but-only-for-smooth}
\Var_{\mu}(f) \leq C \int_{\RR^d} |\nabla f(x)|^2 d\mu(x)
\end{equation}
for all smooth convex $1$-Lipschitz functions $f\colon\RR^d\to\RR$.
Then $\mu$ satisfies the convex Poincar\'e inequality~\eqref{eq:definition-convex-Poincare}
with the same constant
(for all convex functions, with the length of the gradient).
\end{lemma}

\begin{proof}
Testing~\eqref{eq:convex-poincare-but-only-for-smooth}
with functionals $(x_1,\dots,x_d)\mapsto x_i$ yields $\int_{\RR^d} \abs{x}^2 d\mu(x) < \infty$
    Let $f\colon\RR^d\to \RR$ be convex and $1$-Lipschitz (but not necessarily smooth).
    Introduce functions $f_\varepsilon$, $\varepsilon>0$,
	which are standard mollifications of $f$
	(we choose the mollifier to be nonnegative, smooth, and compactly supported).
	For any $\varepsilon>0$,
	the function $f_\varepsilon$ is smooth, convex, and $1$-Lipschitz, so
	\begin{equation}
		\label{eq:counterexample-2-convex-poincare-for-smooth}
		\Var_{\mu}(f_\varepsilon)
		\leq C \int_{\RR^d} \abs{\nabla f_\varepsilon}^2 d\mu.
	\end{equation}
	For $\varepsilon\to 0^+$,
	$f_\varepsilon$ converges pointwise to $f$ (everywhere, since $f$ is $1$-Lipschitz and the mollifier has compact support).
	so 	$\Var_{\mu_d}(f_\varepsilon)\to \Var_{\mu_d}(f)$
	by the dominated convergence theorem
	(recall that $f_\varepsilon$ is $1$-Lipschitz and $\int_{\RR^d} \abs{x}^2 d\mu(x) < \infty.$).
	Moreover, $\limsup_{\varepsilon\to 0^+} \abs{\nabla f_\varepsilon (x)} \leq \abs{\nabla f(x)}$ for any $x\in\RR^d$.
	Indeed, $\abs{\nabla f_\varepsilon (x)} \leq 1$ (by Lipschitz continuity),
	so we can find a vector $u_x\in\RR^d$
	and
	 a sequence $\varepsilon_k \to 0^+$, such that $\lim_{k\to\infty} \abs{\nabla f_{\varepsilon_k}  (x)} = \limsup_{\varepsilon\to 0^+} \abs{\nabla f_\varepsilon (x)}$, $\nabla f_{\varepsilon_k}(x) \to u_x$, $\abs{u_x} \leq 1$.
	 By convexity, for every $y\in\RR^d$,
	 \[
	 f_{\varepsilon_k}(x) + \langle \nabla f_{\varepsilon_k}(x),y-x\rangle \leq f_{\varepsilon_k}(y);
	 \]
	 taking $\varepsilon_k\to 0$, we conclude that $u_x$ is a subgradient of $f$ at $x$.
	 Thus,
	 \[
	 \limsup_{\varepsilon\to 0^+} \abs{\nabla f_\varepsilon (x)}
	 =\abs{u_x} \leq \abs{\nabla f(x)}.
	 \]
	 By virtue of Fatou's lemma, we obtain~\eqref{eq:counterexample-2-convex-poincare-for-smooth} for $f$
     (with the same constant $C$).

     We can then pass to arbitrary convex functions
     since they can can be pointwise approximated by convex Lipschitz functions.
\end{proof}

\subsection{Consequences of weak transportation inequalities}

Below we prove some of the auxiliary facts
stated at the beginning of Section~\ref{sec:proofs-weak}.

\begin{proof}[Sketch of the proof of Corollary~\ref{cor:weak-transport-implies-convex-Poincare}]
Let $f\colon\RR^d \to \RR$ be convex, Lipschitz, and bounded from below.
For $x\in\RR^d$ choose any  subgradient $u_x$ of $f$ at $x$
and define
\begin{equation*}
f^x (z) \coloneqq f(x) + \langle u_x,  z-x\rangle, \quad z\in\RR^d,
\end{equation*}
so that $f^x \le f$ on $\RR^d$.
Taking $z = x + \varepsilon u_x$ with $\varepsilon \to 0^+$ yields $|u_x| \le |\nabla f(x)|$
(recall that in the case of the convex Poincar\'e inequality we work with the length of the gradient).
Since $f$ is Lipschitz, for sufficiently small $\varepsilon>0$
we have $\varepsilon \abs{\nabla f(x)}\leq \delta$ for all $x\in\RR^d$.
Hence,
\begin{align*}
Q_1^{c} (\varepsilon f)(x)
&\geq \inf_{y\in\RR^d} \bigl\{ \varepsilon f^x(x-y) + c(y) \bigr\}\\
&= \varepsilon f(x) +\inf_{y\in\RR^d} \{ -\varepsilon \langle u_x,y \rangle + c(y)\}\\
& = \varepsilon f(x) - c^*(\varepsilon u_x)
\geq  \varepsilon f(x) -\frac{1}{2} \varepsilon^2 \abs{\nabla f(x)}^2
\end{align*}
(recall that $\varepsilon\abs{u_x} \leq \varepsilon\abs{\nabla f(x)} \leq \delta$
and note that $c^*(y) = \abs{y}^2/2$ if $\abs{y}\leq \delta$).
We now substitute $\varepsilon f$ into the dual formulation~\eqref{eq:dual-formulation}
and use the above estimate.
An inspection of the Taylor expansions up to order $\varepsilon^2$ yields the convex Poincar\'e inequality with constant $C$
(one can now pass to general convex functions by approximation).
\end{proof}

For the proof of Lemma~\ref{lem:aux-exponential-moment-of-mixture}
we also need the following explicit bounds.

\begin{lemma}
\label{lem:aux-mean-estimate}
Let
$\{\mu_\theta\}_{\theta \in \Theta}$
be a family of  probability measures on $\RR^d$
such that
\[
D_\Theta \coloneqq \sup_{\theta\in\Theta, \theta'\in\Theta}
\BarCostQuadratic{\mu_\theta}{\mu_{\theta'}} < \infty.
\]
Then, for any $\theta_1,\theta_2\in\Theta$,
\begin{equation*}
\Bigl(\int_{\RR^d} \abs{x} d\mu_{\theta_1}(x) - \int_{\RR^d} \abs{x} d\mu_{\theta_2}(x) \Bigr)^2
\leq
2 D_\Theta.
\end{equation*}
\end{lemma}

\begin{proof}
This follows immediately from calculations as in the second part of the proof
of Proposition~\ref{prop:stability-of-convex-Poincare-improved}
(applied to the function $f(x)=\abs{x}$, $x\in\RR^d$, which satisfies $\abs{\nabla f(x)} \leq 1$).
\end{proof}

\begin{lemma}
\label{lem:aux-exponential-moment-estimate}
Assume that $\mu$ is a probability measure on $\RR^d$
which satisfies the transport--entropy inequality $\bfTbar_2^-$
with constant $C$,
Then, for $s>0$,
\[
\int_{\RR^d} \exp(s\abs{x}) d\mu(x)
\leq \exp(2Cs^2) + \exp\bigl(2s\int_{\RR^d} \abs{x} d\mu(x)\bigr) < \infty.
\]
\end{lemma}

\begin{proof}
Let $Q_1$ stand for the infimum convolution operator with the quadratic cost $c(x) = \abs{x}^2/2$.
For $r>0$ consider the convex function $f(x) = r\abs{x}$, $x\in\RR^d$.
We have
\begin{equation*}
    Q_1f(x)
    =\inf_{y\in\RR^d} \Bigl\{ r\abs{y} + \frac{1}{2} \abs{x-y}^2\Bigr\}
    =\begin{cases}
    r\abs{x} -\frac{1}{2} r^2 & \text{ if } \abs{x} \geq r,\\
    \frac{1}{2} \abs{x}^2 & \text{ if } r\geq \abs{x},
    \end{cases}
\end{equation*}
since the infimum on the left-hand side is attained
at $y=(1-r/\abs{x}) x$ if $\abs{x}\geq r$
and at $y=0$ if $r\geq \abs{x}$.
In particular, $Q_1f(x) \geq r\abs{x}/2$ if $\abs{x}\geq r$.
Thus,
\begin{align*}
\int_{\RR^d} \exp(C^{-1} r\abs{x}/2) d\mu(x)
&= \int_{\RR^d} \exp(C^{-1} r\abs{x}/2)
\bigl(\indicatorbraces{r> \abs{x}} + \indicatorbraces{\abs{x}\geq r}\bigr)
d\mu(x)\\
&\leq  \exp(C^{-1} r^2/2) + \int_{\RR^d} \exp(C^{-1} Q_1f(x)) d\mu(x)\\
&\leq  \exp(C^{-1} r^2/2) + \exp\bigl(C^{-1} \int_{\RR^d} f(x) d\mu(x) \bigr),
\end{align*}
where in the last inequality we used Theorem~\ref{thm:duality-from-Kantorovich-duality}.
Setting $r=2Cs$ completes the proof.
\end{proof}

\begin{proof}[Proof of Lemma~\ref{lem:aux-exponential-moment-of-mixture}]
Denote $C\coloneqq C_\Theta$, $D\coloneqq D_\Theta$, $M\coloneqq M_\Theta$;
from Remark~\ref{rem:chi-square-vs-calTbar} we know that $D\leq CM$,
so necessarily $D<\infty$.
Theorem~\ref{thm:characterization-from-Kantorovich-duality} implies that all the mixed components have all exponential moments;
moreover, for fixed $s>0$, $\int_{\RR^d} \exp(s\abs{x}) d\mu_\theta(x)$
can be bounded (uniformly in $\theta$)
by Lemmas~\ref{lem:aux-exponential-moment-estimate} and~\ref{lem:aux-mean-estimate}.
This completes the proof.
\end{proof}

\subsection{Auxiliary facts for the proof of Theorem~\ref{thm:stability-of-Tbar_2^-}}

In the following lemma we prove the existence of minimizers in weak transportation problems. We remark that it could be obtained from abstract results from \cite{MR4029731}, however, since this would still require a verification of assumptions, we prefer to provide a direct short argument specialized to the type of cost we consider.

\begin{lemma} \label{lem:minimizer-existence} Assume that $\mu$ and $\nu$ are probability measures on $\RR^d$ with finite first moment. Let $c:\RR^d \to [0,\infty)$ be convex.
Then
the infimum in the definition of  $\BarCostGeneral{c}{\nu}{\mu}$ is attained,
i.e.,
there exists a coupling $\pi$ between $\mu$ and $\nu$ for which
\begin{displaymath}
\BarCostGeneral{c}{\nu}{\mu}
=  \int_{\RR^d}	c\bigl(x - \int_{\RR^d} y p_x(dy)\bigr) \mu(dx),
\end{displaymath}
where $p_x(\cdot)$ is the conditional measure defined ($\mu$ almost surely) by $\pi(dxdy) = p_x(dy)\mu(dx)$.
\end{lemma}

\begin{proof} For $n \ge 1$ let $\pi_n \sim (X_n,Y_n)$ be couplings between $\mu$ and $\nu$ such that $\lim_n \EE c(X_n - \EE(Y_n|X_n)) = \BarCostGeneral{c}{\nu}{\mu}$. Define $Z_n := \EE(Y_n|X_n)$. Since $X_n \sim \mu, Y_n \sim \nu$ are in $L_1$,
	 the family of  random vectors
	  $(X_n,Y_n,Z_n)$ is tight and uniformly integrable.
	  Upon passing to a subsequence we may therefore assume that $(X_n,Y_n,Z_n) \xrightarrow{d} (X,Y,Z)$, where $X \sim \mu$ and $Y \sim \nu$. By Skorokhod's representation theorem we may assume that those random variables live on the same probability space and the convergence is almost sure. We aim to show that $\EE(Z|X) = \EE(Y|X)$. Pick any $f \in C_b(\RR^d)$. Since $(X_n,Y_n,Z_n)$ are uniformly integrable and, from the very definition of $Z_n$, $\EE(Y_n f(X_n)) = \EE(Z_n f(X_n))$, by passing to a limit we obtain $\EE(Y f(X)) =\EE(Z f(X))$. We can approximate indicator functions of closed sets by a sequence of continuous and bounded functions, therefore $\EE Z \indicatorbraces{X \in A} = \EE Y \indicatorbraces{X \in A}$ for every closed $A \subset \RR^d$. It remains to use Dynkin's lemma to get $\EE(Z|X) =\EE(Y|X)$ as claimed. Hence, by convexity of $c$,
\begin{align*}
\EE c(X - \EE(Y|X))
&= \EE c(X - \EE(Z|X)) \le \EE c(X - Z) = \EE \liminf_{n\to \infty}  c(X_n - Z_n)\\
&\le \liminf_{n\to \infty} \EE c(X_n - Z_n) =
\liminf_{n\to \infty} \EE c(X_n - \EE(Y_n|X_n)) = \BarCostGeneral{c}{\nu}{\mu}.
\end{align*}
This completes the proof.
\end{proof}

\begin{lemma}\label{lem:Markov-reverse}
Let $X,Y,Z$ be random variables, defined on the same probability space, with values in a measurable space $(S,\mathcal{S})$.
Assume that for any $A \in \mathcal{S}$, $\PP(Z\in A|X,Y) = \PP(Z \in A|Y)$ a.s. Then for any $A \in \mathcal{S}$, $\PP(X \in A|Y,Z) = \PP(X \in A|Y)$ a.s.
\end{lemma}

\begin{proof}
Consider arbitrary $A,B,C \in \mathcal{S}$.
Using
conditioning on appropriate $\sigma$-fields in the steps denoted by ``c'' and
measurability with respect to appropriate $\sigma$-fields in the steps denoted by ``m'',
we can write:
\begin{align*}
\EE \indicatorbraces{X\in A}\indicatorbraces{Y\in B} \indicatorbraces{Z\in C}
&\overset{\text{c}}{=} \EE\Bigl(
\conditionalEE[\Big]{
\indicatorbraces{X\in A}
\indicatorbraces{Y\in B}
\indicatorbraces{Z\in C}
}{X,Y}
\Bigr) \\
&\overset{\text{m}}{=} \EE\Bigl( \indicatorbraces{X\in A} \indicatorbraces{Y\in B}\conditionalEE[\big]{\indicatorbraces{Z\in C} }{X,Y} \Bigr)\\
&= \EE\Bigl( \indicatorbraces{X\in A} \indicatorbraces{Y\in B}\conditionalEE[\big]{\indicatorbraces{Z\in C} }{Y} \Bigr)\\
& \overset{\text{c}}{=} \EE\Bigl(
\conditionalEE[\Big]{
\indicatorbraces{X\in A}
\indicatorbraces{Y\in B}
\conditionalEE[\big]{\indicatorbraces{Z\in C} }{Y}
}{Y}
\Bigr)\\
& \overset{\text{m}}{=}   \EE\Bigl(
\conditionalEE[\big]{
\indicatorbraces{X\in A}
}{Y}
\indicatorbraces{Y\in B}
\conditionalEE[\big]{\indicatorbraces{Z\in C} }{Y}
\Bigr)\\
& \overset{\text{m}}{=} \EE\Bigl(
\conditionalEE[\Big]{\conditionalEE[\big]{\indicatorbraces{X\in A}}{Y}
\indicatorbraces{Y\in B}\indicatorbraces{Z\in C} }{Y}\Bigr)\\
& \overset{\text{c}}{=}\EE\Bigl(
\conditionalEE[\big]{\indicatorbraces{X\in A}}{Y}
\indicatorbraces{Y\in B}
\indicatorbraces{Z\in C} \Bigr),
\end{align*}
which in combination with Dynkin's $\pi$-$\lambda$ lemma implies that for any $D \in \mathcal{S}\otimes \mathcal{S}$, $\EE[\indicatorbraces{X \in A} \indicatorbraces{(Y,Z)\in  D}] = \EE[\conditionalEE[\big]{\indicatorbraces{X\in A}}{Y}\indicatorbraces{(Y,Z) \in D}]$, i.e.,
$\conditionalPP{X\in A}{Y,Z} = \conditionalPP{X\in A}{Y}$.
\end{proof}


\section{Counterexample}

\label{sec:counterexample-convex-poincare}

It is well-known that
the convex Poincaré inequality~\eqref{eq:definition-convex-Poincare}
implies subexponential estimates for the upper tail
of convex Lipschitz functions:
there exists a \emph{universal} constant $c>0$,
such that if $\mu$ is a probability measure on $\RR^d$
which satisfies the convex Poincar\'e inequality
and $X$ is a random vector with law $\mu$,
then
\begin{equation}
	\label{eq:upper-tail-under-convex-poincare}
	\PP(f(X) \geq \EE f(X) - t) \leq 2\exp\bigl( - \frac{c t}{\sqrt{\ConstantConvexPoincare{\mu}}}\bigr), \quad t\geq 0,
\end{equation}
for every convex $1$-Lipschitz function $f\colon\RR^d\to\RR$,
see, e.g., \cite[Section~2]{adamczak-strzelecki}.
The goal of this section is to show
that the convex Poincaré inequality
 does \emph{not} imply subexponential estimates
 for the lower tail of convex 1-Lipschitz functions
 with a constant independent of the dimension $d$.
This provides a negative answer to \cite[Question 7.3]{adamczak-strzelecki}.

For $d\geq 1$,
let $c_d$ be the optimal (i.e., largest possible) constant such that the following holds:
for every probability measure $\mu$ on $\RR^d$
which satisfies the convex Poincar\'e inequality with constant $1$ (i.e., $\ConstantConvexPoincare{\mu}\leq 1$),
if $X$ is a random vector with law $\mu$,
then
\begin{equation}
\label{eq:definition-lower-tail-under-convex-poincare-optimal-constant}
\PP(f(X) \leq \EE f(X) - t) \leq 2\exp( -c_d t), \quad t\geq 0,
\end{equation}
for every convex $1$-Lipschitz function $f\colon\RR^d\to\RR$.
We shall prove the following exact asymptotics:
\[
c_d = (1+o(1))\frac{\log(d)}{\sqrt{d}} \quad \text{as } d\to \infty.
\]

\begin{proposition}
	\label{prop:counterexample-convex-poincare-sharp-upper-bound}
	The constant $c_d$, defined in~\eqref{eq:definition-lower-tail-under-convex-poincare-optimal-constant}, satisfies
	\[
	c_d \leq \frac{d\log(2d)}{(d-1)^{3/2}}, \quad d\geq 2.
	\]
	In particular,  $\limsup_{d\to\infty} c_d\sqrt{d}/\log(d) \leq 1$.
\end{proposition}

\begin{proof}
	Let $d\geq 2$.
	Denote $p_d = 1/d$ and consider
	the probability measure $\mu_d \coloneqq p_d \delta_0 + (1-p_d) \nu_d$,
	where $\nu_d$ is the uniform measure on the sphere $\mathcal{S}_{d-1} \coloneqq \sqrt{d-1} \sphere^{d-1}$
	(and $\delta_0$ is the Dirac delta at $0\in\RR^d$).
	Note that $\nu_d$ satisfies the following Poincar\'e inequality:
	\[
	\Var_{\nu_d}(f) \leq \int_{\mathcal{S}_{d-1}} \abs{\nabla_S f}^2 d\nu_d
	\]
	for any function smooth function $f\colon \RR^d\to\RR$, see, e.g., \cite[Section~4.8]{bakry-gentil-ledoux} and the citations therein. Here and below
	for $x\in\mathcal{S}_{d-1}$ we denote by
	\begin{equation*}
		\nabla_S f(x)  \coloneqq \nabla f(x) - \bigl\langle \nabla f(x), \frac{x}{\abs{x}}\bigr\rangle \frac{x}{\abs{x}}
	\end{equation*}
  the spherical gradient, i.e., the component of $\nabla f$ tangent to $\mathcal{S}_{d-1}$.

	We claim that $\mu_d$ satisfies the convex Poincar\'e inequality
	with a dimension independent constant.
	Fix a smooth convex $1$-Lipschitz function $f\colon\RR^d \to\RR$
	(by Lemma~\ref{lem:convex-poincare-class-of-funtions}
	it suffices to consider such functions).
	By Jensen's inequality,
	\[
	f(0) \leq \int_{\RR^d} f d\nu_d.
	\]
	For $x\in\mathcal{S}_{d-1}$ denote
	\begin{equation*}
	r(x) \coloneqq \langle \nabla f(x), x/\abs{x}\rangle,
	\end{equation*}
	so  that $\nabla f(x) = \nabla_S f(x) + r(x) x/\abs{x}$.
	For $x\in\mathcal{S}_{d-1}$, by convexity,
	\[
	f(x)  - f(0) \leq \langle \nabla f(x), x\rangle = \abs{x} r(x) =  \sqrt{d-1} r(x).
     \]
	From these two inequalities we conclude that
	\[
	0 \leq \int_{\RR^d} f d\nu_d - f(0)
	 \leq \sqrt{d-1} \int_{\RR^d} r(x) d\nu_d(x)
	 \leq \sqrt{d-1} \Bigl( \int_{\RR^d} r(x)^2 d\nu_d(x) \Bigr)^{1/2},
	\]
	i.e.,
	\[
	p_d \Bigl( \int_{\RR^d} f d\nu_d -  \int_{\RR^d} f d\delta_0 \Bigr)^{2}
	\leq 	\frac{d-1}{d}\int_{\RR^d} r(x)^2 d\nu_d(x).
	\]
	Hence, by the variance decomposition~\eqref{eq:variance-of-mixture},
	\begin{align*}
		\Var_{\mu_d}(f)
		&= (1-p_d) \Var_{\nu_d}(f) + p_d(1-p_d)   \Bigl( \int_{\RR^d} f d\nu_d -  \int_{\RR^d} f d\delta_0 \Bigr)^{2}\\
		&\leq  (1-p_d) \int_{\RR^d} \abs{\nabla_S f}^2 d\nu_d + (1-p_d)  \int_{\RR^d} r(x)^2 d\nu_d(x)\\
		&=  (1-p_d) \int_{\RR^d} \abs{\nabla f}^2 d\nu_d
		\leq \int_{\RR^d} \abs{\nabla f}^2 d\mu_d.
	\end{align*}
	
	Let now $X_d$ be a random vector with values in $\RR^d$ with law $\mu_d$.
	Consider the convex $1$-Lipschitz function $f_d\colon\RR^d\to\RR$
	given by $f_d(x) = \abs{x}$.
	The random variable $f(X_d)$ takes only two values:
	$0$ (with probability $p_d$) and $\sqrt{d-1}$ (with probability $1-p_d$).
	Set $t_d \coloneqq (1-p_d)\sqrt{d-1} = (d-1)^{3/2}/d$.
	Then
	\[
	\EE f_d(X_d) -t_d = (1-p_d) \sqrt{d-1} -t_d = 0,
	\]
	so
	\[
	\PP(f_d(X_d) \leq \EE f_d(X_d) - t_d) = \PP(X_d = 0) = p_d = 1/d.
	\]
	By~\eqref{eq:definition-lower-tail-under-convex-poincare-optimal-constant}, $1/d \leq 2\exp(-c_d t_d)$, i.e.,
	$c_d \leq \log(2d)/t_d = d\log(2d)/(d-1)^{3/2}$.
\end{proof}

\begin{remark}
If one does not care about exact constants,
then in the above proof one can work everywhere with the full gradient
(and not introduce $\nabla_S f(x)$ and $r(x)$).
One could also take $\nu_d$ to be the uniform measure on the discrete cube $\{-1,1\}^d$.
This is a product of measures with bounded supports
and as such satisfies the convex Poincar\'e inequality
with a dimension independent constant
(see, e.g., \cite[Theorem~1.1]{MR1399224}).
\end{remark}

\begin{proposition}
	\label{prop:counterexample-convex-poincare-lower-bound-c_d}
	The constant $c_d$, defined in~\eqref{eq:definition-lower-tail-under-convex-poincare-optimal-constant}, satisfies
\[
	\liminf_{d\to\infty}  c_d\sqrt{d}/\log d \geq 1.
\]
\end{proposition}

\begin{proof}
	Fix $d\geq 2$
	(which will be later specified to be sufficiently large).
	Take any probability measure $\mu$ on $\RR^d$ such that $\ConstantConvexPoincare{\mu}\leq 1$,
	let $X$ be a random vector with law $\mu$,
	and take any convex $1$-Lipschitz function $f\colon\RR^d\to\RR$.
	We shall estimate $\PP(f(X) \leq \EE f(X) - t)$, $t\geq 0$,
	in order to get a lower bound on $c_d$.
	
	Let $g\colon\RR^d\to\RR^d$ be a measurable subgradient map
	so that
	\[
	f(y) + \langle g(y), x-y\rangle \leq f(x)
	\]
	for every $x,y\in\RR^d$ (see, e.g., \cite{mathoverflow-453991} (cf.\ \cite[6.9.7. Theorem]{bogachev})
	or \cite[Proposition~3.1]{bobkov2026subgradientsconvexfunctionsorlicz}).
	Let $Y$ be an independent copy of $X$.
    Using the above inequality with $y=Y$ yields (after rearranging and taking expectations)
	\begin{align*}
		 \EE f(Y) - f(x)
		&\leq \EE \langle g(Y), Y - x \rangle\\
		&= \EE \langle g(Y), Y-\EE Y\rangle + \EE \langle g(Y), \EE Y - x\rangle\\
		&\leq \sqrt{d} - \bigl\langle   x -\EE Y, \EE g(Y) \rangle,
	\end{align*}
	where in the last step we used the Cauchy--Schwarz inequality together with the facts that
	$\abs{g(x)}\leq \abs{\nabla f(x)} \leq 1$ (since $f$ is $1$-Lipschitz)
	and $\EE \abs{Y-\EE Y}^2 \leq d$
	(which we obtain by testing the convex Poincar\'e inequality for $\mu$ with the convex functions
	$(y,\dots,y_d)\mapsto y_i$).
	Hence,
	\[
	\{ f(X) \leq \EE f(X) - t\} \subset \bigl\{   - \langle  X -\EE X, \EE g(X) \rangle \geq t-\sqrt{d}\bigr\}.
	\]
    The function $\RR^d\ni x\mapsto - \langle   x -\EE X, \EE g(X) \rangle$ is convex $1$-Lipschitz
    and has mean $0$ (under $\mu$).
    Therefore, we can estimate its upper tail using~\eqref{eq:upper-tail-under-convex-poincare}
    and obtain
	\begin{align}
	\label{eq:lower-tail-under-convex-poincare-large-t}
	\PP\bigl( f(X) \leq \EE f(X) - t\bigr)
	&\leq \PP\bigl(   - \langle  X -\EE X, \EE g(X) \rangle \geq t-\sqrt{d}\bigr)\\
	&\leq 2 \exp( - c (t-\sqrt{d}))\nonumber
	\end{align}
	for $t\geq \sqrt{d}$;
	here $c>0$ is the universal constant from~\eqref{eq:upper-tail-under-convex-poincare}.
	
	Since $\Var(f(X)) \leq 1$ (by the convex Poincar\'e inequality),
	the Chebyshev--Cantelli inequality gives
	\begin{equation}
			\label{eq:lower-tail-under-convex-poincare-small-t-cantelli}
		\PP\bigl( f(X) \leq \EE f(X) - t\bigr) \leq \frac{1}{1+t^2}
	\end{equation}
    for $t\geq 0$.

    Fix any  $A>0$ such that $Ac > 1$ (where $c$ is the universal constant from~\eqref{eq:upper-tail-under-convex-poincare})
    and define
    \begin{align*}
    	T_d &\coloneqq \sqrt{d} + A\log d,\\
    	\widetilde{c}_d&\coloneqq \frac{\log\bigl( 2(1+T_d^2)\bigr)}{T_d}.
    \end{align*}
    Then $T_d\geq \sqrt{d}$, $T_d\to\infty$ and $\widetilde{c}_d\to 0$ as $d\to \infty$
    (in particular, $\widetilde{c}_d < c$ if $d$ is sufficiently large).
    We claim that if $d$ is sufficiently large (i.e., larger than some $d_0$, which depends only on $A$ and $c$), then for $t\in [0,T_d]$,
    \begin{align}
    	\label{eq:lower-tail-under-convex-poincare-part-1}
    \PP\bigl( f(X) \leq \EE f(X) - t\bigr) &\leq \frac{1}{1+t^2} \leq 2\exp(-\widetilde{c}_d t),\\
    \intertext{while for $t\geq T_d$,}
   	\label{eq:lower-tail-under-convex-poincare-part-2}
    \PP\bigl( f(X) \leq \EE f(X) - t\bigr) &\leq 2 \exp( - c (t-\sqrt{d})) \leq 2\exp(-\widetilde{c}_d t).
    \end{align}

    Indeed, by the definition of $\widetilde{c}_d$, for $t=T_d$ we have equality in the second inequality in~\eqref{eq:lower-tail-under-convex-poincare-part-1}.
    Moreover, the function $(0,\infty)\ni t\mapsto \log(2(1+t^2))/t$ is decreasing as
    \[
    \Bigl( \frac{\log\bigl( 2(1+t^2)\bigr)}{t} \Bigr)'
    = \frac{\frac{2t^2}{1+t^2} - \log(2(1+t^2))}{t^2}
    =\frac{2-\log(2) - \frac{2}{1+t^2} - \log(1+t^2)}{t^2}
    \]
    and it is easy to check that the numerator is negative (the maximum is attained when $1+t^2=2$).
    This (and~\eqref{eq:lower-tail-under-convex-poincare-small-t-cantelli}) implies~\eqref{eq:lower-tail-under-convex-poincare-part-1}.

    The first inequality in~\eqref{eq:lower-tail-under-convex-poincare-part-2}
    holds by~\eqref{eq:lower-tail-under-convex-poincare-large-t};
    since for large $d$ we have $c\geq \widetilde{c}_d$,
    we only need to verify that the second inequality in~\eqref{eq:lower-tail-under-convex-poincare-part-2} holds for $t=T_d$.
     But,
      by the definitions of $\widetilde{c}_d$ and $T_d$,
        \[
     \exp(\widetilde{c_d} T_d) = 2(1+T_d^2)
     = 2(1+d + 2A\sqrt{d}\log d + A^2\log^2d) \leq d^{Ac} = \exp( c(T_d - \sqrt{d})),
     \]
     where the inequality  holds for sufficiently large $d$ (since $Ac>1$).

     Inequalities \eqref{eq:lower-tail-under-convex-poincare-part-1} and~\eqref{eq:lower-tail-under-convex-poincare-part-2}
     imply  $\widetilde{c}_d \leq c_d$ (for sufficiently large $n$).
     Moreover, it is easy to see that $\widetilde{c}_d = (1+o(1)) \log(d)/\sqrt{d}$ as $d\to\infty$.
     This completes the proof.
\end{proof}


\section*{Acknowledgements}

We would like to thank Nathael Gozlan for insightful comments
concerning the article~\cite{gozlan-roberto-samson-new-characterization}.

R.A.'s research was partially supported by The National Science Center, Poland, grant IMPRESS-U 2023/05/Y/ST1/00188.

\section*{Declaration of generative AI and AI-assisted technologies in the manuscript preparation process}

Appendix~\ref{sec:counterexample-convex-poincare}
is a result of interaction with ChatGPT, which suggested the formulations
and proofs of the results presented in this part of the article.
They have been subsequently verified and rewritten by the authors.
The formulations and proofs of the remaining results of the article
have been discovered by the authors.
The authors take full responsibility for the content of the published article.
The access to ChatGPT was obtained through OpenAI’s ChatGPT for Academic Researchers program.


\bibliographystyle{amsplain}
\bibliography{literature}

\providecommand{\bysame}{\leavevmode\hbox to3em{\hrulefill}\thinspace}
\providecommand{\MR}{\relax\ifhmode\unskip\space\fi MR }
\providecommand{\MRhref}[2]{%
  \href{http://www.ams.org/mathscinet-getitem?mr=#1}{#2}
}
\providecommand{\href}[2]{#2}
\begin{thebibliography}{10}

\bibitem{MR3456588}
Rados{\l}aw Adamczak and Micha{\l} Strzelecki, \emph{Modified log-{S}obolev
  inequalities for convex functions on the real line. {S}ufficient conditions},
  Studia Math. \textbf{230} (2015), no.~1, 59--93. \MR{3456588}

\bibitem{adamczak-strzelecki}
\bysame, \emph{On the convex {P}oincar\'{e} inequality and weak transportation
  inequalities}, Bernoulli \textbf{25} (2019), no.~1, 341--374. \MR{3892322}

\bibitem{MR883646}
David~J. Aldous, \emph{Exchangeability and related topics}, \'Ecole d'\'et\'e{}
  de probabilit\'es de {S}aint-{F}lour, {XIII}---1983, Lecture Notes in Math.,
  vol. 1117, Springer, Berlin, 1985, pp.~1--198. \MR{883646}

\bibitem{DBLP:conf/stoc/SanjeevK01}
Sanjeev Arora and Ravi Kannan, \emph{Learning mixtures of arbitrary
  {G}aussians}, Proceedings on 33rd Annual {ACM} Symposium on Theory of
  Computing, July 6-8, 2001, Heraklion, Crete, Greece (Jeffrey~Scott Vitter,
  Paul~G. Spirakis, and Mihalis Yannakakis, eds.), {ACM}, 2001, pp.~247--257.

\bibitem{MR4029731}
J.~Backhoff-Veraguas, M.~Beiglb\"ock, and G.~Pammer, \emph{Existence, duality,
  and cyclical monotonicity for weak transport costs}, Calc. Var. Partial
  Differential Equations \textbf{58} (2019), no.~6, Paper No. 203, 28.
  \MR{4029731}

\bibitem{bakry-gentil-ledoux}
Dominique Bakry, Ivan Gentil, and Michel Ledoux, \emph{Analysis and geometry of
  {M}arkov diffusion operators}, Grundlehren der mathematischen Wissenschaften
  [Fundamental Principles of Mathematical Sciences], vol. 348, Springer, Cham,
  2014. \MR{3155209}

\bibitem{bardet-et-al}
Jean-Baptiste Bardet, Natha\"{e}l Gozlan, Florent Malrieu, and Pierre-Andr\'{e}
  Zitt, \emph{Functional inequalities for {G}aussian convolutions of compactly
  supported measures: explicit bounds and dimension dependence}, Bernoulli
  \textbf{24} (2018), no.~1, 333--353. \MR{3706760}

\bibitem{MR1274699}
Jose-M. Bernardo and Adrian F.~M. Smith, \emph{Bayesian theory}, Wiley Series
  in Probability and Mathematical Statistics: Probability and Mathematical
  Statistics, John Wiley \& Sons, Ltd., Chichester, 1994. \MR{1274699}

\bibitem{MR4719738}
Christopher~M. Bishop and Hugh Bishop, \emph{Deep learning---foundations and
  concepts}, Springer, Cham, [2024] \copyright 2024. \MR{4719738}

\bibitem{bobkov-goetze}
S.~G. Bobkov and F.~G\"{o}tze, \emph{Exponential integrability and
  transportation cost related to logarithmic {S}obolev inequalities}, J. Funct.
  Anal. \textbf{163} (1999), no.~1, 1--28. \MR{1682772}

\bibitem{MR1846020}
Sergey~G. Bobkov, Ivan Gentil, and Michel Ledoux, \emph{Hypercontractivity of
  {H}amilton-{J}acobi equations}, J. Math. Pures Appl. (9) \textbf{80} (2001),
  no.~7, 669--696. \MR{1846020}

\bibitem{bobkov2026subgradientsconvexfunctionsorlicz}
Sergey~G. Bobkov and Friedrich Götze, \emph{On subgradients of convex
  functions and {O}rlicz pseudo-norms for vector-valued functions}, 2026.

\bibitem{bogachev}
V.~I. Bogachev, \emph{Measure theory. {V}ol. {I}, {II}}, Springer-Verlag,
  Berlin, 2007. \MR{2267655}

\bibitem{bolley-villani}
Fran\c{c}ois Bolley and C\'{e}dric Villani, \emph{Weighted
  {C}sisz\'{a}r-{K}ullback-{P}insker inequalities and applications to
  transportation inequalities}, Ann. Fac. Sci. Toulouse Math. (6) \textbf{14}
  (2005), no.~3, 331--352. \MR{2172583}

\bibitem{boucheron-lugosi-massart-2012}
St\'{e}phane Boucheron, G\'{a}bor Lugosi, and Pascal Massart,
  \emph{Concentration inequalities}, Oxford University Press, Oxford, 2013, A
  nonasymptotic theory of independence, With a foreword by Michel Ledoux.
  \MR{3185193}

\bibitem{MR2257848}
Patrick Cattiaux and Arnaud Guillin, \emph{On quadratic transportation cost
  inequalities}, J. Math. Pures Appl. (9) \textbf{86} (2006), no.~4, 341--361.
  \MR{2257848}

\bibitem{chafai-malrieu}
Djalil Chafa\"{\i} and Florent Malrieu, \emph{On fine properties of mixtures
  with respect to concentration of measure and {S}obolev type inequalities},
  Ann. Inst. Henri Poincar\'{e} Probab. Stat. \textbf{46} (2010), no.~1,
  72--96. \MR{2641771}

\bibitem{MR4499279}
Giorgos Chasapis, Piotr Nayar, and Tomasz Tkocz, \emph{Slicing {$\ell_p$}-balls
  reloaded: stability, planar sections in {$\ell_1$}}, Ann. Probab. \textbf{50}
  (2022), no.~6, 2344--2372. \MR{4499279}

\bibitem{MR2678899}
Sourav Chatterjee, \emph{Spin glasses and {S}tein's method}, Probab. Theory
  Related Fields \textbf{148} (2010), no.~3-4, 567--600. \MR{2678899}

\bibitem{chen-et-al}
Hong-Bin Chen, Sinho Chewi, and Jonathan Niles-Weed, \emph{Dimension-free
  log-{S}obolev inequalities for mixture distributions}, J. Funct. Anal.
  \textbf{281} (2021), no.~11, Paper No. 109236, 17. \MR{4316725}

\bibitem{DBLP:conf/focs/Dasgupta99}
Sanjoy Dasgupta, \emph{Learning {M}ixtures of {G}aussians}, 40th Annual
  Symposium on Foundations of Computer Science, {FOCS} 1999, New York, NY, USA,
  October 17-18, 1999, {IEEE} Computer Society, 1999, pp.~634--644.

\bibitem{MR786142}
P.~Diaconis and D.~Freedman, \emph{Partial exchangeability and sufficiency},
  Statistics: applications and new directions ({C}alcutta, 1981), Indian
  Statist. Inst., Calcutta, 1984, pp.~205--236. \MR{786142}

\bibitem{djellout-guillin-wu}
H.~Djellout, A.~Guillin, and L.~Wu, \emph{Transportation cost-information
  inequalities and applications to random dynamical systems and diffusions},
  Ann. Probab. \textbf{32} (2004), no.~3B, 2702--2732. \MR{2078555}

\bibitem{MR3846841}
Alexandros Eskenazis, Piotr Nayar, and Tomasz Tkocz, \emph{Gaussian mixtures:
  entropy and geometric inequalities}, Ann. Probab. \textbf{46} (2018), no.~5,
  2908--2945. \MR{3846841}

\bibitem{evans}
L.~C. Evans, \emph{Partial differential equations}, second ed., Graduate
  Studies in Mathematics, vol.~19, American Mathematical Society, Providence,
  RI, 2010. \MR{2597943}

\bibitem{gozlan-leonard-survey}
N.~Gozlan and C.~L\'{e}onard, \emph{Transport inequalities. {A} survey}, Markov
  Process. Related Fields \textbf{16} (2010), no.~4, 635--736. \MR{2895086}

\bibitem{gozlan-phd-thesis}
Natha\"{e}l Gozlan, \emph{Principe conditionnel de {G}ibbs pour des contraintes
  fines approch\'{e}es et {I}n\'{e}galit\'{e}s de transport}, Doctoral
  dissertation, Université Paris X -- Nanterre, 2005, Available at
  \url{https://hal.science/tel-00010173v1}.

\bibitem{gozlan-integral-criteria}
Nathael Gozlan, \emph{Integral criteria for transportation-cost inequalities},
  Electron. Comm. Probab. \textbf{11} (2006), 64--77. \MR{2231734}

\bibitem{MR2573565}
\bysame, \emph{A characterization of dimension free concentration in terms of
  transportation inequalities}, Ann. Probab. \textbf{37} (2009), no.~6,
  2480--2498. \MR{2573565}

\bibitem{MR2946156}
\bysame, \emph{Transport-entropy inequalities on the line}, Electron. J.
  Probab. \textbf{17} (2012), no. 49, 18. \MR{2946156}

\bibitem{gozlan-leonard}
Nathael Gozlan and Christian L\'{e}onard, \emph{A large deviation approach to
  some transportation cost inequalities}, Probab. Theory Related Fields
  \textbf{139} (2007), no.~1-2, 235--283. \MR{2322697}

\bibitem{gozlan-roberto-samson-new-characterization}
Nathael Gozlan, Cyril Roberto, and Paul-Marie Samson, \emph{A new
  characterization of {T}alagrand's transport-entropy inequalities and
  applications}, Ann. Probab. \textbf{39} (2011), no.~3, 857--880. \MR{2789577}

\bibitem{gozlan-roberto-samson-hj-metric}
\bysame, \emph{Hamilton {J}acobi equations on metric spaces and transport
  entropy inequalities}, Rev. Mat. Iberoam. \textbf{30} (2014), no.~1,
  133--163. \MR{3186934}

\bibitem{MR3311918}
\bysame, \emph{From dimension free concentration to the {P}oincar\'{e}
  inequality}, Calc. Var. Partial Differential Equations \textbf{52} (2015),
  no.~3-4, 899--925. \MR{3311918}

\bibitem{grsst-characterization}
Nathael Gozlan, Cyril Roberto, Paul-Marie Samson, Yan Shu, and Prasad Tetali,
  \emph{Characterization of a class of weak transport-entropy inequalities on
  the line}, Ann. Inst. Henri Poincar\'{e} Probab. Stat. \textbf{54} (2018),
  no.~3, 1667--1693. \MR{3825894}

\bibitem{grst-Kantorovich-duality}
Nathael Gozlan, Cyril Roberto, Paul-Marie Samson, and Prasad Tetali,
  \emph{Kantorovich duality for general transport costs and applications}, J.
  Funct. Anal. \textbf{273} (2017), no.~11, 3327--3405. \MR{3706606}

\bibitem{mathoverflow-453991}
Piotr~Hajlasz (https://mathoverflow.net/users/121665/piotr hajlasz), \emph{Is
  there a {B}orel measurable $f:\mathbb{R}^d \to \mathbb{R}^d$ such that $f(x)
  \in \partial \varphi (x)$ for all $x$?}, MathOverflow,
  URL:https://mathoverflow.net/q/453991 (version: 2023-12-02, accessed:
  2026-09-22).

\bibitem{MR2141356}
O.~Johnson, \emph{Convergence of the {P}oincar\'{e} constant}, Teor.
  Veroyatnost. i Primenen. \textbf{48} (2003), no.~3, 615--620. \MR{2141356}

\bibitem{MR4231819}
Todd Kemp and David Zimmermann, \emph{Random matrices with log-range
  correlations, and log-{S}obolev inequalities}, Ann. Math. Blaise Pascal
  \textbf{27} (2020), no.~2, 207--232. \MR{4231819}

\bibitem{MR2219345}
I.~Kontoyiannis and M.~Madiman, \emph{Measure concentration for compound
  {P}oisson distributions}, Electron. Comm. Probab. \textbf{11} (2006), 45--57.
  \MR{2219345}

\bibitem{MR1399224}
Michel Ledoux, \emph{On {T}alagrand's deviation inequalities for product
  measures}, ESAIM Probab. Statist. \textbf{1} (1995/97), 63--87. \MR{1399224}

\bibitem{MR1849347}
\bysame, \emph{The concentration of measure phenomenon}, Mathematical Surveys
  and Monographs, vol.~89, American Mathematical Society, Providence, RI, 2001.
  \MR{1849347}

\bibitem{MR2797986}
Yutao Ma, Shi Shen, Xinyu Wang, and Liming Wu, \emph{Transportation
  inequalities: from {P}oisson to {G}ibbs measures}, Bernoulli \textbf{17}
  (2011), no.~1, 155--169. \MR{2797986}

\bibitem{MR838213}
K.~Marton, \emph{A simple proof of the blowing-up lemma}, IEEE Trans. Inform.
  Theory \textbf{32} (1986), no.~3, 445--446. \MR{838213}

\bibitem{MR1404531}
\bysame, \emph{Bounding {$\overline d$}-distance by informational divergence: a
  method to prove measure concentration}, Ann. Probab. \textbf{24} (1996),
  no.~2, 857--866. \MR{1404531}

\bibitem{MR1392329}
\bysame, \emph{A measure concentration inequality for contracting {M}arkov
  chains}, Geom. Funct. Anal. \textbf{6} (1996), no.~3, 556--571. \MR{1392329}

\bibitem{MR1789474}
Geoffrey McLachlan and David Peel, \emph{Finite mixture models}, Wiley Series
  in Probability and Statistics: Applied Probability and Statistics,
  Wiley-Interscience, New York, 2000. \MR{1789474}

\bibitem{nilesweed2026reweightedinformationinequalities}
Jonathan Niles-Weed, \emph{Reweighted information inequalities}, 2026.

\bibitem{MR1760620}
F.~Otto and C.~Villani, \emph{Generalization of an inequality by {T}alagrand
  and links with the logarithmic {S}obolev inequality}, J. Funct. Anal.
  \textbf{173} (2000), no.~2, 361--400. \MR{1760620}

\bibitem{MR1835574}
Robert~R. Phelps, \emph{Lectures on {C}hoquet's theorem}, second ed., Lecture
  Notes in Mathematics, vol. 1757, Springer-Verlag, Berlin, 2001. \MR{1835574}

\bibitem{MR274683}
R.~Tyrrell Rockafellar, \emph{Convex analysis}, Princeton Mathematical Series,
  No. 28, Princeton University Press, Princeton, NJ, 1970. \MR{274683}

\bibitem{MR826433}
O.~S. Rothaus, \emph{Hypercontractivity and the {B}akry-{E}mery criterion for
  compact {L}ie groups}, J. Funct. Anal. \textbf{65} (1986), no.~3, 358--367.
  \MR{826433}

\bibitem{schlichting}
André Schlichting, \emph{Poincaré and log–{S}obolev inequalities for
  mixtures}, Entropy \textbf{21} (2019), no.~1, Article number 89.

\bibitem{shu-strzelecki}
Yan Shu and Micha{\l} Strzelecki, \emph{A characterization of a class of convex
  log-{S}obolev inequalities on the real line}, Ann. Inst. Henri Poincar\'{e}
  Probab. Stat. \textbf{54} (2018), no.~4, 2075--2091. \MR{3865667}

\bibitem{srinivasan2026twoscalecriteriapoincarelogsobolev}
Vishwak Srinivasan, \emph{Two-scale criteria for poincar\'{e} and log-sobolev
  inequalities with applications to markov chain monte carlo}, 2026.

\bibitem{MR1392331}
M.~Talagrand, \emph{Transportation cost for {G}aussian and other product
  measures}, Geom. Funct. Anal. \textbf{6} (1996), no.~3, 587--600.
  \MR{1392331}

\bibitem{MR4628026}
A.~W. van~der Vaart and Jon~A. Wellner, \emph{Weak convergence and empirical
  processes---with applications to statistics}, second ed., Springer Series in
  Statistics, Springer, Cham, [2023] \copyright 2023. \MR{4628026}

\bibitem{wang-wang}
Feng-Yu Wang and Jian Wang, \emph{Functional inequalities for convolution
  probability measures}, Ann. Inst. Henri Poincar\'{e} Probab. Stat.
  \textbf{52} (2016), no.~2, 898--914. \MR{3498015}

\bibitem{zimmermann-2013}
David Zimmermann, \emph{Logarithmic {S}obolev inequalities for mollified
  compactly supported measures}, J. Funct. Anal. \textbf{265} (2013), no.~6,
  1064--1083. \MR{3067796}

\bibitem{zimmermann-2016}
\bysame, \emph{Elementary proof of logarithmic {S}obolev inequalities for
  {G}aussian convolutions on {$\Bbb R$}}, Ann. Math. Blaise Pascal \textbf{23}
  (2016), no.~1, 129--140. \MR{3505572}

\end{thebibliography}

\end{document}